\documentclass{amsart}
  \usepackage[top=1in, bottom=1in, left=1.375in, right=1.375in]{geometry}
 \usepackage{amsmath,amssymb}
 \usepackage{algpseudocode}
 \usepackage{algorithm}
 \usepackage{algorithmicx}
 \usepackage{caption}
 \usepackage{subcaption}
 \usepackage{amsthm}
 \numberwithin{equation}{section}
 \usepackage{amsfonts}
 \usepackage{graphicx}
 \usepackage{hyperref}

\usepackage{amsmath}
\usepackage{amssymb}
\usepackage{amsthm}
\usepackage{amsmath}
\usepackage{setspace}
\usepackage{xcolor}
\usepackage{fancyhdr}
\usepackage{amssymb}
\usepackage{amsthm}
\usepackage{listings}
    \usepackage{cite}
    \usepackage{tabularx}
    \usepackage{graphicx}
    \usepackage{epstopdf}    
    \usepackage{epsfig}
    \usepackage{float}
    \usepackage{ listings}
    \usepackage{appendix}
 \theoremstyle{plain}
  \newtheorem{thm}{Theorem}
 \newtheorem{prop}{Proposition}[section]
 \newtheorem{lem}[prop]{Lemma}
 
 \theoremstyle{definition}
 \newtheorem{definition}[prop]{Definition}
 
 \theoremstyle{remark}
 \newtheorem{remark}[prop]{Remark}

 \let\pa=\partial
 \let\al=\alpha
 \let\b=\beta
 \let\d=\delta
 \let\g=\gamma
 \let\e=\varepsilon

 \let \kp = \kappa
 \let\lam=\lambda

 \let\f=\frac
 \let\inf = \infty
 \let \les = \lesssim
 \let \gtr = \gtrsim
 \let\om=\omega
 
 \let \th = \theta

\let\B = \Big
 \let\D=\Delta

 \let\td = \tilde
 
 \let\wh=\widehat

 \let\teq \triangleq
 
 \let\pa=\partial
 \let \bsh = \backslash

 \def\cH{{\mathcal H}}

 \def\cL{{\mathcal L}}

 \def\cR{{\mathcal R}}

 \def\na{\nabla}
 \def\la{\langle}
 \def\ra{\rangle}
\def\lt{\left}
\def\rt{\right}
\def\one{\mathbf{1}}

 \newcommand{\beq}{\begin{equation}}
 \newcommand{\eeq}{\end{equation}}
  \newcommand{\bal}{\begin{aligned} }
  \newcommand{\eal}{\end{aligned}}
 \newcommand{\ben}{\begin{eqnarray}}
 \newcommand{\een}{\end{eqnarray}}
 \newcommand{\beno}{\begin{eqnarray*}}
 \newcommand{\eeno}{\end{eqnarray*}}

 \newcommand{\uu}{\mathbf{u}}

 \newcommand{\R}{\mathbb{R}}

  \newcommand{\bZ}{\mathbb{Z}}

\newcommand{\sgn}{\mathrm{sgn}}

 \author{Jiajie Chen }
 \address{Applied and Computational Mathematics, California Institute of Technology, Pasadena, CA 91125, USA. Email: jchen@caltech.edu}
 \date{\today}

\title[ Regularity of the De Gregorio Model]{
On the Regularity of the De Gregorio Model for the 3D Euler equations}
 
\begin{document}

\begin{abstract}

% We study the regularity of the De Gregorio (DG) model $\omega_t + u\omega_x = u_x \omega$ on $S^1$ for odd initial data $\omega_0$ with $\pi$ period and $\omega_0 \leq 0 $ on $[0,\pi/2]$. These sign and symmetry properties are the same as those of the smooth initial data that leads to singularity formation of the De Gregorio model on $\R$ or the generalized Constantin-Lax-Majda (gCLM) model on $\R$ or $S^1$ with positive parameter. Thus, to establish global regularity of the DG model with general smooth data, an important step is to rule out finite time blowup with initial data in the above class. We accomplish this by establishing a one-point blowup criterion and proving global well-posedness for $C^{1,\alpha}$ initial data with any $\alpha \in (0,1)$. These results verify the regularity conjecture of the DG model for smooth initial data in the above class. On the other hand, for any $ \alpha \in (0,1)$, we construct finite time blowup solution from a sub-class of initial data with $\omega_0 \in C^{\alpha} \cap C^{\infty}(S^1 \backslash \{0\})$. Our results imply that singularities developed in the DG model on $S^1$ and the gCLM model can be regularized by stronger advection. 

We study the regularity of the De Gregorio (DG) model $\omega_t + u\omega_x = u_x \omega$ on $S^1$ for initial data $\omega_0$ with period  $\pi$ and in class $X$: $\omega_0$ is odd and $\omega_0 \leq 0 $ (or $\omega_0 \geq 0$) on $[0,\pi/2]$. These sign and symmetry properties are the same as those of the smooth initial data that lead to singularity formation of the De Gregorio model on $\mathbb{R}$ or the generalized Constantin-Lax-Majda (gCLM) model on $\mathbb{R}$ or $S^1$ with a positive parameter. Thus, to establish global regularity of the DG model for general smooth initial data, which is a conjecture on the DG model, an important step is to rule out potential finite time blowup from smooth initial data in $X$. We accomplish this by establishing a one-point blowup criterion and proving global well-posedness for initial data  $ \om_0 \in H^1 \cap X$ with $\om_0(x) x^{-1} \in L^{\inf}$.
%$C^{1,\alpha}$ initial data with any $\alpha \in (0,1)$. 
On the other hand, for any $ \alpha \in (0,1)$, we construct a finite time blowup solution from a class of initial data with $\omega_0 \in C^{\alpha} \cap C^{\infty}(S^1 \backslash \{0\}) \cap X$. Our results imply that singularities developed in the DG model and the gCLM model on $S^1$  can be prevented by stronger advection.

\end{abstract}

 \maketitle

\section{Introduction}

% The three-dimensional (3D) incompressible Euler equations in the vorticity formulation read
% \beq\label{eq:Euler}
%   \omega_{t} + u \cdot \nabla \omega = \omega \cdot \nabla u,  \quad u = \na \times (-\D)^{-1} \om
% \eeq
% where $u$ is the velocity field and $\omega = \nabla \times u$ is the vorticity vector of the fluid. The problem of finite time blowup vs global regularity of the 3D Euler equations from smooth initial data remains one of the major open questions in nonlinear partial differential equations. One of the main difficulties is the vortex stretching term $ \omega \cdot \nabla u$, which is nonlinear and nonlocal. 

To model the effect of the vortex stretching in the three-dimensional(3D) incompressible Euler equations, Constantin, Lax, and Majda \cite{CLM85} proposed a one-dimensional model (CLM)
%To model the effect of the vortex stretching, Constantin, Lax and Majda \cite{CLM85} proposed an one-dimensional model (CLM)
\beq\label{eq:CLM}
\om_t = u_x\om , \quad u_x = H \om,
\eeq
where $H$ is the Hilbert transform. %{\color{blue}and established the singularity formation of \eqref{eq:CLM}.}
Singularity formation of \eqref{eq:CLM} was established and was studied in detail in \cite{CLM85}. The effect of the advection in the 3D Euler equations is not modeled in \eqref{eq:CLM}. 

De Gregorio \cite{DG90,DG96} considered both effects by adding an advection term $u \om_x$ to \eqref{eq:CLM}
%De Gregorio \cite{DG90,DG96} considered the effects of advection and vortex stretching by adding $u \om_x$ to \eqref{eq:CLM}
%both effects by adding an advection term $u \om_x$ to \eqref{eq:CLM}
\beq\label{eq:DG}
\bal
\om_t +  u \om_x& = u_x \om , \quad u_x=  H  \om,
\eal
\eeq
and provided some evidence that \eqref{eq:DG} admits no blowup. To understand the effect of the advection in \eqref{eq:DG}, we can neglect the vortex stretching 
$u_x \om$ in \eqref{eq:DG}. The resulting model can also be seen as \eqref{eq:gCLM} with infinite weight $a=\infty$ in the advection. One can obtain the global well-posedness of this model 
using the conservation of $|| \om||_{L^{\inf}}$, see, e.g. \cite{OSW08}.
Numerical simulations performed in \cite{OSW08,lushnikov2020collapse} and the report in \cite{Sve19} suggest that a solution of \eqref{eq:DG} from smooth initial data exists globally. 
%This leads to the regularity conjecture of the De Gregorio (DG) model 
These lead to the conjecture that the De Gregorio (DG) model is globally well-posed for smooth initial data, which was made in \cite{OSW08,lushnikov2020collapse,Elg17}.
%and a natural question of whether the advection in \eqref{eq:DG} can prevent the singularity developed in the CLM model \eqref{eq:CLM}.
Note that the question of regularity for the DG model is listed as one of the open problems in \cite{grafakos2017some}. In contrast to the CLM model, there is a strong competition in the DG model between the nonlocal stabilizing effect due to the advection and a destabilizing effect due to the vortex stretching. These two effects are comparable, making it very challenging to analyze \eqref{eq:DG}. 
%which makes it very challenging to analyze \eqref{eq:DG}. 
We remark that the stabilizing effect of the advection has been studied in \cite{lei2009stabilizing,hou2008dynamic} for the 3D Navier Stokes equations.
%and Euler equations.

Regarding the global regularity of \eqref{eq:DG}, the first result seems to be established only recently by Jia-Stewart-Sverak \cite{Sve19}, who proved the nonlinear stability of a steady state $A \sin (2x)$ of \eqref{eq:DG} with period $\pi$ using spectral theories. %Since the CLM model with initial data close to $ A\sin (2x)$ blows up in finite time, as a result of \cite{Sve19}, the advection in \eqref{eq:DG} prevents singularity formation for such initial data. 
In \cite{lei2019constantin}, Lei-Liu-Ren discovered a novel equation (see \eqref{eq:imp0}) and a conserved quantity for initial data $\om_0$ with a fixed sign and established the global regularity of \eqref{eq:DG} for such initial data.
%We note that for initial data $\om_0$ that is strictly positive or negative,
We note that for strictly positive or negative initial data $\om_0$, the CLM model \eqref{eq:CLM} does not blow up. 
On the other hand, in recent joint work with Hou and Huang \cite{chen2019finite}, we established finite time blowup of \eqref{eq:DG} on $\R$ with initial data $\om_0 \in C_c^{\inf}$
%odd initial data $\om_0 \in C_c^{\inf}$ and $\om_0 \leq 0$ on $\R_+$ 
by proving the nonlinear stability of an approximate blowup profile. %This result implies that the 
Thus the above conjecture on the regularity of the DG model is not valid for all smooth initial data in the case of $\R$.  
%This result disproves the regularity conjecture of \eqref{eq:DG} in the case of $\R$. 

%For initial data $\om_0$ with a fixed sign, the global regularity of \eqref{eq:DG} was obtained by Lei-Liu-Ren in \cite{lei2019constantin} 
%\eqref{eq:DG} 

In this paper, we study the regularity of the De Gregorio model \eqref{eq:DG} on $S^1$ with period $\pi$. We focus on odd initial data $\om_0$ in class $X$ (see \eqref{eq:X}): $\om_0(x) \geq  0$ or $\om_0(x) \leq 0$ for all $x \in [0,  \f{\pi}{2}]$. These properties are preserved dynamically. 
The class of initial data in $X$ seems to provide the most promising scenario for a potential blowup solution of \eqref{eq:DG} on $S^1$ up to now for the following reasons. 
%The scenario with initial data in $X$ seems to be the most promising potential blowup scenario for \eqref{eq:DG} on $S^1$ up to now for the following reasons. 
Firstly, the initial data considered in \cite{chen2019finite} that lead to finite time blowup of \eqref{eq:DG} on $\R$ satisfies the same sign and symmetry properties as those in $X$. Secondly, for the generalized Constantin-Lax-Majda (gCLM) model \cite{OSW08} 
\beq\label{eq:gCLM}
\om_t + a u \om_x = u_x \om, \quad u_x = H \om
\eeq
%with positive parameter $a$, which replaces $u\om_x$ in \eqref{eq:DG} by $a u\om_x$ with $a>0$, 
with $a>0$, which is closely related to \eqref{eq:DG}, singularity formation \cite{Elg17,chen2020singularity,chen2019finite,chen2020slightly,Elg19} all develops from initial data with the same sign and symmetry properties as those in $X$. %In particular, in \cite{chen2020slightly}, we established singularity formation of \eqref{eq:gCLM} on $S^1$ with $a$ slightly less than $1$, which can be seen as a slight perturbation to \eqref{eq:DG}, with smooth initial data in $X$.
In particular, in \cite{chen2020slightly}, we established that the gCLM model on $S^1$ with $a$ slightly less than $1$, which can be seen as a slight perturbation to \eqref{eq:DG}, develops finite time singularity from some smooth initial data in $X$.
Thirdly, this scenario can be seen as a 1D analog of the hyperbolic blowup scenario for the 3D Euler equations reported by Hou-Luo \cite{luo2013potentially-2,luo2014potentially}. See also \cite{kiselev2013small,kiselev2015finite,chen2019finite2}. 
In fact, the restriction of the (angular) vorticity in \cite{luo2013potentially-2,luo2014potentially,kiselev2013small,kiselev2015finite,chen2019finite2} to the boundary has the same sign and symmetry properties as those in $X$. Thus, to establish global regularity of \eqref{eq:DG} for general smooth initial data, we need to address the important question of whether there is a finite time blowup in this class. We note that the initial data considered in \cite{Sve19} is close to the steady state $A \sin(2x)$ of \eqref{eq:DG}. Thus it belongs to or is close to that in $X$.  
%and initial data considered in \cite{Sve19} are belong to $X$ or close to that in $X$.

%The analysis of \eqref{eq:DG} with initial data in $X$ can also provide useful insight on the analysis of \eqref{eq:DG} with general initial data $\om_0$ near a potential singularity. 
%Note that the CLM model blows up at the point in $S = \{ x : \om(x) = 0, \ H\om(x) >0 \}$ \cite{CLM85}.
%at $x$ in finite time if and only if $\om(x)=0$ and $H\om(x) > 0$ \cite{CLM85}.
%See also Section \ref{sec:CLM}. 

Note that the CLM model can only blow up in finite time at the zeros of $\om$ \cite{CLM85}. Since the vortex stretching is the driving force for a potential blowup of \eqref{eq:DG}, it is likely that a potential singularity of \eqref{eq:DG} with general data is also located at the zeros of $\om$. For a zero $x_0$ of $\om$, across which $\omega$ changes sign, the leading order term of $\om$ near $x_0$ is $\pa_x^k \om(x_0) (x-x_0)^k$ for some odd $k \in Z_+$. It has the same sign and symmetry properties as those in $X$. Thus, our analysis of \eqref{eq:DG} with  $\om \in X$ can provide valuable insights on the local analysis of these potential singularities. 
For a zero $x_0$ of $\om$, across which $\om$ does not change sign, the local analysis could benefit from \cite{lei2019constantin}.
%estimate in \cite{lei2019constantin} provides useful analysis. 

%the point in $S$. For each point in $S$, 

%Moreover, the analysis of \eqref{eq:DG} with initial data in $X$ can 

%The data considered in \cite{Sve19} is closely related to the above class of initial data since $\sin(2x )$ is in such class. 

There are other 1D models for the 3D Euler equations and SQG equation, see, e.g. \cite{cordoba2005formation,choi2014on}. We refer to \cite{Elg17,choi2014on} for excellent surveys and \cite{Elg17,choi2014on,chen2021HL} for discussions on the connections.

\subsection{Main results}

Throughout this paper, we consider initial data $\om_0$ in the following class $X$
\beq\label{eq:X}
X \teq \B\{ f: f \mathrm{ \ is \ odd \ }, \pi-\mathrm{periodic \ and \ }   f(x) \leq 0 , x \in [0, \f{\pi}{2} ] \B\},
\eeq
unless we specify otherwise. We assume $\om_0 \leq 0$ on $[0, \f{\pi}{2}]$ without loss of generality. For the case of $\om_0 \geq 0$ on $ [0, \f{\pi}{2} ]$, we can consider a new variable $\om_{new}(x) \teq \om(x + \f{\pi}{2})$ and then reduce it to the previous case. 
It is not difficult to show that the solution  $\om(t)$ remains in $X$.

Our first main result is a one-point blowup criterion. A similar blowup criterion has been obtained in our previous work \cite{chen2020singularity} for the DG model and the gCLM model with dissipation.

\begin{thm}\label{thm:criterion}
Suppose that $\om_0 \in X \cap  H^1 $ and $\int_{0}^{\pi/2} \B| \f{\om_{0,x}^2}{\om_0} \sin(2x)  \B| dx <+\infty$. The unique local in time solution of \eqref{eq:DG} cannot be extended beyond $T > 0$ if and only if
\beq\label{eq:onept}
\int_0^{T} u_x(0, t) d t = \infty.
\eeq

\end{thm}

For $\om \in X \cap H^1$, we have $u_x(0, t) \geq 0$.
%Under the assumption on the initial data, $u_x(t, 0)$ is nonnegative. 
Suppose that $\om$ vanishes to the order $|x|^{\b}, \b >0$ near $x=0$. Then $\f{\om_x^2}{\om} \sin(2x)$ is of order $| x|^{2 (\b - 1) - \b + 1} = |x|^{\b - 1}$ near $x=0$, which is locally integrable. A similar conclusion holds for the local integrability near $ x = \f{\pi}{2}$. 
For $\om \in C^{1,\al} \cap X$, the sign condition in $X$ implies that $\om$ degenerates at its zeros in $S^1 \backslash \{0, \pi/2 \}$ with an order $\b > 1$, if it exists, and thus $\f{\om_x^2}{\om} \sin(2x)$ is still locally integrable. 
% For smooth initial data in $X$, e.g. $\om_0 \in C^{\inf} $ with $\om_0 <0$ on $(0,\pi/2)$, $\pa_x^k \om_0(0) \neq 0, \pa_x^l \om_0(\pi/2) \neq 0$ for some $k, l \in \bZ_+$, the assumption $\int_{0}^{\pi/2} \B| \f{\om_{0,x}^2}{\om_0} \sin(2x)  \B| dx < +\inf$ in Theorem \ref{thm:criterion} holds automatically. 
In particular, for $\om_0 \in C^{\inf} \cap X$ with a finite number of zeros and a finite order of degeneracy, the assumption 
$\int_{0}^{\pi/2} \B| \f{\om_{0,x}^2}{\om_0} \sin(2x)  \B| dx < +\inf$ holds automatically. 
%For smooth initial data in $X$, e.g. $\om_0 \in C^{\inf} $ with $\om_0 <0$ on $(0,\pi/2)$, $\pa_x^k \om_0(0) \neq 0, \pa_x^l \om_0(\pi/2) \neq 0$ for some $k, l \in \bZ_+$, the assumption $\int_{0}^{\pi/2} \B| \f{\om_{0,x}^2}{\om_0} \sin(2x)  \B| dx < +\inf$ in Theorem \ref{thm:criterion} holds automatically. 
%Based on the above blowup criterion, 
Based on Theorem \ref{thm:criterion}, we obtain the following global well-posedness result.
% of \eqref{eq:DG}. 
% Based on the above result, we obtain global regularity of \eqref{eq:DG} with initial data in the class $X$ that
% %that  smooth initial data in the class $X$. 
% vanishes at least linearly near $x=0$.

\begin{thm}\label{thm:GWP}
Suppose that $\om_0 \in X \cap H^1$, $ \om_0(x) x^{-1} \in L^{\inf}$, and $A(\om_0) = \int_{0}^{\pi/2} \B| \f{\om_{0,x}^2}{\om_0} \sin(2x)  \B| dx < +\inf$. There exists a global solution $\om$ of \eqref{eq:DG} with initial data $\om_0$. In particular, (a) for $\om_0 \in X \cap C^{1,\al}$ with $ \al \in (0, 1)$ and $A(\om_0) < +\inf$, there exists a global solution from $\om_0$; 
%there exists a global solution from $\om_0 \in X \cap C^{1,\al}$ with some $ \al \in (0, 1)$ and $A(\om_0) < +\inf$;
(b) for $\om_0 \in X \cap C^1 $ with $A(\om_0) < +\inf$, the unique local solution $\om \in \cap_{\al < 1} C^{\al} $ from $\om_0$ exists globally. If the initial data  further satisfies $\om_0 \in C^{1,\al}$ with $\al \in (0, 1)$ and $\om_{0,x}(0) =0 $, we have
\[
 || \om(t)||_{L^1} + | u_x(0,t) | \leq K(\om_0) e^{ C Q(2) t},  \quad ||\om(t)||_{L^{\inf}} \leq K(\om_0) \exp( 2 \exp( K(\om_0) \exp( C Q(2) t)   )),
 \]
where $Q(2) = \int_0^{\pi/2} |\om_0| \cot^2 y dy$ and $K(\om_0)$ is some constant depending on $ H \om_0(0) , H\om_0( \f{\pi}{2})$, $ || \om_0||_{L^1}, Q(2), A(\om_0)$.
\end{thm}

% \begin{thm}\label{thm:GWP}
% Suppose that $\om_0 \in X \cap C^{1,\al}$ for some $\al \in (0, 1)$  and $A(\om_0) = \int_{0}^{\pi/2} \B| \f{\om_{0,x}^2}{\om_0} \sin(2x)  \B| dx < +\inf$. There exists a global solution $\om$ of \eqref{eq:DG} with initial data $\om_0$. If the initial data further satisfies $\om_{0,x}(0) =0 $, we have
% \[
%  || \om(t)||_{L^1} + | u_x(0,t) | \leq K(\om_0) e^{ C Q(2) t},  \quad ||\om(t)||_{L^{\inf}} \leq K(\om_0) \exp( 2 \exp( K(\om_0) \exp( C Q(2) t)   )),
%  \]
% where $Q(2) = \int_0^{\pi/2} |\om_0| \cot^2 y dy$ and $K(\om_0)$ is some constant depending on $ H \om_0(0) , H\om_0( \f{\pi}{2})$, $ || \om_0||_{L^1}, Q(2), A(\om_0)$.
% \end{thm}

% {\color{blue}In the general case}, the a-priori estimates are much weaker. See Lemma \ref{lem:boost} and Remark \ref{rem:boost} for more discussions. Note that $C^k(S^1) \subset C^{1,\al}(S^1)$ for $k \in \bZ, k \geq 2$. Since $H^s \hookrightarrow C^{1,\al}$ for $s > \al + \f{3}{2}$, Theorem \ref{thm:GWP} implies the global well-posedness (GWP) in $H^s \cap X$ with $s> \f{3}{2}$. 
% {\color{blue}
% The condition that $\om_0$ vanishes at least linearly near $x=0$ in Theorem \ref{thm:GWP}, i.e. $\om_0(x)x^{-1} \in L^{\inf}$, is necessary since we can obtain a finite time blowup for $\om_0$ that is less regular near $x=0$.
% }

In the general case, the a-priori estimates are much weaker. See Lemma \ref{lem:boost} and Remark \ref{rem:boost} for more discussions. %Note that $C^k(S^1) \subset C^{1,\al}(S^1)$ for $k \in \bZ, k \geq 2$. 
Since $H^s \hookrightarrow C^{1,\al}$ for $s > \al + \f{3}{2}$, Theorem \ref{thm:GWP} implies the global well-posedness (GWP) in $H^s \cap X$ with $s> \f{3}{2}$. The condition $\om_0(x)x^{-1} \in L^{\inf}$ in Theorem \ref{thm:GWP} is necessary since we can obtain a finite time blowup for $\om_0$ that is less regular near $x=0$.

%On the other hand, %for initial data that is less regular near $x=0$, we obtain a finite time blowup. 
%vanishes near $x=0$ at order less than linear, we obtain finite time blowup. 
%{\color{blue}On the other hand, for initial data that vanishes near $x=0$ at order $|x|^{\g}$ with $\g < 1$, we obtain finite time blowup.}

\begin{thm}\label{thm:blowup}
For any $0 < \al <1, s < \f{3}{2}$, there exists $\om_0 \in X \cap C^{\al} \cap H^s \cap C^{\inf}(S^1 \bsh \{0\} )$ with $\int_{0}^{\pi/2} \B| \f{\om_{0,x}^2}{\om_0} \sin(2x)  \B| dx <+\infty$,
%odd $\om_0$ with $\om_0 < 0$ on $( 0, \f{\pi}{2})$ and $\om_0 \in C^{\al}$, 
such that the solution of \eqref{eq:DG} with initial data $\om_0$ develops a singularity in finite time. In particular, we have $\int_0^T u_x(0, t) dt = \inf$.
\end{thm}

%Near $x=0$, the initial data in Theorem \ref{thm:blowup} satisfies $\om_0 \approx -C |x|^{\al}$.  
%In view of

One can establish the local well-posedness of \eqref{eq:DG} in $C^{k, \al}$ with any $k \in \bZ_+ \cup \{ 0\}$
%integer $k \geq 0$ 
and $\al \in (0, 1)$ using the particle trajectory method \cite{majda2002vorticity}. From the ill-posedness result for the incompressible Euler equations in \cite{bourgain2015strong}, it is conceivable that \eqref{eq:DG} is ill-posed in $C^1$. 
%$\om_0 \in C^1$. 
For $C^1$ initial data, there is a unique local solution in $\cap_{\al < 1} C^{\al}$.
Thus, in view of the above Theorems, in the class $\om \in X$, the blowup criterion in Theorem \ref{thm:criterion} and the regularity results in Theorems \ref{thm:GWP} and \ref{thm:blowup} are sharp. 

Theorem \ref{thm:GWP} verifies the conjecture on the GWP of \eqref{eq:DG} on $S^1$ and rules out potential blowup of \eqref{eq:DG} from initial data in $C^{\inf} \cap X$. It also addresses the conjecture made in \cite{Elg17} in the case of $S^1$ that the strong solution to \eqref{eq:DG} is global for $C^1$ initial data in class $X$.
Note that the smooth initial data that lead to singularity formation of the gCLM model \eqref{eq:gCLM} on $S^1$ \cite{chen2019finite,chen2020singularity,chen2020slightly} or the CLM model \cite{CLM85} can be chosen in the class in Theorem \ref{thm:GWP}. Thus, Theorem \ref{thm:GWP} implies that the advection in \eqref{eq:DG} can prevent singularity formation in the CLM model or the gCLM model for such initial data. The global regularity results in Theorem \ref{thm:GWP} can be generalized to the DG model \eqref{eq:DG} with an external force $f \om$ linear in $\om$, where  $f \in C^{\inf}$ is a given even function. Theorem \ref{thm:blowup} resolves the conjecture made in \cite{Elg17,sverak2017certain} 
that \eqref{eq:DG} develops a finite time singularity from initial data $\om_0 \in C^{\al}$ or $\om_0 \in  H^s$ for any $\al \in (0,1) $ and $s < \f{3}{2}$ in the case of $S^1$. The case of $\R$ has been resolved in \cite{chen2019finite} with $\om_0 \in C_c^{\inf}$.

% Nov 8 old remark 
% \begin{remark}\label{remark:order}

% As we will see later, for $\om_0 \in X$, the vanishing order of $\om_0$ near $x=0$ characterizes the 
% relative strength between the advection and vortex stretching and determines the regularity of the solution. 
% %term and the vortex stretching term in \eqref{eq:DG} and determines the regularity of solution. 
% As the vanishing order increases, the effect of advection increases.
% %  The advection $u \om_x$ in \eqref{eq:DG} is stronger than, comparable to, and weaker than the 
% % vortex stretching $u_x \om$, if near $x =0$, $\om$ vanishes to the order $|x|^a$ with $a<1$, $a=1$ and $a>1$, respectively.  See Section \ref{sec:intro_comp} for heuristic discussions on these relations. The case $a=1$ relates to the global well-posedness (GWP) result in Theorem \ref{thm:GWP} with $\om_{0,x}(0) \neq 0$.  
% %In the case of $\om \in C^{1,\al}$ with $\om_{0,x} \neq 0$ in Theorem \ref{thm:GWP}, 
% {\color{blue}In the case of $\om$ vanishing linearly near $x=0$ in Theorem \ref{thm:GWP}, e.g. 
% $\om \in C^{1,\al}$ with $\om_{0,x}(0) \neq 0$, }
% %$\om \in C^{1,\al}$ with $\om_{0,x} \neq 0$ in Theorem \ref{thm:GWP}, 
% the effects of two terms balance, which makes it very challenging to establish the GWP result in Theorem \ref{thm:GWP}.

% \end{remark}

In \cite{Elg17}, Elgindi-Jeong made an important observation that the advection can be substantially weakened by choosing $C^{\al}$ data with sufficiently small $\al$, and constructed $C^{\al}$ self-similar blowup solution of \eqref{eq:DG} on $\R$ with small $\al$.
For \eqref{eq:DG} on $S^1$, a finite time blowup from $C_c^{\al}$ data with small $\al$ was obtained in \cite{chen2019finite}. In Theorem \ref{thm:blowup}, the H\"older exponent $\al$ can be arbitrary close to $1$. As we will see in the proof, it suffices to weaken the advection slightly. Theorem  \ref{thm:blowup} is inspired by our previous work \cite{chen2020slightly}, where we constructed a finite time blowup solution for the gCLM model \eqref{eq:gCLM} with $a$ slightly less than $1$ and smooth initial data.
%in the case where the advection is only slightly weaker than the vortex stretching. 

% {\color{blue}
% As we will see later, it suffices to weaken the advection slightly. Theorem  \ref{thm:blowup} is inspired by our previous work \cite{chen2020slightly} on singularity formation of the gCLM model
%  \eqref{eq:gCLM} with $C^{\inf}$ initial data and $a$ slightly less than $1$.
% }

%{\color{blue}%on singularity formation of the gCLM model \eqref{eq:gCLM} with $a$ slightly less than $1$}.

%In \cite{Elg17}, Elgindi-Jeong constructed the $C^{\al}$ self-similar blowup solution of \eqref{eq:DG} with sufficiently small $\al$ on the real line. In the joint work with Hou and Huang \cite{chen2019finite}, we established asymptotically self-similar blowup of \eqref{eq:DG} from $C_c^{\al}$ data with sufficiently small $\al$ both on the real line and the circle. See also \cite{Elg19}. In these results, the advection is substantially weakened by choosing $C^{\al}$ data with sufficient small $\al$. This observation was first made in \cite{Elg17}. In Theorem \ref{thm:blowup}, the H\"older exponent $\al$ can be arbitrary close to $1$. As we will see in the proof, it suffices to weaken the advection slightly. Theorem  \ref{thm:blowup} is inspired by and an anologue of the result in our previous work \cite{chen2020slightly}, where we constructed finite time blowup solution for the gCLM model in the case where the advection is only slightly weaker than the vortex stretching. 

\subsection{Connection with the CLM model}\label{sec:CLM}

%Explicit formula of the solution to \eqref{eq:CLM} was obtained in \cite{CLM85}
The CLM model \eqref{eq:CLM} can be solved explicitly \cite{CLM85}
%In \cite{CLM85},  Constantin-Lax-Majda obtained the explicit solution of \eqref{eq:CLM} 
%In \cite{CLM85},  Constantin-Lax-Majda obtained the explicit solution of \eqref{eq:CLM} 
\beq\label{eq:CLM_sol}
\om(x, t) = \f{ 4 \om_0(x)}{ (2 - t H\om_0(x))^2 + t^2 \om_0^2(x) }, 
\quad H \om(x, t) = \f{2 H\om_0(x) (2 - t H \om_0(x)) - 2 t \om^2_0(x) }{ (2 - t H\om_0(x))^2 + t^2 \om_0^2(x) } .
\eeq

%From the above formula, 
We consider the solution of \eqref{eq:CLM} with period $\pi$ . From \eqref{eq:CLM_sol}, the solution can blow up at $x$ in finite time if and only if $\om_0(x) =0 $ and $H \om_0(x) > 0$.
Consider odd $\om_0$ with $\om_0 < 0$ on $( 0, \f{\pi}{2})$. 
%For $\om_0 \in X$, 
Since $ H\om_0( 0 ) > 0$ and $H \om_0(\f{\pi}{2}) < 0$, the only point $x$ with $\om_0(x) =0 $ and $H \om_0(x) > 0$ is $x =0$. 
Within this class of initial data, from Theorem \ref{thm:criterion}, $u_x(0, t)$ controls the 
blowup in both the CLM model and the De Gregorio model. On the other hand, the CLM model blows up in finite time for smooth initial data , while from Theorems \ref{thm:GWP}, \ref{thm:blowup}, the advection term in the De Gregorio model can prevent singularity formation if the initial data is smooth enough.

%There are  connections among our results, the DG model and incompressible fluids. We will discuss them in Section \ref{sec:incomp}.

\subsection{Competition between advection and vortex stretching}\label{sec:intro_comp}

%The relative strength of the advection and vortex stretching and 
The competition between advection and vortex stretching and its relation with the vanishing order of $\om \in X$ near $x=0$
%regularity of $\om \in X$ 
can be illustrated by a simple Taylor expansion. Suppose that near $x=0$, $\om = - x^{a} + l.o.t.$  for $a>0$ and $u = c x + l.o.t.$ for some $c>0$, where $l.o.t.$ denotes the lower order terms. We impose the latter assumption on $u$ since $u = -(-\pa_{xx})^{-1/2} \om$ is odd and at least $C^1$ with $u_x(0) > 0$ for nontrivial $\om \in X$. The leading order term of $u \om_x$ and $u_x\om$ near $x=0$ are given by 
\[
u \om_x = - ac x^{a} + l.o.t., \quad u_x \om = - c x^a + l.o.t.
\]

This simple calculation suggests that $a-1$ characterizes the relative strength between the advection $|u\om_x|$ and the vortex stretching $|u_x\om|$ near $x=0$. The advection is weaker than, comparable to, and stronger than the vortex stretching if $a<1$, $a=1$, and $a>1$, respectively. 
% This simple calcultion suggests that $a-1$ characterizes the relative strength between the advection $|u\om_x|$ and the vortex stretching $|u_x\om|$ near $x=0$. In particular, the advection is weaker if $a < 1$, and it is comparable to and stronger than the vortex stretching if $a= 1$ and $a>1$, respectively. 
Considering the stabilizing effect of advection \cite{lei2009stabilizing,chen2020slightly,OSW08} and the destabilizing effect of vortex stretching \cite{CLM85}, one would expect that
% Since the advection can have stablizing effect that regularizes the solution \cite{lei2009stabilizing,chen2020slightly,OSW08} and the vortex stretching can lead to fast growth of the solution, for $\om \in X$, in three different cases, one would exoect that
there exists singularity formation in the case of $a<1$ and global well-posedness in the case of $a \geq 1$. Theorems \ref{thm:GWP} and \ref{thm:blowup} confirm this formal analysis. In the case of $a=1$, e.g. $\om_0 \in C^{1,\al}$ with $\om_{0,x}(0) \neq 0$ in Theorem \ref{thm:GWP}, the effects of two terms balance, making it very challenging to establish the GWP result in Theorem \ref{thm:GWP}.
%From the above heuristics and Theorems \ref{thm:GWP}, \ref{thm:blowup}, the case $a<1, a=1$ and $a>1$ can be interpreted as the \textit{supercritical, critical}, and \textit{subcritical} case, respectively. 
To prove these results, we need to quantitatively characterize the competition in three different cases and precisely control the effects of advection and vortex stretching. See more discussions in Section \ref{sec:idea}.

\subsection{Connections with incompressible fluids}\label{sec:incomp}
%We discuss the connections between our results and 
\subsubsection{The effect of advection}

Theorem \ref{thm:GWP} provides some valuable insights on potential singularity formation in incompressible fluids. We consider the 2D Boussinesq equations 
\beq\label{eq:Bous}
\om_t + \uu \cdot \na \om  = \th_x,  \quad \th_t + \uu \cdot \na \th = 0,
\eeq
where $\om$ is the vorticity, $\th$ is the density, and $\uu$ is the velocity field determined by $ \na^{\perp}(-\D)^{-1}\om$.

In the whole space, a promising potential blowup scenario is the hyperbolic-flow scenario with $\th_x, \om$ being odd in both $x, y$, and positive $\th_x, \om$ in the first quadrant. 
The induced flow is clockwise in the first quadrant near the origin. A similar scenario has been used in \cite{zlatovs2015exponential,he2021small}. In this scenario, the flow in the $y$-direction in the first quadrant moves away from the origin. To understand the effect of $y-$advection, we derive a model on $\th_x$, which is the driving force for the growth in \eqref{eq:Bous}. Taking $x-$derivative on \eqref{eq:Bous} and using the incompressible condition $u_{2, y} = -u_{1,x}$, we yield 
\beq\label{eq:bous20}
\pa_t \th_x + \uu \cdot \na  \th_{x} = - u_{1,x} \th_x - u_{2,x}\th_y =
u_{2,y}\th_x -  u_{2,x}\th_y.
\eeq

Dropping $\th_y$ term and the advection in $x$ direction and simplifying $\om = \th_x$, we further derive
\begin{align}
&\pa_t \th_x + u_2  \pa_y \th_x = u_{2, y} \th_x , \label{eq:Bous2} \\
&\uu = \na^{\perp} (-\D)^{-1} \th_x, \quad u_{2, y} = \pa_{xy}(-\D)^{-1} \th_x . \label{eq:Bous3} 
\end{align}

See more  motivations for these simplifications in Appendix \ref{app:bous}. 
% we derive the following model on $\th_x$
% from \eqref{eq:Bous}, which is the driving force for the growth in \eqref{eq:Bous}%
% \begin{align}
% &\pa_t \th_x + u_2  \pa_y \th_x = u_{2, y} \th_x , \label{eq:Bous2} \\
% &\uu = \na^{\perp} (-\D)^{-1} \th_x, \quad u_{2, y} = \pa_{xy}(-\D)^{-1} \th_x , \label{eq:Bous3} 
% \end{align}
% In the derivation, we have dropped $\th_y$, the advection in $x$ direction, applied the incompressible condition $u_{2,y} = u_{1,x}$ and simplified $\om = \th_x$. We refer to Appendix \ref{app:bous} for the derivation. 
% Note that the $\th-$equation in \eqref{eq:Bous} with \eqref{eq:Bous3} reduces to the incompressible porous media equation \cite{cordoba2007analytical,cordoba2011lack}. Equation \eqref{eq:Bous2} captures the competition between the vortex stretching $u_{2,y}\th_x$ and the $y$-advection in the Boussinesq equations \eqref{eq:bous20}. This model relates to \eqref{eq:DG} via the connections $\th_x \to -\om, \pa_{xy}(-\D)^{-1} \to -H$.
Note that the $\th-$equation in \eqref{eq:Bous} with \eqref{eq:Bous3} reduces to the incompressible porous media equation \cite{cordoba2007analytical,cordoba2011lack}. %Equations \eqref{eq:Bous2}-\eqref{eq:Bous3} capture the competition between advection $u_{2} \pa_y\th_x$ and vortex stretching $ u_{2,y}\th_x$ in \eqref{eq:Bous}.
Equation \eqref{eq:Bous2} captures the competition between the vortex stretching $u_{2,y}\th_x$ and the $y$-advection $u_{2} \pa_y\th_x$ in \eqref{eq:bous20}. This model relates to \eqref{eq:DG} via the connections $\th_x \to -\om, \pa_{xy}(-\D)^{-1} \to -H$. 
%Both solutions enjoy similar sign and symmetry properties. 
Moreover, the solutions of the two models enjoy similar sign and symmetry properties. 
See more discussions in Appendix \ref{app:bous}. The connection between $\pa_{xy}(-\D)^{-1} $ and $H$ can be justified under some assumptions \cite{choi2014on,hou2013finite,chen2021HL}, though it may not be consistent with the current setting.

Valuable insight from Theorem \ref{thm:GWP} and the connection between the above model and \eqref{eq:DG} is that if $\th_x(x, y)$ vanishes near $y=0$ to order $|y|^{a}$ with $a \geq 1$, the advection may be strong enough to destroy potential singularity formation. 
In the hyperbolic flow scenario, due to the odd symmetry in $y$, a typical $\th$ near the origin is of the form $\th(x, y) \approx c_1  x^{1+\al} y + l.o.t.$ for $\th \in C^{1,\al}$ and $\th(x, y) \approx c_1  x^{2} y + l.o.t.$ for $\th \in C^{\infty}$. In both cases, $\th_x$ vanishes linearly in $y$, and thus the effect of $y-$advection can be an obstacle to singularity formation. Such effect can be overcome by imposing a solid boundary on $y=0$ and singularity formation with $C^{1,\al}$ velocity has been established in \cite{chen2019finite2}. For smooth data, the importance of boundary has been studied in \cite{luo2014potentially,luo2013potentially-2}. In the absence of a boundary, new mechanisms to overcome the advection or a new scenario may be required to obtain singularity formation of \eqref{eq:Bous} in $\R^2$.

\subsubsection{Connections with the SQG equation}

In \cite{Cor10}, Castro-C\'ordoba observed that a solution $\om(y, t)$ of the De Gregorio model \eqref{eq:DG} can be extended to a solution of the SQG equation 
\beq\label{eq:SQG}
\th_t + \uu \cdot \na \theta = 0, \quad \uu = \na^{\perp}(-\D)^{-1/2} \th
\eeq
with infinite energy via the connection $ \th(x, y, t) = x \om(y, t)$. 
We can perform derivations for \eqref{eq:SQG} similar to those in \eqref{eq:Bous}-\eqref{eq:Bous3}. 
%We can derive the competition in \eqref{eq:SQG} similar to those in \eqref{eq:Bous}-\eqref{eq:Bous3}. 
Under this connection, the terms dropped in the derivations are {\it exactly} $0$, and the SQG equation in the hyperbolic-flow scenario \cite{he2021small} reduces {\it exactly} to the DG model \eqref{eq:DG} with a solution in class $X$. Hence, our analysis of \eqref{eq:DG} provides valuable insight into the effect of advection in \eqref{eq:SQG} in such a scenario.
%the hyperbolic-flow scenario.
%in the SQG equation. 
% and the terms dropped in the derivations are {\it exactly} $0$ under this connection. In particular, the SQG equation in the hyperbolic-flow scenario \cite{he2021small} 
% reduces {\it exactly} to the DG model \eqref{eq:DG} with a solution in class $X$. 
Moreover, from Theorem \ref{thm:GWP}, we obtain  a new class of globally smooth non-trivial solutions to \eqref{eq:SQG} with infinite energy.
%and finite time $C^{\al}$ singular solution with $\al < 1$ to \eqref{eq:SQG} with infinite energy.
Note that a globally smooth solution to \eqref{eq:SQG} with finite energy has been constructed in \cite{castro2020}. See also \cite{gravejat2019smooth}. 
Singularity formation of \eqref{eq:SQG} from smooth initial data with infinite energy follows from \cite{chen2019finite}.

Under the radial homogeneity ansatz $\th(t, r, \b) = r^{2 - 2\al} g (t, \b)$, Elgindi-Jeong \cite{elgindi2020symmetries} established a connection between a solution $\th$ to the generalized SQG equation and a solution $g(t, \b)$ to the gCLM model \eqref{eq:gCLM} with $a>1$ up to some lower order term in the velocity operator. Our analysis of the global regularity of \eqref{eq:DG} sheds useful light on the analysis of %gCLM model 
\eqref{eq:gCLM} with $a>1$ and constructing globally non-trivial solutions to the generalized SQG equation using the connection in \cite{elgindi2020symmetries}. In particular, our argument to analyze $u_x(0)$ and a singular integral, which is defined in \eqref{eq:idea_Q} and characterizes the competition between advection and vortex stretching in \eqref{eq:DG}, can be generalized to the gCLM model with $a>1$. See more discussions in Section \ref{sec:conclude}.

% Using this connection and derivations similar to those in \eqref{eq:bous20}-\eqref{eq:Bous3}, 

% we can reduce the SQG equation in the hyperbolic-flow scenario \cite{he2021small}  {\it exactly} to the DG model \eqref{eq:DG} with solution in the class $X$. 

%Derivations for the SQG equation similar to \eqref{eq:bous20}-\eqref{eq:Bous3} can performed. Under such connection, the SQG equation in the hyperbolic

%Finally, it is interesting to note that a solution of the De Gregorio model \eqref{eq:DG} can be extended to a solution of the SQG equation with infinite energy \cite{Cor10}.

\vspace{0.1in}
\paragraph{\bf{Organization of the paper}}

In Section \ref{sec:idea}, we discuss the main ideas in the proofs of the main theorems.
In Section \ref{sec:criterion}, we establish the one-point blowup criterion. In Section \ref{sec:regular}, we discuss the stabilizing effect of the advection in \eqref{eq:DG} and study the positive-definiteness of several quadratic forms, which are the building blocks for the %global well-posedness 
GWP results in Theorem \ref{thm:GWP}. In Section \ref{sec:GWP}, we prove Theorem \ref{thm:GWP}. In Section \ref{sec:blowup}, we construct finite time blowup of \eqref{eq:DG} with $C^{\al}\cap H^s$ data. %In Section \ref{sec:incomp}, we discuss some connections between the DG model and incompressible fluids. 
We make some concluding remarks on the potential generalization of the results in Section \ref{sec:conclude}. Some technical Lemmas and derivations are deferred to the Appendix.

\section{Main ideas and the outline of the proofs}\label{sec:idea}

In this section, we discuss the main ideas and outline the proofs of the main theorems.

\subsection{ Difference between the De Gregorio on $\R$ and on $S^1$}\label{sec:intro_diff}

%In \cite{chen2019finite}, Chen-Hou-Huang established singularity formation of \eqref{eq:DG} on $\R$ from odd initial data $\om_0 \in C_c^{\inf}$ with $\om_0 \leq 0$ for $x \geq 0$. The symmetry and sign properties of such data are the same as those in $X$. 
%Note that from \cite{chen2019finite}, \eqref{eq:DG} on $\R$ blows up in finte time from initial data that has the same sign and symmetry properties as those in $X$.
Note that the initial condition considered in \cite{chen2019finite} that leads to finite time blowup of \eqref{eq:DG} on $\R$ has the same sign and symmetry properties as those in $X$.
To establish the well-posedness results in Theorems \ref{thm:criterion} and \ref{thm:GWP}, 
we need to understand the mechanism on $S^1$ that prevents singularity formation similar to \cite{chen2019finite}. 

 For \eqref{eq:DG} on $S^1$ with $\om \in X$, we have two special points $x=0, x = \pi/2$, which correspond to $x=0, x = \infty$ in the case of $\R$. One of the key differences between two cases is captured by the evolution of $|| \om||_{L^1}$ 
%In \eqref{eq:EE_L12},\eqref{eq:EE_L13}, we derive the following identity for $|| \om||_{L^1}$
%\beq\label{eq:intro_L1}
\[
\f{d}{dt} \B( -\int_0^{\pi/2} \om(x) dx\B)
= \f{2}{\pi} \int_0^{\pi/2}\int_0^{\pi/2} \om(x) \om(y) \cot(x + y) dx dy ,
\]
%\eeq
which is derived in \eqref{eq:EE_L12},\eqref{eq:EE_L13}. Since $\om \leq 0$ on $[0,\pi/2]$, 
%For $ \om \in X $, we have $ \om \leq 0$ and thus 
$-\int_0^{\pi/2} \om(x) dx$ is the same as $|| \om||_{L^1}$.

%The interaction with $x + y \leq \f{\pi}{2}$ on the right hand side leads to the growth of $|| \om||_{L^1}$ due to $\cot(x+y) \geq 0$, while in the caes of $x+ y \geq \pi/2$, it contributes to the decrease of $|| \om||_{L^1}$ due to $\cot(x+y) \geq 0$.
For $x + y \leq \f{\pi}{2}$, the interaction on the right hand side has a positive sign due to $\cot(x+y) \geq 0$, which leads to the growth of $|| \om||_{L^1}$. On the other hand, for $x+ y \geq \f{\pi}{2}$, the interaction has a negative sign, which contributes to the decrease of $|| \om||_{L^1}$. 
% For $x + y \leq \f{\pi}{2}$, the interaction on the right hand side has a positive sign due to $\cot(x+y) \geq 0$, which leads to the growth of $|| \om||_{L^1}$. On the other hand, for $x+ y \geq \f{\pi}{2}$, the interaction has a negative sign, which contributes to the decrease of $|| \om||_{L^1}$. 
The former and the latter interaction can be seen as the interaction near $0$ and $\pi/2$, respectively. % while the latter as that near $x=\pi/2$.
%As we will see later, 
The latter plays a crucial role in our proof as a damping term. For comparison, a similar ODE can be derived 
for \eqref{eq:DG} on $\R$ with $\cot(x+y)$ replaced by $\f{1}{x+y}$. The interaction is always positive and can contribute to the unbounded growth of the singular solution in \cite{chen2019finite} in the far field. 
Yet, for \eqref{eq:DG} on $S^1$, similar growth near $ x= \pi/2$ is prevented due to the above damping term.

Moreover, for \eqref{eq:DG} on $S^1$ with $\om \in X$, we have $- u \in X$ and thus $u_x(0) >0$ and $u_x( \f{\pi}{2}) < 0$ for nontrivial $\om$. The sign of $u_x( \f{\pi}{2})$ suggests that near $x =  \f{\pi}{2}$, the vortex stretching $u_x \om$ in \eqref{eq:DG} depletes the growth of the solution. Using these observations, we show that the nonlinear terms near $x = \f{\pi}{2}$ are harmless. Thus, the main difficulty is the analysis of \eqref{eq:DG} near $x=0$. 

%there are other damping terms near $x= \pi/2$ similar to the above. Using these quantities, we show that interaction near $x =\pi/2$ is harmless. Thus, the main difficulty is the analysis of \eqref{eq:DG} near $x=0$. 

%Another important quantity on $S^1$ is $u_x( \f{\pi}{2})$, which is absent in the case of $\R$ (the analogue is $u_x(\infty)$). We derive another damping term in the ODE of $u_x( \f{\pi}{2})$ that enables us to control $u_x( \f{\pi}{2})$. See more discussions in Section \ref{sec:idea_criterion}. 
%Using these two main damping terms, we show that interaction near $x =\pi/2$ is harmless. Thus, the main difficulty lies in the analysis of \eqref{eq:DG} near $x=0$. 

\subsection{The one-point blowup criterion}\label{sec:idea_criterion}

In \cite{lei2019constantin}, an important equation was discovered 
%To prove Theorem \ref{thm:criterion}, we use an important equation of $\f{\om_x^2}{\om}$ discovered in \cite{lei2019constantin}
\beq\label{eq:imp0}
\f{1}{2} \pa_t ( ( \sqrt{\om} )^{\prime} )^2
= -\f{1}{2} u ( ( (\sqrt{\om}) ^{\prime} )^2 )^{\prime} - \f{1}{2} H \om ( ( \sqrt{\om} )^{\prime} )^2 + \f{1}{4} ( H \om)^{\prime} \om^{\prime},
\eeq
%To prove Theorem \ref{thm:criterion}, we use the following identity derived from \eqref{eq:imp0}
%A direct consequence of \eqref{eq:imp0} is the following identity
which implies
\beq\label{eq:imp}
\f{1}{2} \pa_t \f{\om_x^2}{\om} = - \f{1}{2} \B(  u \f{\om_x^2}{\om} \B)_x + \om_x H \om_x.
\eeq
Identity \eqref{eq:imp} can also be obtained from the equation of $\om_x$ and $\om^{-1}$ using \eqref{eq:DG}. 

To prove Theorem \ref{thm:criterion}, one of the key steps is the estimate of a new quantity $\int_{0}^{\pi/2}  \f{\om_x^2}{\om} \sin(2x) dx$. The vanishing property of $\sin(2x)$ near $x =0, \f{\pi}{2}$ cancels the singularity caused by $\f{1}{\om}$ for $\om \in X$. Since $\om(t)$ remains in $X$ \eqref{eq:X} and $\om \leq 0$ on $[0, \f{\pi}{2} ]$, 
%It is not difficult to obtain that the symmetry and sign properties of the solution are preserved and 
$ \f{\om_x^2}{\om} \sin(2x) $ has a fixed sign. To control the nonlinear terms in the energy estimate, 
%right hand side of \eqref{eq:imp} in the energy estimate, 
we will exploit the conservation form $ \B(  u \f{\om_x^2}{\om} \B)_x $, use 
an important cancellation on a quadratic form of $\om_x$ and a crucial extrapolation inequality on $u$.
%, and estimate $u  (\sin 2x)^{-1}$ using $u_x(0), || \om||_{L^1}$, $\int_{0}^{\pi/2}  \f{\om_x^2}{\om} \sin(2x) dx$ and a crucial extrapolation inequality. 
Using some estimates in \cite{chen2020singularity,chen2019finite}, we derive a-priori estimates on $ u_x(0), || \om||_{L^1}, \int_{0}^{\pi/2}  \f{\om_x^2}{\om} \sin(2x) dx$, which controls $\om(x)$ away from $x = \f{\pi}{2} $ by interpolation. 
By exploiting the damping mechanisms near $x= \pi/2$ discussed in Section \ref{sec:intro_diff}, we further show that  $u_x( \f{\pi}{2}, t)$ cannot blow up before the blowup of $u_x(0, t)$. 
%We further control the solution near $x = \f{\pi}{2}$ using the important observation in Section \ref{sec:intro_diff} that the vortex stretching $u_x\om$ depletes the growth of the solution near $x = \f{\pi}{2}$. In particular, we show that  $u_x( \f{\pi}{2})$ cannot blowup before the blowup of $u_x(0)$. 
With these estimates, we obtain an a-priori estimate on $|| \om||_{L^{\inf}}$ in terms of $\int_0^{t} u_x(0,s) ds$, and establish the one-point blowup criterion by applying the Beale-Kato-Majda type blowup criterion \cite{beale1984remarks,Sve19}. See also \cite{OSW08}.

\subsection{Global well-posedness}\label{sec:idea_GWP}

To prove Theorem \ref{thm:GWP} using Theorem \ref{thm:criterion}, we need to further control $u_x(0)$. In the special case of $\om_0 \in C^{1,\al}$ with $\om_{0, x}(0) = 0$, the key step is to establish
\beq\label{eq:conser}
 \f{d}{dt} \int_{ 0 }^{ \f{\pi}{2}} \om \cot^2 x dx  =  \int_{ 0 }^{ \f{\pi}{2}} (u_x \om - u \om_x) \cot^2 x dx \geq 0 .
%\f{d}{dt} Q(2, t)  = - \int_{ 0 }^{ \f{\pi}{2}} (u_x \om - u \om_x) \cot^2 x dx \leq 0 , \quadQ(\b, t) \teq -\int_{ 0 }^{ \f{\pi}{2}} \om \cot^2 x dx
\eeq
The quantity $\int_0^{\pi/ 2} \om \cot^2 x  dx $ is well-defined for $\om \in C^{1,\al}$ with $\om_x(0) = 0$ and $\al > 0$. The above inequality quantifies that the stabilizing effect of advection is stronger than the effect of vortex stretching in some sense for $\om$ in this case.
%if $\om$ vanishes near $x=0$ to order $|x|^{a}, a> 1$. 
%The above inequality characterizes the competition between the advection and vortex stretching, and quantifies that the regularizing effect of advection is stronger than the effect of vortex stretching in some sense. 
%Formally, the case $\om_{0,x}(0) = 0$ can be regarded as the subcritical case. 
 %advection dominates. 
 We will exploit the convolution structure in the quadratic form 
in \eqref{eq:conser} and use an idea from Bochner's theorem for a positive-definite function to establish \eqref{eq:conser}. We remark that an inequality similar to \eqref{eq:conser} has been established in the arXiv version of \cite{chen2019finite}, where a more singular function $\cot^{\b} x$ with $\b \geq 2.2$ is used.
%the singular weight $\cot^2 x$ is 
%exponent $\b$ of the singular weight $\cot^{\b} x$ satisfies $\b \geq 2.2$. 
The inequality \eqref{eq:conser} is stronger than that in \cite{chen2019finite} since 
$\int_0^{\pi/2} \om (\cot x)^{\b} dx$ is not well-defined for $\om \in C^{1,\al}$ with $\al \in (0,\b-2)$ and $\om_x(0)=0$.  
 %the latter is not applicable and $\om(y) \cot^{\b}(y) \notin L^1( [0,\pi/2])$ for $\om \in C^{1,\al}$ with $\al \in (0,\b-2)$. 
%for $\om_0$ that vanishes near $x=0$ to order $|x|^{\g}$ with $1 < \g < \b - 1$.  
Since $\om \leq 0$ on $[0, \pi /2]$, \eqref{eq:conser} implies an a-priori estimate of $ \int_0^{\pi/2} | \om \cot^2 x | dx$, based on which we can further control $||\om||_{L^1}, u_x(0)$ and establish the global well-posedness. 
%Then following the argument in the arXiv version of \cite{chen2019finite} and using Theorem \ref{thm:criterion}, we control $u_x(0)$ and establish the global well-posedness. 

% The above case with $\om_{0, x}(0) = 0$ can be interpreted as the subcritical case since $\om \approx - C |x|^{1+\al}$ near $x=0$ and the regularizing effect of advection is stronger. See the discussion in section \ref{sec:intro_comp}. Similarly, the case $\om_{0, x} \neq 0$ can be interpreted as the critical case.

%In the case of $\om_{0,x}(0) \neq 0$, 
In the general case, $\om_0$ can vanish only linearly near $x=0$. The proof is much more challenging since $\int_0^{\pi/2} | \om \cot^2 x | dx$ is not well-defined, and there is no similar coercive conserved quantity. Note that in this case, for $\om_0$ close to $A \sin 2 x$ in the $C^2$ norm, the solution $\om(x, t)$ converges to $A \sin 2 x$ as $t \to \infty$ \cite{Sve19}. As pointed out in \cite{Sve19}, this imposes strong constraints on possible conserved quantities. 
Thus, it is not expected that there is any good conserved quantity similar to some weighted norm of $\om$. 
%Similar observations have been made in \cite{Sve19}. 

To illustrate our main ideas, we consider $\om_0 \in  C^{1,\al} \cap X$ with $\om_{0,x} \neq 0$. In this case,
%In this case with  $\om_0 \in C^{1,\al} \cap X$, 
the only conserved quantities seems to be $\om_x( x, t) \equiv \om_{0,x}(x)$ for $x = 0, \f{\pi}{2}$. Surprisingly, the one-point conservation law $\om_x(0, t) \equiv \om_{0,x}(0)$ allows us to control $Q(\b, t)$ defined below
%$Q(\b, t) = \int_0^{\pi/ 2} |\om \cot^{\b} x | dx$ 
for $\b <2$. We remark that we do not have monotonicity of $Q(\b, t)$ in $t$ similar to \eqref{eq:conser} when $\b < 2$. A crucial observation is the following leading order structure 
\beq\label{eq:idea_Q}
Q(\b, t) \teq \int_0^{\pi/2} -\om(y,t) (\cot y )^{\b} dy 
= \f{ - \om_x(0)}{2-\b} + \cR(\b, t), \quad  | \cR(\b, t)| \les_{\al} || \om||_{C^{1,\al}},
\eeq
for any $\b < 2$. As long as $\om(t)$ remains in $C^{1,\al}$, we can choose $\b$ sufficiently close to $2$, such that $(2-\b) Q(\b, t)$ is comparable to $ - \om_x(0)$, which is time-independent. 
Using this observation, an ODE of $Q(\b, t)$ similar to \eqref{eq:conser} but with a nonlinear forcing term and an additional extrapolation-type estimate, we can control $Q( \b(t), t)$ with $\b(t)$ sufficiently close to $2$. %{\color{blue}For the less regular initial data $\om_0 \in X \cap H^1$ with $\om_0 x^{-1} \in L^{\inf}$,  we will establish an estimate similar to the above.} 
In the case of the less regular initial data $\om_0 \in X \cap H^1$ with $\om_0 x^{-1} \in L^{\inf}$, we will establish an estimate similar to \eqref{eq:idea_Q}. This enables us to further control $u_x(0)$ and establish the global well-posedness.

% In the above subcritical case ($\om_{0, x}(0)= 0$), the proof relies on the coercive conserved quantity $Q(2,t)$. % In the critical case, the proof is much more challenging since $\int_0^{\pi} | \om \cot^2 x | dx$ is not well defined and there is no similar coercive conserved quantity. Note that in the critical case, from \cite{Sve19}, for $\om_0$ close to $A \sin 2 x$, the solution $\om(x, t)$ converges to $A \sin 2 x$ as $t \to \infty$. This put strong constraints on possible conserved quantities. Thus, it is not expected that there is any good conserved quantity similar to some weighted norm of $\om$. See also similar discussion in \cite{Sve19}. 

\subsection{Finite time blowup}
To prove Theorem \ref{thm:blowup}, we follow the method in the work of Chen-Hou-Huang \cite{chen2019finite}. We also adopt an idea developed in our previous work  \cite{chen2020slightly} 
that a singular solution of the gCLM model \eqref{eq:gCLM} can be constructed by perturbing the equilibrium $ \sin (2 x)$ of \eqref{eq:DG}. %Our key observation is that the advection in \eqref{eq:DG} can be weaken slightly by constructing 
We first construct a $C^{\al}$ approximate self-similar profile of \eqref{eq:DG} $\om_{\al} = C \cdot \sgn(x) | \sin 2x|^{\al}$ with $\al < 1$ sufficiently close to $1$. Our key observation is that for $\al < 1$, the advection $u \om_x$ is slightly weaker than the vortex stretching $u_x \om$. See the discussions in the paragraph before Section \ref{sec:CLM} and in Section \ref{sec:intro_comp}. Then we establish the nonlinear stability of the profile $\om_{\al}$ in the dynamic rescaling formulation of \eqref{eq:DG} based on the coercivity estimates of a linearized operator established in \cite{lei2019constantin} and several weighted estimates. 
Using the nonlinear stability results and the argument in \cite{chen2019finite,chen2020slightly}, we further establish finite time blowup.

The finite time singularity of \eqref{eq:DG} on $\R$ from $C_c^{\inf}$ initial data established in \cite{chen2019finite} has expanding support, and the vorticity blows up at $\inf$.
The singularities of the gCLM model \eqref{eq:gCLM} with weak advection constructed in \cite{Elg17,chen2019finite,Elg19,chen2020singularity} are focusing, and the blowups occur at the origin. Due to the relatively strong advection and the compactness of a circle, the $C^{\al}$ singular solution of \eqref{eq:DG} on $S^1$ we construct is neither expanding nor focusing, which is similar to that in \cite{chen2020slightly}. Moreover, the solution blows up in most places at the blowup time. 
%Compared to the blowup analysis of the gCLM model in our previous work \cite{chen2020slightly}, 
Compared to the analysis of the gCLM model in \cite{chen2020slightly}, the blowup analysis of \eqref{eq:DG} with $C^{\al}$ data is more complicated due to the less regular profile and its estimates in the nonlinear stability analysis with singular weights.

\section{One-point blowup criterion}\label{sec:criterion}

In this section, we establish the one-point blowup criterion in Theorem \ref{thm:criterion}.

Recall the class $X$ defined in \eqref{eq:X} and the Hilbert transform on a circle with period $\pi$
\beq\label{eq:hil}
u_x = H \om = \f{1}{\pi} P.V. \int_{-\pi/2}^{\pi/2} \om(y) \cot (x-y) dy,  \quad
u =  -\f{1}{\pi} \int_{-\pi/2}^{\pi/2} \om(y) \log \B| \f{\sin(x+y)}{\sin(x-y)} \B| dy.
\eeq
For \eqref{eq:DG} with initial data $\om_0 \in X$, it is not difficult to obtain that $\om( \cdot, t ), -u(\cdot , t)$ remain in $X$.
%$, and $-u( \cdot, t) \in X$.

\subsection{Energy estimate}\label{sec:EE_near}
To perform energy estimate using \eqref{eq:imp}, we multiply both sides of \eqref{eq:imp} with $-\sin(2x) \in X$ so that $-\f{\om_x^2}{\om} \sin(2x) \geq 0$. Integrating them over $S^1$, we obtain 
\beq\label{eq:EE_A1}
\f{1}{2} \f{d}{dt} \int_{S^1}  -\f{\om_x^2}{\om} \sin(2x) dx 
=   \f{1}{2} \int_{S^1} \B(  u \f{\om_x^2}{\om} \B)_x \sin (2x) dx 
- \int_{S^1} \om_x H \om_x \sin(2x ) dx \teq I + II.
\eeq

We introduce the following functionals  
\beq\label{eg:A}
A(\om) \teq  \int_{S^1}  -\f{\om_x^2}{\om} \sin(2x) dx , \quad 
E(\om) = A(\om) + u_x(0) + || \om||_{L^1}, 
\quad U(t) \teq \int_0^{t} u_x(0, s) ds.
\eeq

%\subsection{ Estimate of $II$}

We choose the special function $\sin 2x$ due to the crucial cancellation in Lemma \ref{lem:cancel}
\beq\label{eq:cancel1}
II =  \int_{S^1} \om_x H \om_x \sin(2x ) dx  = 0.
\eeq

% An important reason that we use the special function $\sin 2x$ is to obtain a crucial cancellation 
% \beq\label{eq:cancel1}
% II =  \int_{S^1} \om_x H \om_x \sin(2x ) dx  = 0,
% \eeq
% whose proof is deferred to Section \ref{sec:ineq}.

For $I$, using integration by parts, we obtain 
\[
I = - \f{1}{2} \int_{S^1}  u \f{\om_x^2}{\om}  (\sin (2x) )_xdx 
= -\int_{S^1} \f{u \cos(2x) }{\sin(2x)} \f{\om_x^2}{\om} \sin(2x) dx 
= -2 \int_0^{\pi/2} \f{u \cos(2x) }{\sin(2x)} \f{\om_x^2}{\om} \sin(2x) dx .
\]

A crucial observation is that by taking advantage of the conservation form $( u \f{\om_x^2}{\om})_x$ and performing estimate on \eqref{eq:imp} with an explicit function, the coefficient $ \f{u \cos (2x)}{ \sin(2x)} $ in the nonlinear term $I$ for $x$ away from $x =0, \f{\pi}{2}$ is of lower order than $u_x, \om$. %This enables us to use an extrapolation below to close the energy estimate. 
We further estimate $I$ from above. 
Since $\om , - u\in X$, we derive $- \f{\om_x^2}{\om} \sin(2x) \geq 0,  \f{u}{\sin(2x)} \geq 0 $, and $\cos(2x) \leq 0 $ on $ [\f{\pi}{4}, \f{\pi}{2}]$. It follows 
% Since $\om \in X$, it is not difficult to obtain $u$ is odd with $u \geq 0$ on $[0, \f{\pi}{2}]$. Note that $- \f{\om_x^2}{\om} \sin(2x) \geq 0,  \f{u}{\sin(2x)} \geq 0 $, and $\cos(2x) \leq 0 $ on $ [\f{\pi}{4}, \f{\pi}{2}]$. 
%We obtain 
\beq\label{eq:EE_A2}
I \leq -2 \int_0^{\pi/4} \f{u \cos(2x) }{\sin(2x)} \f{\om_x^2}{\om} \sin(2x) dx 
\les \B| \B| \f{u}{\sin x}\B|\B|_{L^{\inf}[0,\f{\pi}{4}]} A(\om),
\eeq
where $A(\om)$ is defined in \eqref{eg:A}. The fact that the nonlinear term in $[\pi/4, \pi/2]$ is harmless is related to the discussion in Section \ref{sec:intro_diff}. To control $ \f{u}{\sin x}$, we use the following extrapolation.

\begin{lem}\label{lem:extrap}
Suppse that $\om \in X$ satisfies $A(\om) <+\inf, u_x(0) < +\inf$ and $\om \in L^1$. We have 
\begin{align}
\B| \B| \f{u}{\sin x} \B| \B|_{L^{\inf}[0,\f{\pi}{4}]} 
&\les ( u_x(0) + || \om||_{L^1} + 1) \log( || \om||_{L^{\inf}[0, \f{\pi}{3}]} + 2) ,  \label{eq:extrap}  \\
|| \ |\cos x|^{1/2} \om||_{L^{\inf} } &\les ( A(\om) ( u_x(0) + ||\om||_{L^1} )  )^{1/2} ,
%|| |\cos x|^{1/2} \om||_{L^{\inf}[0, \f{\pi}{3}  ] } &\les ( A(\om) u_x(0) )^{1/2} ,
\quad 
|| \sin x \cdot \om||_{L^{\inf}  } \les ( A(\om) |u_x( \pi/2 )| )^{1/2} .\label{eq:interp1} 
\end{align}
\end{lem}

We remark that $|| \om||_{L^{\inf}[0, \f{\pi}{3}]}$ can be further bounded by $|| \ |\cos x|^{1/2} \om||_{L^{\inf} }$.

% The proof of Lemma \ref{lem:extrap} is deferred to Section \ref{sec:ineq}. We remark that
% $|| \om||_{L^{\inf}[0, \f{\pi}{3}]}$ can be further bounded by $|| \ |\cos x|^{1/2} \om||_{L^{\inf} }$. Let us motivate the above inequality. 
% From \eqref{eq:hil}, we get
% \beq\label{eq:extrap1}
% \f{u}{\sin x}  = -\f{1}{\pi} \int_0^{\pi/2} \om(y)  \f{ 1}{\sin x} \log \B| \f{\sin(x+y)}{\sin(x-y)} \B| dy =
% -\f{1}{\pi} \int_0^{\pi/2} \f{\om(y) }{\sin y} \f{ \sin y}{\sin x} \log \B| \f{\sin(x+y)}{\sin(x-y)} \B| dy 
% \eeq
% To control the integral away from the singular region, we use $u_x(0) + || \om||_1$, which can be seen as a weighted $L^1$ norm. To control the integral in the singular region, we use $||\om||_{L^{\inf}}$.  

%\subsection{ Extrapolation inequalities} \label{sec:ineq}

\begin{proof}
Denote 
\[
K(x, y) =  \f{ \sin y}{\sin x} \log \B| \f{\sin(x+y)}{\sin(x-y)} \B|
= \f{ \sin y}{\sin x} \log \B| \f{\tan x + \tan y}{\tan x - \tan y} \B|, \quad f(x) = x \log \B| \f{x+1}{x-1} \B|.
\]

From \eqref{eq:hil}, we get
\beq\label{eq:extrap1}
\f{u}{\sin x}  = -\f{1}{\pi} \int_0^{\pi/2} \om(y)  \f{ 1 }{\sin x} \log \B| \f{\sin(x+y)}{\sin(x-y)} \B| dy =
-\f{1}{\pi} \int_0^{\pi/2} \f{\om(y) }{\sin y} K(x, y) dy .
%\f{ \sin y}{\sin x} \log \B| \f{\sin(x+y)}{\sin(x-y)} \B| dy  .
\eeq

%In this subsection, we prove Lemma \ref{lem:extrap}. 

For $\e < \f{1}{10}$ to be determined, we decompose \eqref{eq:extrap1} as follows 
\[
\bal
\B|\f{u}{\sin x} \B|
&\les \int_0^{\pi/2}\one_{ |y/x - 1| > \e} \B| \f{\om(y)}{\sin y} K(x, y) \B| dy 
+\int_0^{\pi/2} \one_{ |y/x - 1 | \leq \e}  \B| \f{\om(y)}{\sin x} \B| \log \B| \f{\sin(x+y)}{\sin(x-y)} \B| dy   \\
&\teq I + II.
\eal
\]

Denote $z = \f{\tan y}{\tan x}$. For $ |y/x - 1|  > \e, x , y\in [0, \pi/2]$, we have 
\[
 |z-1| = \B|\f{\tan y - \tan x}{\tan x}\B|
 = \B| \f{\sin(x-y) }{ \cos x \cdot  \cos y \cdot \tan x} \B|
 =\B| \f{\sin(x-y)}{ \cos y \cdot \sin x}\B| \gtr \f{|x-y|}{x} \gtr \e.
\]

For $ x \in [0, \f{\pi}{4}]$ and $y \in [0, \f{\pi}{2}]$, using $\sin x \asymp \tan x, \sin y \leq \tan y$ and the above estimate, we get
\[
K(x, y) \les  \f{\tan y}{\tan x}  \log \B| \f{\tan x + \tan y}{\tan x-  \tan y}\B|
= z \log \B|\f{ z+1}{z-1} \B| = f(z) \les \log \e^{-1}, 
\] 
where we have used $f(z) \les 1$ for $z>2$ and $z< \f{1}{2}$ to obtain the last inequality. It follows 
\[
I \les \log \e^{-1} \int_0^{\pi/2} \f{|\om|}{\sin y} dy
\les \log \e^{-1} \int_0^{\pi/2} (-\om(y)) (\cot y + 1) dy
\les \log \e^{-1}  ( u_x(0) + || \om||_1).
\]

For $II$, since $ |\f{y}{x}-1| \leq \e < \f{1}{10} $ and $x \in [0, \pi / 4]$, we yield $y \in [0,\f{\pi}{3}]$. Since $\sin z \asymp z $ on $[0, 3\pi/4]$, we get $ \B|\f{\sin(x+y)}{\sin(x-y)} \B| \les \B|\f{ x+y}{x-y}\B|$. Using these estimates, we derive 
\[
II \les || \om||_{L^{\inf}(0,\f{\pi}{3})} \int_{ |y/x - 1| \leq \e}  (1 + \log \B| \f{y+x}{y-x} \B| ) \f{1}{x} dy
=|| \om||_{L^{\inf}(0,\f{\pi}{3})}  \int_{1-\e}^{1+\e} (1 + \log \B|\f{1+z}{1-z}\B| ) d z.
\]
Using a change of variable $ s = z-1 \in [-\e, \e]$, we further obtain 
\[
II \les || \om||_{L^{\inf}(0,\f{\pi}{3})} \int_{ |s| \leq \e} \log |s|^{-1} ds \les 
\e \log \e^{-1} || \om||_{L^{\inf}(0,\f{\pi}{3})} .
\]

Choosing $\e = ( || \om||_{L^{\inf}(0,\f{\pi}{3})} + 10)^{-1} <\f{1}{10}$, we prove 
\[
|| u (\sin x)^{-1} ||_{L^{\inf}[0,\pi/4]} \les
( u_x(0) + || \om||_1 + 1) \log \e^{-1}
\les ( u_x(0) + || \om||_1 + 1) \log (  || \om||_{L^{\inf}(0,\f{\pi}{3})} + 2),
\]
which is exactly \eqref{eq:extrap}. 

For $x \in [0, \f{\pi}{2}]$, using the Cauchy-Schwarz inequality, we prove 
\[
\bal
| \om(x) (\cos x)^{1/2}| 
&\leq   (\cos x)^{1/2}\int_0^x |\om_x(y) | dy 
\leq \int_0^x |\om_x(y)| (\cos y)^{1/2} dy \\
&\les \B( \int_0^{\pi/2} \f{ \om_x^2}{ |\om|} \sin(2x) dx
\int_0^{\pi/2} |\om| (\cot x + 1) dx \B)^{1/2}  
 \les ( A(\om)  (u_x(0) + ||\om||_{L^1} ) )^{1/2},
\eal
\]
which is the first inequality in \eqref{eq:interp1}. The proof of the second inequality in \eqref{eq:interp1} is similar.
\end{proof}

\subsubsection{Estimates of $|| \om||_{L^1}, u_x(0)$}
To close the energy estimate using Lemma \ref{lem:extrap}, we further estimate $|| \om||_{L^1}, u_x(0)$ in terms of $U(t)$. Similar estimates have been established in  \cite{chen2020singularity} and the arXiv version of \cite{chen2019finite}. Integrating \eqref{eq:DG} over $[0, \f{\pi}{2}]$ and using integration by parts, we yield 
\beq\label{eq:EE_L1}
\f{d}{dt} \int_0^{\pi/2} -\om d x  = \int_0^{\pi/2} -u_x \om + u \om_x dx 
= - 2 \int_0^{\pi/2} u_x \om dx \teq III.
\eeq

Since $\om$ is odd, symmetrizing the kernel in \eqref{eq:hil}, we obtain
\beq\label{eq:EE_L12}
\bal
&III  = - \f{2}{\pi} \int_0^{\pi/2} \om(x)  \int_0^{\pi/2} \om(y) \B( \cot(x-y) - \cot(x+y) \B) dy  dx \\
= &  \f{2}{\pi} \int_0^{\pi/2} \int_0^{\pi/2} \om(x) \om(y) \cot( x+ y) dx dy 
= \f{4}{\pi} \int_0^{\pi/2} (-\om(x) ) \B( - \int_0^{x}  \om(y) \cot(x+y ) dy\B).
\eal
\eeq

Since $-\om(x) \geq 0 $ on $[0, \f{\pi}{2}]$ and $\cot z$ is decreasing on $[0, \pi]$, we get
\beq\label{eq:EE_L13}
 - \int_0^{x}  \om(y) \cot(x+y ) dy
 \leq -\int_0^{x} \om(y) \cot y dy 
 \leq -\int_0^{\pi/2} \om(y) \cot y dy  \les u_x(0).
\eeq

It follows 
\[
III \les u_x(0) \int_0^{\pi/2} (-\om(y)) dy , \quad 
\f{d}{dt} \int_0^{\pi/2} - \om (y) dy = III \les u_x(0)  \int_0^{\pi/2} - \om (y) dy .
\]

Using Gronwall's inequality, we establish 
\[
|| \om(t)||_{L^1} \leq || \om_0||_{L^1} \exp(  C \int_0^t u_x(0, s) ds) \les || \om_0||_{L^1} \exp( C U(t)).
\]

Taking the Hilbert transform on both side of \eqref{eq:DG} and applying Lemma \ref{lem:tri}, we derive %the equation for $u_x(0)$
\beq\label{eq:EE_A3}
\bal
\f{d}{dt} u_x(0) & = H( u_x\om - u \om_x)(0)
= 2 H(u_x \om )(0) - H ( \pa_x( u \om))(0) \\
&= u_x^2(0) - \om^2(0) + \f{1}{\pi} \int_{-\pi/2}^{ \pi/2} \cot y ( u\om)_x(y) dy 
= u_x^2(0)  + \f{1}{\pi } \int_{-\pi/2}^{\pi/2} \f{1}{\sin^2 y} u \om dy .
\eal
\eeq

Note that $u \om \leq 0$ for all $x$ and $u_x(0) \geq 0$ for $\om \in X$. It follows 
\[
\f{d}{dt} u_x(0) \leq u_x^2(0).
\]

Using Gronwall's inequality, we obtain 
%\beq\label{eq:EE_ux0}
\[
0 \leq u_x( 0, t) \leq  u_x(0, 0) \exp( U(t)) = H \om_0(0) \exp( U(t)) .
\]

Plugging the above estimates, \eqref{eq:cancel1}, \eqref{eq:EE_A2}  and Lemma \ref{lem:extrap} in \eqref{eq:EE_A1}, we obtain 
\[
\bal
\f{d}{dt} A(\om) &\les  C( || \om_0||_{L^1}, H\om_0(0)) \exp( C U(t)) \cdot A(\om) 
\log\B(  ( A(\om) ( u_x(0) + ||\om||_{L^1}) )^{1/2} + 2\B)  , \\
%&\les C( || \om_0||_{L^1}, H\om_0(0)) \exp(U(t))A^2(\om) ( \log( A(\om) + 2)+ \log( H \om_0(0) + 2) + U(t) )
\eal
\] 
where $C( || \om_0||_{L^1}, H\om_0(0))$ is some constant only depending on $|| \om_0||_{L^1}, H\om_0(0)$. Recall the energy $E(\om)$ in \eqref{eg:A}. Combining the above estimates, we establish 
\[
\f{d}{dt} E(\om) \les C( || \om_0||_{L^1}, H\om_0(0)) \exp( C U(t))  \cdot E \log( E+ 2).
\]

Solving the differential inequality, we prove 
\beq\label{eq:EE_Aw}
E(\om )  \leq ( E(\om_0) + 2) \exp( \exp (  C( || \om_0||_{L^1}, H\om_0(0)) \int_0^{t}\exp( C U(s)) ds  )) .
\eeq
%  Note that 
% \[
% \bal
% \log( A(\om) ( u_x(0) + ||\om||_{L^1}) + 2)
% &\les \log( A(\om) + 2) + \log( (H \om_0(0) + || \om_0||_{L^1} +  2)  \exp( C U(t)) \\
% &\les C( ||\om_0||_{L^1}, H\om_0(0)) ( U(t)  + 1)  \log( A(\om) + 2).
% \eal
% \]

% We yield
% \[
% \f{d}{dt} A^2(\om) \les C( || \om_0||_{L^1}, H\om_0(0))  \exp( 2U(t))A^2(\om) 
% \log( A(\om) + 2),
% \]
% which further implies 
% \beq\label{eq:EE_Aw}
% A(\om )  \leq ( A(\om_0) + 2) \exp( \exp (  C( || \om_0||_{L^1}, H\om_0(0)) \int_0^{t}\exp( 2U(s)) ds  )) .
% \eeq

%\subsubsection{ Estimate of $u_x(t,\pi/2)$}
\subsection{ Estimate near $x = \f{\pi}{2}$}\label{sec:EE_far}

In view of Lemma \ref{lem:extrap}, we have control of  $||\om||_{L^{\inf}[0,a]}$  using $A(\om), u_x(0)$ and $|| \om||_{L^1}$ only away from $x = \f{\pi}{2}$, i.e. $a< \f{\pi}{2}$,  due to the vanishing weight $(\cos x)^{1/2}$. 
%on part of the domain, e.g. $[0,\pi/3]$, using $A(\om), u_x(0)$ and $|| \om||_{L^1}$. 
We further estimate $u_x(\pi/2, t)$ so that we can apply Lemma \ref{lem:extrap} to control $||\om||_{\inf}$. This will enable us to apply the BKM type blowup criterion for \eqref{eq:DG} to establish Theorem \ref{thm:criterion}.

Using a derivation similar to that in \eqref{eq:EE_A3}, we obtain 
\beq\label{eq:EE_B0}
\f{d}{dt} u_x(\f{\pi}{2}) = u_x^2(\f{\pi}{2}) + \f{1}{\pi} \int_0^{\pi} \f{1}{\cos^2(y)} u \om dy
\teq I + II.
\eeq

A crucial observation is that for $\om \in X$,  $u_x(\f{\pi}{2}) = \f{1}{\pi}\int_0^{\pi}  \om(y) \tan(y) dy $ 
% \f{1}{\pi} \int_{-\pi/2}^{\pi/2} \om(y) \tan(y) dy $ 
is negative. Thus the vortex stretching term $u_x^2(\f{\pi}{2})$ depletes the growth of $u_x(\f{\pi}{2})$, which is the main mechanism that $u_x(\f{\pi}{2})$ does not blowup as long as $U(t)$ is bounded. See also Section \ref{sec:intro_diff}. On the other hand, since $u\om \leq 0$, the advection term $\f{1}{\pi} \int_0^{\pi} \f{1}{\cos^2(y)} u \om dy$ is negative and contributes to the growth of $u_x(\f{\pi}{2})$. Our goal is to show that the growing effect is weaker. The main difficulty is the singular functions $ (\cos y)^{-2}, \tan y$ near $y = \f{\pi}{2}$ in $I$ and $II$ since we can control $\om$ away from $y = \f{\pi}{2}$.

For $II$, we decompose it as follows 
\[
II = \f{1}{\pi} \int_0^{\pi} \tan^2(y) u \om dy + \f{1}{\pi} \int_0^{\pi} u \om dy = II_1 + II_2. 
\]
Since  $II_2$ does not involve a singular function, the estimate of $II_2$ is simple. Using \eqref{eq:hil}, we get
\[
\bal
|u(x)| &\les \int_0^{\pi} |\om(y)| |\cos y|^{1/2} | \cos y|^{-1/2}   | \log |\sin(x-y)| | dy \\
 &\les || \ |\cos x|^{1/2} \om ||_{\inf} 
|| |\cos x|^{-1/2} ||_{L^{4/3}} || \log x ||_{L^4} \les || \ |\cos x|^{1/2} \om ||_{\inf} .
\eal
\]

It follows 
\beq\label{eq:EE_B1}
|II_2| \leq || u||_{L^{\inf}} || \om ||_{L^1}
\les  || \ |\cos x|^{1/2} \om ||_{\inf}  || \om||_{L^1}.
\eeq

For $I$ and $II_1$, our goal is to establish 
\beq\label{eq:EE_B2}
I + II_1 \geq  \f{1}{4} u_x^2(\f{\pi}{2}) - C |u_x( \f{\pi}{2})| \cdot || \om||_{L^{\inf}}.
\eeq

We will further use Lemma \ref{lem:extrap} and $\e-$Young's inequality to estimate $|u_x( \f{\pi}{2})| \cdot || \om||_{L^{\inf}}$ and close the estimate of $u_x(\f{\pi}{2})$ in \eqref{eq:EE_B0}.
Note that near $y = \f{\pi}{2}$, we have $ (\cos y)^{-1}, \tan y = \f{1}{\pi/2 - y} + O( | \pi/2 - y| )$. For simplicity, we consider the coordinate near $\f{\pi}{2}$ and introduce 
\beq\label{eq:rota_h_pi}
f = \om(x + \f{\pi}{2}), \quad g = u(x + \f{\pi}{2}), \quad s(x, y) = \f{\tan y}{\tan x}.
\eeq

\begin{remark}
Since $\tan z = z + O(z^3), \sin z = z + O(z^3)$ near $z = 0$, in the following derivations, we essentially  treat $\tan z, \sin z$ similar to $z$.
\end{remark}

Clearly, $g_x = H f$, $g$ and $f$ are odd and $f \geq  0, g \leq 0$ on $(0, \f{\pi}{2})$. Using \eqref{eq:hil}, \eqref{eq:rota_h_pi}, $(\tan(x+\pi/2))^2 =  (\tan x)^{-2}$ and symmetrizing the integrals in $I, II_1$, we get 
\[
\bal
I &= (H \om(\f{\pi}{2}))^2 = (H f(0) )^2
= \f{4}{\pi^2}\int_0^{\pi/2} \int_0^{\pi/2} f(x) f(y) \cot x \cot y dx dy \\
&= \f{4}{\pi^2} \int_0^{\pi/2} \int_0^{\pi/2} \f{f(x) f(y)}{\tan x \cdot \tan y}  dx dy, \\
% \eal
% \]
% \[
% \bal]
II_1 &= \f{1}{\pi} \int_0^{\pi} \f{f g}{\tan^2 x} dx
= \f{2}{\pi}\int_0^{\pi/2} \f{f g}{\tan^2 x} dx
= -\f{2}{\pi^2} \int_0^{\pi/2} \f{f(x)}{\tan^2 x} \int_0^{\pi/2} f(y) \log \B| \f{\sin(x+y)}{ \sin(x-y)} \B| dy \\
& =  -\f{1}{\pi^2} \int_0^{\pi/2}\int_0^{\pi/2}f(x) f(y) ( \f{1}{\tan^2 x} + \f{1}{\tan^2 y})  \log \B| \f{\sin(x+y)}{ \sin(x-y)} \B| dy. \\
%& = -\f{1}{\pi^2} \int_0^{\pi/2}\int_0^{\pi/2} \f{f(x) f(y)}{ \tan^2 y } ( \f{1}{\tan^2 x} + \f{1}{\tan^2 y})  \log \B| \f{\sin(x+y)}{ \sin(x-y)} \B| dy
%= \f{1}{\pi^2} \iint_{[ 0, \pi/2]^2}
\eal
\]

Recall $s$ from \eqref{eq:rota_h_pi}. Note that 
\beq\label{eq:log_s}
 \B|\f{\sin(x+y)}{\sin(x-y)} \B|= \B|\f{\tan x+ \tan y}{\tan x - \tan y} \B| = \B|\f{s+1}{1-s}\B|,\quad
 \f{1}{\tan x } =  s \f{1}{\tan y}. 
 \eeq

We further obtain 
\[
I + II_1 = \f{1}{\pi^2} \int_0^{\pi/2}\int_0^{\pi/2} \f{f(x) f(y)}{ \tan^2 y }
\B( 4 s -(1+s^2) \log \B|\f{s+1}{1-s}\B| \B) dx dy. 
\]

Note that $f(x) f(y) \geq 0$ for $x , y\in[0, \f{\pi}{2}]$. The competition between $I, II_1$ is characterized by the interaction kernel $K(s) = 4 s -(1+s^2) \log \B|\f{s+1}{1-s}\B|, s\in [0,\infty)$. An important observation is that for large $s$ or small $s$, $K(s) \approx 2 s $. In particular, it is easy to obtain 
\[
K(s) = s^2 K(s^{-1}),  \quad K(s)  \geq s  - (1 + s^2) \log \B| \f{1 + s}{1- s} \B| \one_{a \leq s \leq a^{-1}}
\geq s  - C \log \B| \f{1 + s}{1- s} \B| \one_{a \leq s \leq a^{-1}}
\]
for some absolute constant $0 < a < 1$ and $C>0$. It follows 
\[
I + II_1 \geq 
\f{1}{\pi^2} \int_0^{\pi/2}\int_0^{\pi/2} \f{f(x) f(y)}{ \tan^2 y }
\B(  s - C \log \B|\f{s+1}{1-s}\B| \one_{a \leq s \leq a^{-1}} \B) dx dy.
\]

Repeating the above derivations, we get 
\beq\label{eq:EE_B3}
I + II_1 \geq \f{1}{4} (H f(0))^2 - C\int_0^{\pi/2} \f{ f(y)}{\tan^2 y}
\int_0^{\pi/2}  \log \B|\f{s+1}{1-s}\B| \one_{a \leq s \leq a^{-1}}  f(x) dx  dy.
%\teq \f{1}{4} (H f(0))^2 - J.
\eeq

Next, we show that 
\[
|J(y)| \les || f||_{L^{\infty}}, \quad  J(y) \teq \f{1}{\tan y} \int_0^{\pi/2}  \log \B|\f{s+1}{1-s}\B| \one_{a \leq s \leq a^{-1}}  f(x) dx .
\]

We consider a change of variable $ z= \tan x$. The restriction $s\in [a, a^{-1}]$ implies $z \in [a \tan y, a^{-1} \tan y]$. Using $dx  = \f{1}{1+ z^2} d z$ and \eqref{eq:log_s}, we yield 
\[
\bal
J(y) &\les \f{ || f||_{\inf}}{\tan y} \int_{a \tan y}^{a^{-1} \tan y} \log \B| \f{z + \tan y}{z - \tan y} \B| \f{1}{1 + z^2} dz
\les  \f{ || f||_{\inf}}{\tan y} \int_{a \tan y}^{a^{-1} \tan y} \log \B| \f{z + \tan y}{z - \tan y} \B| dz  \\
&\les  || f||_{L^{\inf}} \int_{a}^{a^{-1}} \log \B| \f{\tau + 1}{\tau - 1}  \B| d\tau \les  || f||_{L^{\inf}} ,
\eal
\]
where we have used another change of variable $z = \tau \tan y$ to obtain the third estimate.

%For $y > \f{\pi}{4}$, the estimate is trivial. For $ y \leq \f{\pi}{4}$, the restriction $s  = \f{\tan y}{\tan x} \in [a , a^{-1}]$ implies that $ \tan x \leq a^{-1}$. Since $a$ is absolute and $   \tan z \asymp z $ for $z \in [0, \arctan a^{-1}  ]$, we obtain 
% \[
% \f{y}{x} \leq C  \f{\tan y}{\tan x} \leq C a^{-1} = C_1, 
% \quad \f{y}{x} \geq C^{-1} \f{\tan y}{\tan x} \geq C^{-1} a = C_1^{-1} .
% \]
% Moreover, since $|x+y|, |x-y| \leq \f{3\pi}{4}$ and $  \sin z \asymp z$ for $z \in [0, 3\pi/4]$, from \eqref{eq:log_s}, we obtain 
% \[
% \B|\f{s+1}{s-1}\B| = \B|\f{\sin(x+y)}{\sin(x-y)}\B| \les \B| \f{x+y}{x-y}\B|.
% \]

% Using the above estimates and a change of variable $ x =  y z$, we yield
% \[
% |J(y)| \les \f{1}{y} || f||_{L^{\inf}}\int_{ y/x \in [C_1^{-1}, C_1] } \B( \log \B|\f{ y/x + 1}{y/x - 1} \B|+ 1 \B) d x 
% \les || f||_{L^{\inf}} \int_{C_1^{-1} }^{C_1} \B(\log \B| \f{ z + 1}{ z -1} \B| + 1 \B) d z \les || f||_{ L^{\inf}}.
% \]

Recall $f = \om( x + \f{\pi}{2})$ from \eqref{eq:rota_h_pi}. Plugging the above estimates in \eqref{eq:EE_B3}, we establish 
\[
I + II_1 \geq \f{1}{4} (Hf(0))^2 - C \int_0^{\pi/2} \f{  |f(y)| }{\tan y} d y || f||_{L^{\inf}}
=  \f{1}{4} (Hf(0))^2 - C | H f(0) |  \cdot || f||_{L^{\inf}},
\]
where we have used the facts that $f$ is odd and that $f$ has a fixed sign on $[0, \f{\pi}{2}]$ 
%$f \geq 0$ on $[0, \f{\pi}{2}]$ 
to obtain the equality.  We prove \eqref{eq:EE_B2}.

\subsubsection{ Estimate of $u_x(\f{\pi}{2})$ }

Combining the estimates \eqref{eq:EE_B0}-\eqref{eq:EE_B2}, we obtain 
\[
\f{d}{dt} u_x( \f{\pi}{2} ) 
\geq \f{1}{4} u_x^2(\f{\pi}{2})  -   C |u_x( \f{\pi}{2})| \cdot || \om||_{L^{\inf}}
- || \ |\cos x|^{1/2} \om ||_{\inf}  || \om||_{L^1}  \teq J.
\]

Recall the energies in \eqref{eg:A}. Using Lemma \ref{lem:extrap}, we derive
\[
|| \om||_{L^{\inf}} \les | u_x(\f{\pi}{2})|^{1/2}  (E(\om))^{1/2} + E(\om) ,\quad  || \ |\cos x|^{1/2} \om ||_{\inf}  || \om||_{L^1} \les E^2(\om).
\]

Using $\e-$Young's inequality, we yield 
\[
%\f{1}{4} u_x^2(\f{\pi}{2})  -   C |u_x( \f{\pi}{2})| \cdot || \om||_{L^{\inf}}
J \geq \f{1}{4} u_x^2(\f{\pi}{2}) - C |u_x(\f{\pi}{2})| \B( | u_x(\f{\pi}{2})|^{1/2}  (E(\om))^{1/2} + E(\om) \B)  - C E^2(\om)
%\f{1}{4} u_x^2(\f{\pi}{2})  -  |u_x( \f{\pi}{2})|^{3/2} A(\om)^{1/2}- ||u_x( \f{\pi}{2})| A(\om)^{1/2} ( u_x(0) + || \om||_{L^1} )
\geq \f{1}{8} u_x^2(\f{\pi}{2}) - C E^2(\om).
\]

Since $u_x(\f{\pi}{2}) \leq 0$, we derive 
\[
\f{d}{dt} |u_x(\f{\pi}{2})| \leq -\f{1}{8} u_x^2( \f{\pi}{2}) + C E^2(\om).
\]

Using the estimate \eqref{eq:EE_Aw}, we prove 
\beq\label{eq:EE_B4}
\bal
|u_x(t,\f{\pi}{2})|& \leq  |H\om_0(\f{\pi}{2})| + C\int_0^{t} E^2(\om(s)) ds\\
&\leq |H\om_0(\f{\pi}{2})| + C ( E(\om_0) + 2)^2 \exp(2 \exp(
C( || \om_0||_{L^1}, H \om_0)  \int_0^t \exp ( C U(s)) ds )).
\eal
\eeq

\subsubsection{ The blowup criterion}\label{sec:w_linf}

Using \eqref{eq:EE_Aw}, \eqref{eq:EE_B4} and Lemma \ref{lem:extrap}, we prove 
\beq\label{eq:w_linf}
|| \om||_{L^{\inf}}
\leq K_1(\om_0) \exp(2 \exp( K_1(\om_0) \int_0^t \exp( CU(s)) ds ) ),
\eeq
where $C$ is some absolute constant, and the constant $K_1(\om_0)$ depends on $H\om_0(0), H\om_0(\f{\pi}{2}), || \om_0||_{L^1} $ and $A(\om_0)$. Applying the BKM-type blowup criterion, we conclude the proof of Theorem \ref{thm:criterion}.

\section{ Stabilizing effect of the advection and several quadratic forms}\label{sec:regular}

%Positivity of a quadratic form}

%\section{Global well posedness for smooth initial data}\label{sec:GWP}

In order to apply Theorem \ref{thm:criterion} to establish the well-posedness result, we need to control $u_x(0)$. Yet, $u_x(0)$ itself does not enjoy a good estimate. Recall the ODE of $u_x(0)$ from \eqref{eq:EE_A3}.
% In fact, it is not difficult to obtain its evolution 
\[
 \f{d}{dt} u_x(0) = u^2_x(0)  + \f{2}{\pi} \int_0^{\pi/2} \f{ u \om }{ \sin^2 y} dy .
\]
Since $u_x(0)\geq 0$ for $\om \in X$, the quadratic nonlinearity $u^2_x(0)$ makes it very difficult to obtain a long time estimate on $u_x(0)$. Since $u_x(0) = \f{-2}{\pi} \int_0^{\pi/2} \om(y) \cot y dy $ can be viewed as a weighted integral of $\om$ with a singular weight near $0$, 
it motivates us to estimate other weighted integral that controls $u_x(0)$. For $\b \in (1,3)$, we introduce 
\beq\label{eq:def_Q}
Q(\b,t) \teq - \int_0^{\pi/2} \om(y, t) ( \cot y)^{\b} dy ,
\quad  B(\b, t) \teq 
  \int_{ 0 }^{ \f{\pi}{2}} (u_x \om - u \om_x) \cot^{\b} x dx , 
  %\  s(x,y) = \f{\tan y}{\tan x}.
\eeq
%For $\om \in C^{1,\al}$, 
For $\om \in X \cap H^1$, $Q(\b, t), B(\b,t)$ are well-defined if $\om$ vanishes near $x=0$ at order $|x|^{\g}$ with $\g > \b-1$.  For $ \om \in X$, since $\om \leq 0$ on $[0, \pi/2]$, we have $Q(\b, t) \geq 0$. %The functional $Q(\b, t)$ allows us to control the weighted integral of $\om$ near $0$. In Section \ref{sec:GWP}, we will combine it and $|| \om ||_{L^1}$ to further control $u_x(0)$. 
The boundedness of $Q(\b, t)$ implies that $\om$ cannot be too large near $0$, and it allows us to control the weighted integral of $\om$ near $0$. In Section \ref{sec:GWP}, we will combine it and $|| \om ||_{L^1}$ to further control $u_x(0)$. 

\begin{remark}
The special singular function $(\cot y)^{\b}$ and functional $Q(\b, t)$ are motivated by the homogeneous function $|y|^{-\b}$ and $ \int_{\R^+} \om  /  y^{\b} dy$, which were 
%have been 
used to analyze the gCLM model on the real line in the arXiv version of \cite{chen2019finite}.
\end{remark}

% \begin{remark}
% The special singular function $(\cot y)^{\b}$ and the functional $ Q(\b, t)$ are motivated by the homogeneous function $|y|^{-\b}$ and functional $ \int_{\R^+} \om  /  y^{\b} dy$, which have been used to analyze the gCLM model in the arXiv version of \cite{chen2019finite}.
% \end{remark}

Using \eqref{eq:DG}, we obtain the ODE of $Q(\b, t)$
%These two quantities are related via an ODE
\beq\label{eq:ODE_Q}
\f{d}{dt} Q(\b, t) = - B(\b, t).
\eeq

We should further estimate $B(\b, t)$. The key Lemma to prove Theorem \ref{thm:GWP} is the following. To simplify the notation, we will drop $``t"$ in some places. 

\begin{lem}\label{lem:comp}
%Suppose that $\om  \in C^{1,\al}$ is odd  with $\al \in (0, 1)$. 
Suppose that $\om  \in C^{\al}$ is odd with $\al \in (0, 1)$ and $\om(x)  x^{-1} \in L^{\inf}$. There exists some absolute constant $\b_0 \in (1,2)$, such that for $ \b \in [\b_0, 2) $, we have 
\beq\label{eq:comp_B}
  B(\b)  \geq  -  (2 - \b)  \B( u_x(0)  Q(\b) +  \f{1}{\pi}\iint_{ [0, \pi/2]^2} \om(x)\om(y) (\cot y)^{\b-1}
  \f{   s (s^{\b-1} - 1) }{ s^2 - 1}  dx dy \B),
\eeq
%for some absolute constant $C > 0$, 
where $ s(x, y) = \f{\cot x}{\cot y}$. %If in addition $\om_x(0) = 0$, 
If in addition $\om \in C^{1, \al}$ with $\al \in (0, 1)$ and $\om_x(0) = 0$, for $\b = 2$, we have 
\[
B(2) \geq 0 .
\]
\end{lem}

%Remark
Note that in Lemma \ref{lem:comp}, we do not impose the sign condition: %on $\om$., 
$\om \leq 0$ (or $\geq 0$) on $[0, \pi/2]$. Thus, it is likely that Lemma \ref{lem:comp} can be generalized to study \eqref{eq:DG} with a larger class of data. 

Lemma \ref{lem:comp} quantifies the stabilizing effect of the advection, and 
%which has been studied in \cite{chen2020slightly} for the gCLM model and \cite{lei2009stabilizing} for 3D incompressible flows. It also 
reflects that the advection is stronger or almost stronger than the vortex stretching for $\om$ vanishes at least linearly near $x=0$, which has been discussed heuristically in Section \ref{sec:intro_comp}. In fact, if $\om \in C^{1,\al}$ with $\om_x(0) = 0$, using \eqref{eq:ODE_Q} and Lemma \ref{lem:comp}, we obtain that $Q(2, t)$ is bounded uniformly in $t$ and thus $\om$ can not be too large near $0$. %For $\om_x(0) \neq 0$, 
%{\color{blue}In the general case of $\om x^{-1} \in L^{\inf}$,} 
In the general case, $\om$ can vanish only linearly near $x=0$. Then $Q(2, t)$ is not well-defined since $\om (\cot y)^2$ is not integrable. In this case, we apply \eqref{eq:comp_B}. Though $Q(\b, t)$ may not be bounded uniformly in $t$, the critical small factor $2-\b$ indicates that $Q(\b, t)$ cannot grow too fast.

%\subsection{The regularizing effect of advection}\label{sec:comp}
\subsection{ Symmetrization and derivation of the kernel}\label{sec:comp}

To prove Lemma \ref{lem:comp}, we first symmetrize the quadratic form $B(\b)$ and derive its associated interaction kernel. The symmetrization idea has been used in \cite{choi2014on} to analyze some quadratic forms in the Hou-Luo model. Denote 
\beq\label{eq:def_s}
s = \f{\tan y}{\tan x} = \f{\cot x}{\cot y}.
\eeq
% We consider general quadratic form and denote 
% \beq\label{eq:sym0}
% B(\b) \teq 
%   \int_{ 0 }^{ \f{\pi}{2}} (u_x \om - u \om_x) \cot^{\b} x dx , \quad s = \f{\tan y}{\tan x}.
% \eeq
% for $ \b \in (1,3)$, where $\om$ vanishes near $0$ at order $x^{\g}$ with $\g > \b-1$. For example, when $\b = 2$, $\om \in C^{1,\al}$ with $\om_x(0) = 0$ satisfies this assumption. 
Since $\om$ is odd, applying \eqref{eq:hil} and following the symmetrization argument in the arXiv version of \cite{chen2019finite}, we derive \eqref{eq:sym}
%the following 
in Appendix \ref{app:sym} if $\om$ vanishes near $x= 0$ at order $|x|^{\g}$ with $\g > \b-1$ 
\beq\label{eq:sym}
B(\b) =  \f{1}{\pi} \int_0^{\pi/2}\int_0^{\pi/2} \om(x) \om(y) P_{\b}(x, y) dx dy,
\eeq
where 
\beq\label{eq:kernel}
\bal
P_{\b}(x, y) &=  (\cot y)^{\b-1} \lt(  \f{\b}{2} (s^{\b-1}+1 )\log \B|  \f{s+1}{s-1}   \B| - (s^{\b-1}-1)\f{2s}{s^2-1} \rt) \\
& \quad +   (\cot y)^{\b+1} \lt(  \f{\b}{2} (s^{\b+1}+1 )\log \B|  \f{s+1}{s-1}   \B| - (s^{\b+1}-1)\f{2s}{s^2-1} \rt) \\
&\teq (\cot y)^{\b-1} P_{2, \b}(s) 
+ (\cot y)^{\b+1} P_{1, \b}(s) .
\eal
\eeq

Similar derivations and kernels were obtained in the arXiv version of \cite{chen2019finite}.
The logarithm terms come from the advection $u\om_x$ and are positive. Other terms $-(s^{\tau}-1) \f{2s}{s^2-1}, \tau = \b-1, \b+1$ are from the vortex stretching $u_x\om$ and are negative. Thus, the kernel $P_{\b}$ captures the competition between two terms. The main term in $P_{\b}$ is 
$(\cot y)^{\b+1} P_{1, \b}(s) $ since $(\cot y)^{\b + 1}$ is more singular. For $s$ near $1$, 
$P_{\b}(s)$ is positive due to the singularity in $\log \B| \f{1+s}{s-1}\B|$. It is not difficult to see that 
\beq\label{eq:kernel1}
\bal
%& P_1(s) = s^{\b+1} P_1(s^{-1}), \quad P_2(s) = s^{\b-1} P_2(s^{-1}),  \\
& \lim_{ s \to \inf} P_{1,\b}(s) = (\b - 2) s^{\b},  \quad
\lim_{s \to \inf} P_{2,\b}(s) = (\b - 2) s^{\b- 2}.
\eal
\eeq

Formally, as $\b$ increases, the kernel $P_{\b}(x,y)$ becomes more positive-definite. Recall the
ODE of $Q(\b)$ from \eqref{eq:def_Q}, \eqref{eq:ODE_Q}.  
%Since the quantity $B(\b)$ is associated with the energy $Q_{\b} = \int_0^{\pi/2} \om(y) \cot^{\b} y dy$, 
The higher vanishing order of $\om$ near $0$, the larger $\b$ we can choose with $Q(\b)$ being  well-defined, and it is more likely that $Q(\b, t)$ is decreasing and bounded uniformly in $t$.
%decreasing in $t$ and conserved. 
Therefore, the higher vanishing order of $\om$ near $0$ reflects the stronger effect of the advection, which potentially depletes the growing effect of the vortex stretching. 
%The above asymptotic behaviors of $K_i(s)$ 
The asymptotics \eqref{eq:kernel1} suggests that to obtain the positive definiteness of $P_{\b}$, $\b$ should be at least $2$. Indeed, such result is proved %by Chen-Hou-Huang 
in the arXiv version of \cite{chen2019finite} for $\b = 2.2$ under the sign condition $\om \in X$ \eqref{eq:X} by showing that $P_{i,\b}(s) \geq 0$ pointwisely. However, the method in \cite{chen2019finite} can not be applied to the critical case  $\b = 2$ since numerical result shows that $P_{1,2}(s) <0 $ for $s \leq 0.5$ or $s \geq 2$. 

For $\b<2$, it is not expected that $P_{\b}$ is positive-definite and the gap is of order $2-\b$ quantified in Lemma \ref{lem:comp}. We study the modified kernel and its associated quadratic form 
\beq\label{eq:kernel2}
\bal
&K_{1, \b}(s)  = P_{1, \b}(s) + (2-\b) (s + s^{\b}) , \quad 
K_{2, \b}(s) = P_{2, \b}(s) + (2- \b) \f{ (s^{\b-1} - 1) s}{s^2 - 1} , \\
&K_{\b}  = (\cot y)^{\b + 1} K_{1, \b} + (\cot y)^{\b-1} K_{2, \b} , 
\quad \td B(\b) = \int_0^{\pi/2} \int_0^{\pi/2} \om(x) \om(y) K_{\b}(x, y) dx dy, \\
% & B_1(\b) = \int_0^{\pi/2} \int_0^{\pi/2} \om(x) \om(y) 
%  (\cot y)^{\b + 1} K_{1, \b}(x, y) dx dy,
%  \quad B_2(\b) = \int_0^{\pi/2} \int_0^{\pi/2} \om(x) \om(y)  (\cot y)^{\b-1} K_{2, \b}  dx dy,
\eal
\eeq
where $P_i, s$ are defined in \eqref{eq:kernel}, \eqref{eq:def_s}. Using \eqref{eq:sym}, \eqref{eq:kernel}, \eqref{eq:kernel2}, and the following identities
\beq\label{eq:equiv1}
\bal
(s + s^{\b}) (\cot y)^{\b+1} &= \cot x (\cot y)^{\b} + (\cot x )^{\b} \cot y,  \\
%s (\cot y)^{\b+1} = \cot x (\cot y)^{\b}, \quad s^{\b} (\cot y)^{\b+1} = \cot y (\cot x)^{\b}, 
%\quad u_x(0) = - \f{2}{\pi} \int_0^{\pi/2} \om(y) \cot y dy, \\
\f{1}{\pi}\int_0^{ \f{\pi}{2} } \int_0^{ \f{\pi}{2} }
\om(x) \om(y) (s + s^{\b}) (\cot y)^{\b+1} dx dy 
&=  \f{2}{\pi}\int_0^{ \f{\pi}{2}  } \om \cot y  dy 
 \int_0^{ \f{\pi}{2}} \om (\cot y)^{\b} dy =  u_x(0) Q(\b),
\eal
\eeq
we derive 
\beq\label{eq:equiv2}
\bal
\f{\td B(\b)}{\pi} & = \f{1}{\pi} \int_0^{ \f{\pi}{2} } \int_0^{ \f{\pi}{2} } \om(x) \om(y)  \B\{ P_{\b}(x, y)
+ (2-\b) \B(  (s + s^{\b}) (\cot y)^{\b+1} 
+   \f{ (s^{\b-1} - 1) s}{s^2 - 1}  (\cot y)^{\b-1} \B) \B\}  dx dy  \\
& = B(\b) + (2-\b) \B( u_x(0) Q(\b)
+ \f{1}{\pi} \int_0^{ \f{\pi}{2}  } \int_0^{ \f{\pi}{2}} \om(x) \om(y)     \f{ (s^{\b-1} - 1) s}{s^2 - 1}  (\cot y)^{\b-1}   dx dy  \B) . \\
\eal
\eeq
Hence, Lemma \ref{lem:comp} is equivalent to $ \td B(\b) \geq 0$, or the positive definiteness of $K_{ \b}$ for $\b \in [\b_0, 2]$.

Our key observation is that $s(x, y) = \f{\cot x}{  \cot y} $ can be written as $p( u -v)$, for some function $p$ and variables $u,v$, and $K_{\b}$ can be written as a convolution kernel after a change of variable. This allows us to follow the idea in Bochner's theorem for a positive-definite function to leverage the positive part of $K_{\b}(s)$ and establish that $K_{\b}$ is positive-definite.
%Inspired by the Bochner's theorem for positive-definite function, we reformula $K_{\b}$ to a convolution kernel, which further enables us to leverage the positive part of $K_1(s)$ and show that $K_{\b}$ is positive definite.

%To simplify the notation, we drop the subscript $\b$ in $K_{\b}, K_{1,\b}, K_{2,\b}$. 
In the following derivation, we restrict $\b$ to $ \b \in [1.9, 2]$. The reader can think of the special case $\b = 2$, since we will choose $\b$ to be sufficiently close to $2$.

%In the following derivation, we focus on $\b = 2$ and drop the subscript $\b$ in $K_{\b}, K_{1,\b}, K_{2,\b}$. We prove Lemma \ref{lem:comp} for smooth function $\om \in C^{\inf}$, and the general case $\om \in C^{1,\al} \cap X$ can be obtained by approximation.

%\subsection{ Derivation of the concolution kernel}

%\subsubsection{ Positivity of $K_1$}
\subsubsection{ Reformulation of $K_{1,\b}$}  We introduce 
\beq\label{eq:tK1}
\bal
F_1(x) & \teq \om(x) (\cot x)^{ \f{ \b+1}{2}},  \\ 
\td K_{1,\b}(s)  &\teq s^{- \f{\b+1}{2}} K_{1,\b}
=  \f{\b}{2}  ( s^{ \f{\b+1}{2}}  + s^{- \f{\b+1}{2}} ) \log \B| \f{s+1}{s-1} \B| - \f{  s^{ \f{\b+1}{2}}  - s^{- \f{\b+1}{2}} }{s^2 -1} 2s \\
&\quad + (2- \b) ( s^{ \f{\b-1}{2}} + s^{ \f{1-\b}{2}} ).
%\quad  F_2(x) = \f{F_1( \arctan x )}{1 + x^2} .
\eal
\eeq

Recall $ s \cot y  = \cot x$ from \eqref{eq:def_s}. Using $s^{ \f{\b+1}{2} } ( \cot y)^{\b+1} = ( \cot y \cot x )^{ \f{\b+1}{2}}$, we derive
\[
 ( \cot y)^{\b+1} K_{1,\b}(s) = (\cot y)^{\b+1} s^{ \f{\b+1}{2} } s^{- \f{\b+1}{2}} K_{1,\b}(s)=
  ( \cot y \cot x )^{ \f{\b+1}{2} } \td K_{1,\b}(s).
\]

Hence, we can rewrite the quadratic form associated with $K_{1,\b}$ in $\td B(\b)$ \eqref{eq:kernel2} as follows 
\beq\label{eq:form_B1}
B_1(\b) \teq \int_0^{\pi/2} \int_0^{\pi/2} \om(x) \om(y) (\cot y)^{\b+1} K_{1,\b}(s) dx dy
= \int_0^{\pi/2} \int_0^{\pi/2} F_1(x) F_1(y) \td K_{1,\b}(s) d x dy .
\eeq

For $x , y \in [0, \pi/2]$, we consider a change of variable 
\beq\label{eq:kernelW}
x = \arctan e^r, \quad y = \arctan e^t, \quad F_2( z) = \f{e^z F_1(\arctan e^z)}{1 + e^{2z}},
\quad W_{1,\b} (z) = \td K_{1,\b} (e^z).
\eeq
The variables $r = \log \tan x$  maps $(0, \f{\pi}{2})$ to $\R$. Using $ \f{dx}{dr} = \f{e^r}{1 + e^{2r}}$ and $ s = \f{\tan y}{\tan x} = e^{t-r}$, we obtain 
\[
B_1 = \int_{\R} \int_{\R} \f{F_1 (\arctan e^r) F_1(\arctan e^t)}{ (1 + e^{2r})(1 + e^{2t})}
\td K_{1,\b}( e^{t-r}) e^r e^t d t dr  = \int_{\R}\int_{\R} F_2(r) F_2(t)  W_{1,\b}(t- r) d t dr .
\]

% Consider a change of variable 
% $ x = \arctan \td x, y = \arctan \td y$. We get $s = \f{\tan y}{\tan x} = \f{ \td y}{\td x}$ and 
% \[
% B_1 = \int_0^{\inf}\int_0^{\inf} \f{F_1( \arctan \td x) F_1(\arctan \td y)}{ (1 + \td x^2)( 1 + \td y^2)} \td K_1( \f{\td y}{\td x}) d \td x d \td y.
% \]

% We perform a further change of variable 
% \beq\label{eq:kernelW}
% \td x = e^r, \quad \td y = e^s, \quad F_2(r) = \f{e^r F_1( \arctan e^r)}{ 1 + e^{2r}} , \quad 
% W( z) = \td K_1( e^z),
% \eeq
% and derive 
% \[
% \bal
% B_1 &= \int_{\R} \int_{\R} \f{F_1 (\arctan e^r) F_1(\arctan e^s)}{ (1 + e^{2r})(1 + e^{2s})}
% \td K_1( e^{s-r}) e^s e^r d s dr  = \int_{\R}\int_{\R} F_2(r) F_2(s)  W(s- r) d s dr .
% \eal
% \]
Recall $F_1$ in \eqref{eq:tK1}. Since $ \cot (\arctan e^r) = e^{-r}$, we can rewrite $F_2$ in terms of $\om$
\[
F_2(r) = \f{e^r \om( \arctan e^r) (\cot( \arctan e^r) )^{ \f{\b+1}{2} } }{ 1 + e^{2r}}
= \f{e^{-\f{\b-1}{2} r }  \om( \arctan e^r) } { 1 + e^{2r} }.
\]

Next, we discuss the integrability of $W_{1,\b}$ and $F_2$. %Since $\om \in C^{1,\al}(S^1)$ is odd, $\arctan x \les x$
Since $\om(x)x^{-1} \in L^{\inf}$, $\arctan x \les \min(x, 1)$ and $\b \in [1.9, 2]$, we get 
\[
|F_2(r)| \les e^{- \f{\b-1}{2} r} \min(1, e^r) \les \min( e^{r/4}, e^{-r/4} ).
\]
% and $F_2$ is in the Schwarz class. 

Recall the  definition of $\td K_{1,\b}$ in \eqref{eq:tK1}. Clearly, $|\td K_{1,\b}(s)|^p, |W_{1,\b}(z)|^p$ are locally integrable for any $p > 0$. Using \eqref{eq:kernel1}, $\B|\log \B| \f{s+1}{s-1} \B| - \f{2}{s}\B| \les s^{-3}$ for $s>2$ and a direct estimate, we obtain 
\[
\td K_{1,\b}(s) = \td K_{1,\b}(s^{-1}) , \quad   |\td K_{1,\b}(s) | \les  s^{- \f{\b-1}{2}} 
\les s^{-1/4} \mathrm{ \ for \ } s > 2.
\]

Note that for large $s$, the leading exponents $s^{ \f{\b-1}{2}}$ appeared in each term of $\td K_{1,\b}$ are canceled. As a result, we yield 
\beq\label{eq:decay_W1}
W_{1,\b}(z) = \td K_{1,\b}(e^z) = \td K_{1,\b}(e^{-z}) = W_{1,\b}(-z), \quad | W_{1,\b}(z) | \les e^{-|z|/4}  \mathrm{ \ for \ }|z| > 1.
\eeq
%\eeq

Denote by $\hat f = \int_{\R} \exp(-  i x \xi) f(x ) dx $ the Fourier transform of $f$. Using the Plancherel theorem, for some absolute constant $C_1>0$, we get 
\beq\label{eq:pos_B1}
B_1(\b) = C_1  \int_{\R} |\hat F_2(\xi) |^2 \wh W_{1,\b}(\xi) d\xi.
\eeq
%To obtain $B_{1,\b} \geq 0$, it suffices to prove $\hat W_1(\xi) \geq 0$ for any $\xi$. We will discuss it in Section \ref{sec:ker_pos}.

 %where $\hat f(\xi) = \int_{\R} \exp(-  i x \xi) f(x ) dx .$denotes  the Fourier transform of $f$.
%$\hat f$ is the Fourier transform of $f$: 
%defined by 
% \[
% \hat f(\xi) = \int_{\R} \exp(-  i x \xi) f(x ) dx .
% \]

\subsubsection{Reformulation of $K_{2,\b}$}

Similarly, we reformulate the kernel $K_{2,\b}$ and its associated quadratic form in $\td B(\b)$ in \eqref{eq:kernel2} as follows 
\beq\label{eq:pos_B2}
\bal
B_2(\b)  &\teq \int_0^{\pi/2} \int_0^{\pi/2} \om(x) \om(y) (\cot y)^{\b-1} K_{2,\b}(s) d x dy \\
 &= \int_{\R} \int_{\R} F_4(r) F_4(s) W_{2,\b}(t-r) dt dr 
= C_1 \int_{\R} \int_{\R} |\hat F_4(\xi)|^2 \wh W_{2,\b}(\xi) d\xi
\eal
\eeq
for some absolute constant $C_1> 0$, where 
\beq\label{eq:kernel_W2}
\bal
F_4(r) &= \f{ e^{ \f{3-\b}{2} r } \om( \arctan e^r) }{1 + e^{2r}}, \quad 
W_{2,\b}(z) = \td K_{2,\b}( e^z),  \\
\td K_{2,\b}(s)  &= \f{\b}{2} \B(  ( s^{ \f{\b-1}{2}} + s^{ - \f{\b-1}{2}  } ) \log \B| \f{s+1}{s-1} \B|
 -  ( s^{ \f{\b-1}{2} } - s^{ - \f{\b-1}{2}} ) \f{2s}{s^2-1} \B).
\eal
\eeq
The variable $F_4$ corresponds to $\om(x) (\cot x)^{ \f{\b-1}{2}}$ after a change of variable. For $\td K_{2,\b}, W_{2,\b}, F_4$ with $s>2, |z|>1$, we have 
\[
\bal
|F_4(r) | &\les \min( e^{ \f{r}{4}}, e^{-\f{r}{4}}), \quad  && W_{2,\b}(z)  = W_{2,\b}(-z), \quad \td K_{2,\b}(s) = \td K_{2,\b}(s^{-1}), \\
 |W_{2,\b}(z)|  &\les e^{-|z|/4} , \quad 
&& |\td K_{2,\b}(s)| \les s^{- 1/4 }.
\eal
\]

\subsection{ Positivity of $W_{j, \b}$}\label{sec:ker_pos}

Recall the formulas of $B_j(\b)$ \eqref{eq:pos_B1}, \eqref{eq:pos_B2}. To show that $B_j(\b) \geq 0$, it suffices to prove $\wh W_{j,\b}(\xi) \geq 0$ for any $\xi$. Since $W_{j,\b}$ is even, it is equivalent to show that
\beq\label{eq:def_G}
G_{j,\b}(\xi) \teq \f{1}{2}  \hat W_{j,\b}(\xi) = \f{1}{2} \int_{\R} W_{j,\b}(x) e^{-  i x \xi} dx
%= \f{1}{2}\int_{\R} W_{j,\b}(x) \cos(  x \xi) dx 
=  \int_{\R_+} W_{j,\b}(x) \cos (x \xi) dx \geq 0
\eeq
for any $\xi$. Since $G_{j,\b}(\xi), \wh W_{j,\b}(\xi)$ are even, we can further restrict to $\xi \geq 0$. We first study the positivity of $G_{1,\b}$, which is much more difficult than that of $G_{2,\b}$.

%As we have discussed in the paragraph below \eqref{eq:kernel} that the kernel $P_{2,\b}$ is less singular than $P_{1,\b}$. This fact also reflects that it is easier to establish $ G_{2,\b}(\xi)\geq 0$ than that of $G_{1,\b}(\xi)$. We first consider $G_{1,\b}$.

\subsubsection{ Positivity of $G_{1,\b}$}\label{sec:pos_W1}

%Thanks to the fact that $\b$ can be sufficiently close to $2$,
Since we are interested in the case where $\b$ close to $2$, using continuity, 
%Thanks to the smallness of $2-\b$, 
%Since $\b$ can be sufficiently close to $2$,
%using continuity of $W_{1,\b}$ in $\b$, 
%we can essentially 
we can essentially reduce proving $G_{1,\b} \geq 0$ to  the special case $\b = 2$. %We have the following result.

\begin{lem}\label{lem:pos_W}
Let $W = W_{1,2}, G = G_{1,2 }$. Suppose that there exists $x_0 > 0, M > 0$, such that %$G(\xi) > 0$ for $\xi \in [0, M]$, $W^{\prime}(x) > 0$ for $x \in [0, x_0]$ and 
\begin{align}
& G(\xi) > 0 ,  \quad \xi \in [0, M], \label{eq:ver10}  \\
    & W^{\prime \prime}(x) > 0 , \quad x \in [0, x_0],  \label{eq:ver2} \\ 
   &     - W^{\prime}(x_0) - \f{1}{M} \B(  |W^{\prime \prime}(x_0)|
+ \int_{x_0}^{\inf} |W^{\prime \prime \prime}(x) |dx  \B) > 0.	 \label{eq:ver3}
%\eal
\end{align}
Then there exists $\b_0 \in (1,2)$, such that for any $ \b \in [\b_0, 2]$ and $\xi$, we have $G_{1,\b}(\xi) \geq 0 $.

\end{lem}

Using continuity of $W_{1,\b}$ in $\b$ and the smallness of $2-\b$, we will show that %the conditions 
\eqref{eq:ver10}-\eqref{eq:ver3} hold for $ W_{1,\b},  G_{1,\b}$. The proof 
of this part is standard and is deferred to Appendix \ref{app:pos_W}. 

Next, we prove that \eqref{eq:ver2}, \eqref{eq:ver3} implies $G_{1,2}(\xi) \geq 0$ on $[M, \inf]$, which along with \eqref{eq:ver10} prove $G_{1,2}(\xi)\geq 0$. The same argument applies to $G_{1,\b}$. We simplify $W_{1,2}, G_{1,2}$ defined in \eqref{eq:tK1}, \eqref{eq:kernelW}, \eqref{eq:def_G} as $W, G$.

\

\paragraph{\bf{Large} $\xi$}

We will choose $M$ to be relatively large. This allows us to
%For large $\xi \geq M$, 
 exploit the oscillation in the integral $G(\xi)$ \eqref{eq:def_G} for $\xi \geq M$.
From the definition of $W(x)$ in \eqref{eq:tK1} and \eqref{eq:kernelW}, we know that $W(x)$ is smooth away from $x=0$ and $W(x)$ is singular of order $\log|x|$ near $ x= 0$. Using integration by parts twice, we yield 
\beq\label{eq:IBP_G}
\bal
G(\xi) &= \xi^{-1} \int_{\R_+} W(x)  \pa_x \sin (x \xi) dx
= - \xi^{-1} \int_{\R_+} W^{\prime}(x)  \sin (x\xi) dx  \\
&= - \xi^{-2} \int_{\R_+} W^{\prime}(x) \pa_{x} ( 1- \cos (x\xi) ) dx  
 =  \xi^{-2} \int_{\R_+} W^{\prime \prime}(x)  ( 1- \cos (x\xi)  )dx ,
 \\
\eal
\eeq
where the boundary term vanishes due to $W(x) \sin ( x\xi) = O(x \log x) $ and $W^{\prime}(x) (1 -\cos x \xi) = O( \f{1}{x} x^2) = O(x)$ and the fast decay \eqref{eq:decay_W1}. The advantage of the above formula is that we obtain a nonnegative coefficient $1 - \cos ( x \xi)$. For some $x_0 > 0$, we define 
\beq\label{eq:decomp_G}
G_1(\xi) \teq \int_0^{x_0} W^{\prime \prime}(x)  ( 1- \cos (x\xi)  )dx ,
\quad G_2(\xi ) \teq \int_{x_0}^{\inf} W^{\prime \prime}(x)  ( 1- \cos (x\xi)  )dx ,
\eeq

It suffices to verify $G_1(\xi) \geq 0$ and $G_2(\xi) \geq 0$. Thanks to \eqref{eq:ver2} and $1 - \cos(\xi x) \geq 0$, we obtain $G_1(\xi) \geq 0$.
% For $G_1(\xi)$, we verify 
% \beq\label{eq:ver2}
% W^{\prime \prime}(x) \geq 0  , \quad x \in [0, x_0],
% \eeq
% which implies $G_1(\xi) \geq 0$ since $1 - \cos(\xi x) \geq 0$. 
For $G_2(\xi)$, the main term is associated with $1$ since $\cos (x\xi)$ oscillates. 
%Formally, we have $1 - \cos (x \xi \approx 1$
In fact, using integration by parts again, we yield 
\[
\bal
G_2(\xi) &= - W^{\prime}(x_0) - \int_{x_0}^{\inf} W^{\prime \prime}(x)   \cos (x\xi)  dx 
=  - W^{\prime}(x_0) - \xi^{-1} \int_{x_0}^{\inf} W^{\prime \prime}(x)  \pa_x \sin (x\xi)  dx  \\
 &= - W^{\prime}(x_0)  +  W^{\prime \prime}(x_0) \f{\sin(x_0 \xi)}{\xi}
+  \int_{x_0}^{\inf} W^{\prime \prime \prime}(x)  \f{ \sin (x\xi)}{\xi}  dx  \\
&\geq - W^{\prime}(x_0) - \f{1}{M} \B(  |W^{\prime \prime}(x_0)|
+ \int_{x_0}^{\inf} |W^{\prime \prime \prime}(x) |dx  \B),
 \\ 
\eal
\]
where we have used $\xi \geq M$ in the last inequality. We choose $x_0 > 0$ and decompose the integral into two domains $ x \leq x_0$ and $x > x_0$ in \eqref{eq:decomp_G} since $W^{\prime \prime \prime}$ in the above derivation is not integrable near $x=0$. Using the assumption \eqref{eq:ver3}, we obtain $G_2(\xi) \geq 0$. 

%Therefore, we prove that \eqref{eq:ver2}-\eqref{eq:ver3} implies 

% To show that $G_2(\xi) \geq 0$, we verify the above lower bound is positive 
% \beq\label{eq:ver3}
% - W^{\prime}(x_0) - \f{1}{M} \B(  |W^{\prime \prime}(x_0)|
% + \int_{x_0}^{\inf} |W^{\prime \prime \prime}(x) |dx  \B) > 0.
% \eeq

\subsubsection{ Verification of the conditions in Lemma \ref{lem:pos_W} }\label{sec:ver}
%Since $W_{1,2}$ is given explicitly in \eqref{eq:tK1}, \eqref{eq:kernelW}, 
We discuss how to verify \eqref{eq:ver10}-\eqref{eq:ver3} below. 

%\paragraph{\bf{Small} $\xi \in [0,M]$ }
Firstly, $G(\xi)$ is smooth in $\xi$ and the Lipschitz constant satisfies 
\beq\label{eq:lip}
|\pa_{\xi}  G| \leq \int_{\R_+} | W(x) | x dx \teq b_1.
\eeq
%Since $W = W_{1,2}$ is locally integrable and enjoys the estimate \eqref{eq:decay_W1}, the integral is bounded.
% \[
% |\pa_{\xi}  G| \leq \int_{\R_+} | W(x) | x dx = \int_{\R_+} |\td K_{1,2}(e^x)| x dx 
% = \int_1^{\infty } |\td K_{1,2}(s)| \f{ \log s}{s} ds \teq b_1,
% \]
% where we have used $W(x) = K_{1,2}(e^x)$ and a change of variable $s = e^x$. 
The constant $b_1$ will be estimated rigorously. For small $\xi \in [0, M]$, we compute a lower bound of the integral $G(\xi)$ rigorously 
%compute the integral $G(\xi)$ numerically with rigorous error control 
for the discrete points $\xi = ih, i=0,1,2.., n$, $M = nh$, and verify $G( ih) >0$. %See the discussion in Appendix \ref{app:ver} for rigorous verification. 
For $\xi \in [ih, (i+1) h]$, we use 
\beq\label{eq:ver1}
G(\xi ) \geq \min( G( ih), G( (i+1) h) ) - \f{h}{2} b_1  >0
\eeq
and verify the second inequality to obtain $G(\xi) > 0$. This enables us to establish \eqref{eq:ver10}.

For \eqref{eq:ver2} and \eqref{eq:ver3}, let us first motivate why they 
hold for some $x_0$ and $M$. Using \eqref{eq:tK1} and \eqref{eq:kernelW}, we obtain the asymptotic behavior of $W(x)$ for $x$ near $0$
\[
W(x) \approx  - C \log|e^x - 1| \approx  - C \log x, \quad W^{\prime}(x) \approx - \f{C}{x} < 0,
\quad W^{\prime \prime}(x) \approx \f{C}{x^2 } > 0,
\]
for some constant $C>0$. See also \eqref{eq:conv_W13} for a detailed derivation. Note that $W^{\prime \prime \prime}$ is integrable away from $0$. Thus, \eqref{eq:ver2}, \eqref{eq:ver3} hold for small $x_0$ and large $M$. 

In practice, we choose $x_0 = \log \f{5}{3} $ and $M = 20$ in Lemma \ref{lem:pos_W}. Note that $W_{1, 2}$ is an explicit function. We prove \eqref{eq:ver2} for $x_0 = \log \f{5}{3} $ in Appendix \ref{app:pos_W}. We discuss how to compute the integrals in \eqref{eq:ver1} and \eqref{eq:ver3} and verify these conditions, which are independent of $\xi$, %,  using interval arithmetic \cite{moore2009introduction,rump2010verification} 
rigorously in Appendix \ref{app:ver}. This allows us to establish the conditions in Lemma \ref{lem:pos_W}. The rigorous lower bound of $G(\xi)$ for $\xi = i h \in [0, M]$ is plotted in Figure \ref{fig:Gxi}, and $G(\xi)$ is strictly positive. 
\begin{figure}[t]
\centering
    	\includegraphics[width=0.6\textwidth]{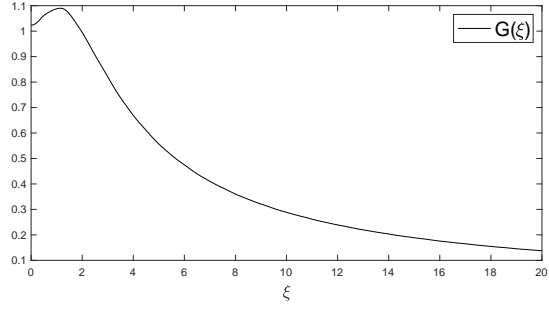}
    \caption{Rigorous lower bound of $G(\xi)$ for $\xi =  i h, h=0.05, 0\leq i \leq 400$, $G(ih) > 0$. } \label{fig:Gxi}
    \vspace{-0.1in}
\end{figure}

%In practice, we choose $x_0 = \log \f{5}{3} $ and $M = 20$ in Lemma \ref{lem:pos_W}. Note that $W_{1, 2}$ is an explicit function. We prove \eqref{eq:ver2} for $x_0 = \log \f{5}{3}$ in Appendix \ref{app:pos_W}.
% We will choose $h$ and verify the conditions \eqref{eq:ver1} and \eqref{eq:ver3}, which are independent of $\xi$, using interval arithmetic \cite{moore2009introduction,rump2010verification}. This allows us to establish the conditions in Lemma \ref{lem:pos_W}. We refer to Appendix \ref{app:ver} for the discussion on rigorous verification. 

\subsubsection{ Positivity of $G_{2,\b}$}

Recall $W_{2,\b}, G_{2,\b}$ defined in  \eqref{eq:kernel_W2} and \eqref{eq:def_G}.
%and its associtaed Fourier transform $G_{2,\b}$ in \eqref{eq:def_G}. 
For $G_{2,\b}$, it is easier to establish its positivity than that of $G_{1,\b}$. 
From the argument in Section \ref{sec:pos_W1} and \eqref{eq:IBP_G}, a sufficient condition for $G_{2,\b}(\xi) \geq 0$ %with any $\xi$ 
is the convexity of $W_{2,\b}$. 
We have the following result.

% Recall $W_{2,\b}$ defined in  \eqref{eq:kernel_W2} and its associtaed Fourier transform 
% $G_{2,\b}$ in \eqref{eq:def_G}. For $G_{2,\b}$, it is easier to establish its positivity than that of $G_{2,\b}$. 
% From the argument in the previous Section and \eqref{eq:IBP_G}, we know that a sufficient condition for $G_{2,\b}(\xi) \geq 0$ for all $\xi$ is the convexity of $W_{2,\b}$. 
% We have the following result.

\begin{lem}\label{lem:pos_W2}
For any $\b \in ( 1, 2] $, we have $W_{2,\b}^{\prime \prime }(x) \geq 0$ for $x \geq 0$.
% and thus 
As a result, $G_{2,\b}(\xi) \geq 0$ for any $\xi$ and $\b \in (1, 2]$.
\end{lem}

The proof is based on estimating $W_{2,\b}^{\prime \prime }$ directly using its explicit formula and elementary inequalities, which is not difficult and deferred to Appendix \ref{app:pos_W}.

\subsubsection{ Proof of Lemma \ref{lem:comp}}
Combining Lemma \ref{lem:pos_W} and Lemma \ref{lem:pos_W2}, we establish that there exists $\b_0 \in (1,2)$, such that for $\b \in [\b_0, 2]$ and any $\xi$, $ \wh W_{j,\b}(\xi) = 2 G_{1,\b}(\xi) \geq 0 , j =1,2$. 
% From \eqref{eq:pos_B1}, \eqref{eq:pos_B2}, we prove $B_j(\b) \geq 0$, which
% along with $\td B(\b) = B_1(\b ) + B_2(\b)$ and \eqref{eq:kernel2} imply the desired result in Lemma \ref{lem:comp}.
From \eqref{eq:pos_B1}, \eqref{eq:pos_B2}, we prove $B_j(\b) \geq 0$. Recall the definitions of $\td B(\b), B_1(\b), B_2(\b)$ from \eqref{eq:kernel2}, \eqref{eq:form_B1}, and \eqref{eq:pos_B2}. We obtain $\td B(\b) = B_1(\b) + B_2(\b) \geq 0$. 

Note that to obtain the equivalence between  the forms of $B(\b)$ in  \eqref{eq:def_Q} and  \eqref{eq:sym}, we require that $\om$ vanishes near $x=0$ at order $|x|^{\g}$ with $\g > \b- 1$. %This requirement is satisfied for $\om, \b$ stated in Lemma \ref{lem:comp}. 
Using the relation \eqref{eq:equiv2} between $\td B(\b)$ and $B(\b)$, we prove \eqref{eq:comp_B} in Lemma \ref{lem:comp} for $\b \in [\b_0, 2)$ and odd $\om \in C^{\al}$ with $\om x^{-1} \in L^{\inf}$. If in addition $\om \in C^{1,\al}$ and $\om_x(0) = 0$, we obtain that the vanishing order of $\om$ near $x=0$ is larger than $1$ and choose $\b = 2$ to establish $B(2) = \td B(2) \geq 0$. We conclude the proof of Lemma \ref{lem:comp}.
% Using the relations in \eqref{eq:sym}, \eqref{eq:kernel}, \eqref{eq:kernel2}, and \eqref{eq:equiv1}, we have 
% \[
% \td B(\b) = B(\b) + (2-\b) 
% \]
%  which
% along with $\td B(\b) = B_1(\b ) + B_2(\b)$ and \eqref{eq:kernel2} imply the desired result in Lemma \ref{lem:comp}.

\section{Global well-posedness}\label{sec:GWP}
%\section{Global well-posedness for $C^{1,\al}$ data}\label{sec:GWP}

In this Section, we use the crucial Lemma \ref{lem:comp} to control $u_x(0,t)$ and then establish the global well-posedness result in Theorem \ref{thm:GWP} using the one-point blowup criterion in Theorem \ref{thm:criterion}. We impose the assumptions $\om_0 \in H^1 \cap X, \om_0(x) x^{-1} \in L^{\inf}$, and $A(\om_0) <+\inf$ stated in Theorem \ref{thm:GWP}.
%Throughout this Section, we impose the assumptions 
%{\color{blue}$\om \in H^1 \cap X$ with $\om_0 / x \in L^{\inf}$ stated in Theorem \ref{thm:GWP}.}
%$\om \in C^{1,\al} \cap X$ stated in Theorem \ref{thm:GWP}.

Recall $Q(\b)$ defined in \eqref{eq:def_Q}. To apply Theorem \ref{thm:criterion}, from H\"older's inequality
\beq\label{eq:hol_ux0}
|u_x(0)| \les \int_0^{\pi/2} |\om(y)| \cot y dy \les Q(\b)^{ 1 / \b} || \om||_{L^1}^{1 - 1 / \b},
\eeq
we only need to control $|| \om||_{L^1}$ and $ Q(\b)$. In 
\eqref{eq:EE_L12},\eqref{eq:EE_L13}, we derive the evolution of $|| \om||_{L^1}$
\beq\label{eq:GWP_L1}
\f{d}{dt} -\int_0^{\pi/2} \om(x) dx
= \f{2}{\pi} \int_0^{\pi/2}\int_0^{\pi/2} \om(x) \om(y) \cot(x + y) dx dy .
\eeq

Recall the discussion of the interaction on the right hand side in Section \ref{sec:intro_diff}. 
For $x+ y \geq \f{\pi}{2}$, the interaction has a negative sign and it will play a crucial role as a damping term.

%For $ \om \in X $, we have $ \om \leq 0$. For $x + y \leq \f{\pi}{2}$, the interaction on the right hand sign has a positive sign due to $\cot(x+y) \geq 0$, which leads to the growth of $|| \om||_{L^1}$. On the other hand, for $x+ y \geq \f{\pi}{2}$, the interaction has a negative sign, which contributes to the decrease of $|| \om||_{L^1}$. As we will see later, the latter plays a crucial role as a damping term. 

\subsection{ Special Case: $\om_0 \in C^{1,\al}, \om_{0,x}(0)=0$}

For initial data $\om_0$ with $\om_{0,x}(0)=0$, $\om_x(0, t) =0$ is preserved and $Q(2,t) = -\int_0^{\pi/2} \om(y) \cot^2 y dy$ is well-defined. Using \eqref{eq:ODE_Q} and Lemma \ref{lem:comp}, we obtain 
\[
\f{d}{dt} Q(2, t) = - B(2,t) \leq 0.
\]

Since $\om \leq 0 $ on $ (0, \f{\pi}{2})$, we derive $Q(2,t) \geq 0$ and
\[
%A(t_0) \leq A(t) \leq 0, \quad  
\int_0^{\pi/2} | \om | \cot^2 y dy 
 =   Q(2, t) \leq   Q( 2, 0 ) = \int_0^{\pi/2} | \om_0 | \cot^2 y dy < + \infty.
\]

Next, we estimate $||\om||_{L^1}$. We first establish an estimate similar to \eqref{eq:EE_L13} 
\beq\label{eq:EE_comp1}
 - \int_0^{x}  \om(y) \cot(x+y ) dy \leq - \int_0^{\pi/2} \om(y) \cot^2 y dy = Q(2,t)
\eeq
for $x \in [0, \f{\pi}{2}]$. Since $\cot z \leq 0$ for $z \geq \f{\pi}{2}$, $\cot y \leq 1$ on $[0, \f{\pi}{4}]$, and $\cot y$ is decreasing on $[0, \pi]$, for $ 0 \leq y \leq x \leq \f{\pi}{2}$, we get
\[
\one_{y \leq x} \cot( x+ y) 
\leq \one_{y \leq \f{\pi}{4} } \one_{y \leq x} \cot (x + y)
\leq \one_{y \leq \f{\pi}{4} } \one_{y\leq x} \cot  y \leq \cot y^2,
\]
where we have used $x+ y \geq \f{\pi}{2}, \cot (x+ y) \leq 0$ if $ \f{\pi}{4} \leq y \leq x$ in the first inequality. Since $\om \leq 0$ on $[0, \f{\pi}{2}]$, we prove \eqref{eq:EE_comp1}. Plugging \eqref{eq:EE_comp1} in the estimates \eqref{eq:EE_L1}-\eqref{eq:EE_L12}, we derive
\[
\f{d}{dt} \int_0^{\pi/2} - \om  dx
\les Q(2,t)  \int_0^{\pi/2} - \om  dx
\leq Q(2,0)  \int_0^{\pi/2} - \om  dx.
\]

Using the above estimate and the interpolation \eqref{eq:hol_ux0} with $\b = 2$, we obtain
\[
|| \om||_{L^1} \leq || \om_0||_{L^1} e^{C Q(2,0) t}, 
\quad | u_x(0) |  \les ( Q(2,t) || \om||_{L^1})^{1/2} 
\les  ( Q(2,0) || \om_0||_{L^1} )^{1/2}e^{C Q(2,0) t},
\]
for some constant $C>0$. Applying the same argument as that in Sections \ref{sec:EE_near} and \ref{sec:EE_far} with $U(t)$ replacing by $C Q(2,0) t$, we establish 
\[
|| \om||_{L^{\inf}}
\leq K(\om_0) \exp( 2 \exp( K(\om_0) \exp( C Q(2,0) t)   )),
 \]
where we have used $\int_0^t \exp( C Q(2,0) s) ds \les K(\om_0) e^{ C Q(2,0) t}$ and $K(\om_0)$ is some constant depending on $ H \om_0(0), H\om_0( \f{\pi}{2}), || \om_0||_{L^1}, Q(2,0)$ and $A(\om_0)$. We prove the result in Theorem \ref{thm:GWP} for the case of $ \om_0 \in C^{1,\al}$ with $\om_{0,x}(0) = 0$.

%\begin{remark}
%We remark that similar a-priori estimates hold for solution from initial data $\om_0$ with lower regularity, e.g. $\om_0 / |x|^{1+\al} \in L^{\inf}$ for some $\al > 0$.
We remark that the above a-priori estimates can be generalized to initial data $\om_0$ with lower regularity, e.g. $\om_0 / |x|^{1+\al} \in L^{\inf}$ for some $\al > 0$ and $\om_0 \in X \cap H^1$.
%with regularity assumption $\om_0 \in C^{1,\al}$ with $\om_{0,x}(0) =0$ can be weaken to 
%\end{remark}

\subsection{ General Case}
%: $\om_{0,x} \neq 0$}
Recall from Section \ref{sec:idea_GWP} the difficulties and ideas in the general case where $\om_0$ can vanish only linearly near $x=0$.
%$\om_{0,x} \neq 0$ 
In this case, %the monotonically decreasing 
the monotone quantity $Q(2,t)$ in the previous case is not well-defined and not applicable. We will exploit a relation similar to the conservation law $\om_x(0, t) = \om_{0,x}(0)$ %in Section \ref{sec:near0} 
and control $Q(\b, t)$ for $\b$ sufficiently close to $2$.

%In the subcritical case, the proof relies on the coercive conserved quantity $Q(2,t)$. In the critical case, the proof is much more challenging since $Q(2,t)$ is not well defined and there is no similar coercive conserved quantity. Note that in the critical case, from \cite{Sve19}, for $\om_0$ close to $A \sin 2 x$, the solution $\om(x, t)$ converges to $A \sin 2 x$ as $t \to \infty$. This put strong constraints on possible conserved quantities. Thus, it is not expected that there is any good conserved quantity similar to some weighted norm of $\om$. See also similar discussion in \cite{Sve19}. 

\subsubsection{ Estimate of $\om x^{-1}$ }\label{sec:near0}
For the less regular initial data $\om_0 \in H^1$ with $\om_0 x^{-1} \in L^{\inf}$, $\om_{x}(0, t)$ is not well-defined. Instead of using the conservation law $\om_{x}(0, t) = \om_{0,x}(0)$, 
%We develop an estimate similar to the conservation law $\om_{x}(0, t) = \om_{0,x}(0)$ by showing 
we show that $\om(x, t) x^{-1}$ cannot grow too fast for $x$ near $0$. Consider the flow map 
\beq\label{eq:flow}
 \f{d}{dt}\Phi(x, t) = u( \Phi(x, t), t), \quad \Phi(x, 0) = x.
\eeq

We focus on $x\in[0, \f{\pi}{2} ]$. Since $u( x, t) \geq 0$, $u(0, t) = 0$, and $u( \f{\pi}{2}, t) = 0$, we get
% Note that $u( x, t) \geq 0$ for $x \in [0,\pi/2]$ and $u(0, t) = 0, u(\pi/2, t) = 0$. For $x \in [0,\pi/2]$, we get
\beq\label{eq:flow2}
%\dot{ \Phi}(x, t) \geq 0 , 
\f{d}{dt}  \Phi (x, t) \geq 0, \quad  0\leq \Phi(x, t_1) \leq \Phi(x, t_2), 
\eeq
for $t_1 \leq t_2$. Using \eqref{eq:DG}, we derive the equation of $\om / x$
\[
\pa_t \f{ \om}{x} + u \pa_x( \f{\om}{x} ) = ( u_x - \f{u}{x} ) \f{\om}{x}.
\]

Fix $\g \in (0, \f{1}{2})$. Using the embedding $ H^1  \hookrightarrow C^{\g}$, we have $\om, u_x \in C^{\g}$. Since $u_x(x) - \f{u(x)}{x} = 0$ at $x = 0$ and $\om \leq 0$ on $[0, \pi/2]$, for $ x \in [0,\pi/2]$, we yield 
% we yield 
% \[
% |u_x(x) - \f{u(x)}{x}| \les x^{\g} || \om||_{C^{\g}}.
% \]
% Recall that $x \in [0, \pi/2]$, we have $\om \leq 0$. It follows 
\[
\bal
\f{d}{dt} \B(-\f{\om( \Phi(x, t), t)}{\Phi(x, t)} \B)
&= (u_x( \Phi(x, t) , t) - \f{ u(\Phi(x, t), t )}{ \Phi(x, t)}  ) \B(   -\f{\om( \Phi(x, t), t)}{\Phi(x, t)} \B) \\
&\les | \Phi(x, t)|^{\g} || \om||_{ H^1 }  \B| \f{\om( \Phi(x, t), t)}{\Phi(x, t)} \B|.
\eal
\]

Denote 
%\beq\label{eq:minf}
\[
m \teq || \om_0 x^{-1}||_{L^{\inf}}.
\]

Using Gronwall's inequality and \eqref{eq:flow2}, we derive
\[
\B|\f{ \om( \Phi(x, t) , t)}{ \Phi(x, t)  } \B|
\leq \exp( C \int_0^t  |\Phi(x, s) |^{\g} || \om(s) ||_{ H^1 } ds )  || \f{\om_0}{x}||_{L^{\inf}}
\leq  m \exp( C |\Phi(x, t)|^{\g}  \int_0^t  || \om(s) ||_{ H^1 } ds )  .
\]

Since $\Phi(\cdot, t)$ is a bijection from $[0, \pi/2]$ to $[0, \pi/2]$ and $x$ is arbitrary, we yield 
\beq\label{eq:conserv}
|\f{ \om(x, t)}{x}| \leq m \exp( C |x|^{\g} \int_0^{ t} || \om(s) ||_{H^1 } ds )
\leq m ( 1+  C |x|^{\g} \exp( C   \int_0^{ t} || \om(s) ||_{ H^1}  ds   ),
\eeq
where we have used $ |x| \leq \pi/2,  e^{ A x} \leq 1 + Ax \cdot e^{ Ax} \leq 1 + C x e^{ C A } $ for some absolute constant $C$ in the last inequality. The above estimate shows that $\limsup_{x \to 0 } | \om(x, t) / x|$ is bounded uniformly in $t$, which is an analog of $\om_x(0, t) = \om_{0, x}(0)$. Moreover, we obtain that $\om(x, t) x^{-1} \in L^{\inf}$.  
%Clearly, similar estimate holds for $\om_0 \in X \cap C^{\al}$ with any $\al \in (0, 1)$ and $\om_0 /x \in L^{\inf}$.

\subsubsection{Weighted $L^1$ estimates}
%Note that $\om_{x}(0,t)$ is conserved: $\om_x(0,t) = \om_{0,x}(0)$. Surprisingly, this one-point conserved quantity allows us to control $Q(\b, t)$ and $|| \om||_{L^1}$ for $\b$ sufficiently close to $2$. 
From the local well-posedness result and \eqref{eq:conserv}, we have $\om(t) \in X\cap H^1$ and $\om(x, t) x^{-1} \in L^{\inf}$, and $\om(t)$ satisfies the assumptions in Lemma \ref{lem:comp}. A key step to control $Q(\b, t)$ is establishing the following weighted $L^1$ estimates.
%A key step is to establish the following weighted $L^1$ estimates.

\begin{lem}\label{lem:extrap_Q}
Let $\b_0$ be the parameter in Lemma \ref{lem:comp}.
For $\b \in [\b_0, 2)$, we have %the following estimates 
\beq\label{eq:extrap_Q0}
 \bal
\f{d}{dt} Q(\b, t) &\leq C (2-\b) Q^2(\b,t) +  C(2-\b) D(t) ,  \quad \f{d}{dt}  || \om||_{L^1}
\leq C  Q^2(\b, t) 
- C_2  D(t), 
\eal
\eeq
for some absolute constant $C, C_2 > 0$, where $D(t) \geq 0$ is a damping term 
%the damping term $D(t) \geq 0$ is 
given by 
\beq\label{eq:damp_L1}
\ D(t) = -\int_0^{\pi/2}\int_0^{\pi/2}
 \om(x)\om(y) \cot(x+y) \one_{x + y > \pi/2} dx dy.
 \eeq
 As a result, for some absolute constant $\lam > 0$, we have 
 \beq\label{eq:extrap_Q}
\f{d}{dt} ( Q(\b, t) + \lam (2-\b)|| \om||_{L^1}) \les (2-\b) Q^2(\b,t).
 \eeq
\end{lem}

At first glance, the estimate \eqref{eq:extrap_Q} looks terrible due to the quadratic nonlinearity $Q^2(\b,t)$. Yet, we have a crucial small factor $2-\b$, which can compensate the nonlinearity. 
%The conservation of $\om_x(0, t) =\om_{0,x}(0)$ 
The boundedness of $\om x^{-1}$ for $x$ near $0$ \eqref{eq:conserv}  implies the following leading order structure of $Q(\b,t)$
\[
\bal
Q(\b, t) &= - \int_0^{\pi/2} \om(x,t) (\cot x )^{\b} dx 
\leq m \int_0^1  x \cdot x^{-\b} dx +  \cR(\b, t) \leq \f{ m }{2-\b} + \cR(\b, t),
% Q(\b, t) &= - \int_0^{\pi/2} \om(x,t) (\cot x )^{\b} dx 
% = -\int_0^1 \om_{x}(0,t) x \cdot x^{-\b} dx +  l.o.t = \f{ -\om_{0,x}(0)}{2-\b} + l.o.t.
\eal
\]
where the remainder $\cR(\b, t)$ is of order lower than $(2-\b)^{-1}$. For $\b$ sufficiently close to $2$, we get $(2-\b) Q(\b, t) \les m $, which is time-independent. 
% The conservation of $\om_x(0, t) =\om_{0,x}(0)$ implies the following leading order structure of $Q(\b,t)$
% \beq\label{eq:lead_Q}
% \bal
% Q(\b, t) &= - \int_0^{\pi/2} \om(x,t) (\cot x )^{\b} dx 
% = -\int_0^1 \om_{x}(0,t) x \cdot x^{-\b} dx + \cR(\b, t) = \f{ -\om_{0,x}(0)}{2-\b} + \cR(\b, t), 
% \eal
% \eeq
% where the remainder satisfies $|\cR( \b, t) | \les_{\al}|| \om||_{C^{1,\al}}$ (see \eqref{eq:lead_R}). 
% For $\b$ sufficiently close to $2$, we get $(2-\b) Q(\b, t) \approx -\om_{0,x}(0)$, which is time-independent. 
Formally, the nonlinearity in \eqref{eq:extrap_Q} becomes linear. In Section \ref{sec:boost}, we will apply \eqref{eq:extrap_Q} and this key observation to prove Theorem \ref{thm:GWP}.

%In Section \ref{sec:boost}, we will use this observation and  control some quantity similar to $\lim_{\b \to 2^-} Q(\b, t)$, which is barely beyond the range $\b \geq 2$ where $Q(\b, t)$ is expected to be conserved, with a growth rate similar to the exponential growth. 

% by a cost similar to 

% exponential growth factor similar to 

% an exponential grow

%See Lemma \ref{lem:comp} and the discussion in Section \ref{sec:comp}. In some sense, our estimates on $Q(\b, t)$ to be established are extrapolation.

%Thus, \eqref{eq:extrap_Q} and related estimate can be interpreted as an extrapolation.
%this estimate can be interpreted as an extrapolation.
%The above estimate can be interpreted as an extrapolation, since we can control some quantity similar to $\lim_{\b \to 2^-} Q(\b, t)$, which is barely beyond the range $\b \geq 2$ where $Q(\b, t)$ is expected to be conserved. 

% The first estimate in \eqref{eq:extrap_Q0} %is highly nontrivial since it 
% shows that the forcing term $u_x(0) Q(\b)$ in \eqref{eq:ODE_Q} and \eqref{eq:comp_B} can be controlled roughly by $Q^2(\b)$, which is highly nontrivial
% %  $u_x(0) \les Q(\b)$ is not 
% since the integrand $\om \cot^{\b}(y)$ in $Q(\b)$ is much weaker than $\om \cot y$ in $u_x(0)$ for $y$ close to $\pi/2$.

The first estimate in \eqref{eq:extrap_Q0} is highly nontrivial since the forcing term $u_x(0) Q(\b)$ (see \eqref{eq:est_Q1}) cannot be controlled by $Q^2(\b)$. %We establish it thanks to the magic damping term $D(t)$ \eqref{eq:damp_L1} from \eqref{eq:GWP_L1}. 
The idea behind Lemma \ref{lem:extrap_Q} is that for the forcing terms $B(\b,t)$ in \eqref{eq:ODE_Q} and \eqref{eq:comp_B} and that in \eqref{eq:GWP_L1}, we use the more singular integral $Q(\b, t)$ to control them near $x=0$, and the %crucial 
magic damping term $D(t)$ from \eqref{eq:GWP_L1} to control them near $x = \pi/2$. To prove Lemma \ref{lem:extrap_Q}, we need several 
inequalities, whose proofs are deferred to Appendix \ref{app:lemma}.
%To prove Lemma \ref{lem:extrap_Q}, we need several elementary inequalities that enable us to compare the kernel. 
\begin{lem}\label{lem:trigon}
Denote $a \wedge b = \min(a, b)$. For $x, y \in [0, \pi/2], \b \in [3/2, 2]$, we have 
\begin{align}
\cot(x+y) 
&\leq \one_{x+y \geq \pi/2} \cot(x+y) + (\cot x \cot y)^{\b}  \label{eq:trigon1} ,\\
  \cot y (\cot x )^{\b-2}  \wedge \cot x (\cot y)^{\b-2} 
&\les (\cot x \cot y)^{\b} + \one_{x+y \geq \pi/2} \cot(\pi - x- y) \label{eq:trigon2} ,\\
  \cot y \one_{y \geq \pi / 3 } & \les  (\cot x \cot y)^{\b} + \one_{x+y \geq \pi/2} \cot(\pi - x- y) \label{eq:trigon3} .
%\eal
\end{align}
%We have the following estimate 
\end{lem}

\begin{proof}[Proof of Lemma \ref{lem:extrap_Q}]

%We first prove the second inequality in \eqref{eq:extrap_Q0}. 
%Recall the identity \eqref{eq:GWP_L1}. 
Using $\om(x) \om(y) \geq 0$ for $x, y \in [0, \pi/2]^2$ and \eqref{eq:trigon1}, we obtain
%\eqref{eq:def_Q}, we obtain 
\[
\bal
\int_0^{ \f{\pi}{2} } \int_0^{ \f{\pi}{2}}
\om(x) \om(y) \cot(x+y)dx dy
&\leq - D(t) + \int_0^{\pi/2} \int_0^{\pi/2} \om(x) \om(y) (\cot x \cot y)^{\b} dx dy  \\
&\leq - D(t) + Q^2(\b, t),
\eal
\]
where $D(t)$ is defined in \eqref{eq:damp_L1}. Using the above estimate and  \eqref{eq:GWP_L1}, we prove the second estimate in \eqref{eq:extrap_Q0}. Recall the ODE of $Q(\b, t)$ \eqref{eq:ODE_Q}. Applying Lemma \ref{lem:comp}, for $\b \in [\b_0, 2)$, we get
\beq\label{eq:est_Q1}
\bal
\f{d}{dt} Q(\b, t)
&\leq (2 - \b)  \B( u_x(0)  Q(\b, t) +  \f{1}{\pi}\iint_{ [0, \pi/2]^2} \om(x)\om(y) (\cot y)^{\b-1}
  \f{   s (s^{\b-1} - 1) }{ s^2 - 1}  dx dy \B) \\
  &\teq (2-\b)( I_1 + I_2),
  \eal
\eeq
where $s = \f{\cot x}{\cot y}$. Next, we estimate $f(s) = \f{   s (s^{\b-1} - 1) }{ s^2 - 1}$. Note that $\b \in (3/2, 2)$. For $s \geq 0$, the following estimate is straightforward 
\[
0 \leq f(s) \les \one_{s < 1/2} s + \one_{ 1/2 \leq s \leq 2}  + \one_{s \geq 2} s^{\b-2}
\les s \wedge s^{\b-2}.
\]

Since $ s = \f{\cot x}{\cot y}$, using the above estimate and \eqref{eq:trigon2}, we yield
 \[
 \bal
 f(s) (\cot y)^{\b-1} & \les  ( s \wedge s^{\b-2} ) \cdot (\cot y)^{\b-1}
= \cot y (\cot x)^{\b-2} \wedge \cot x (\cot y )^{\b-2} \\
&\les (\cot x \cot y)^{\b} + \one_{x+y \geq \pi/2} \cot(\pi - x- y) = (\cot x \cot y)^{\b} - \one_{x+y \geq \pi/2} \cot( x + y).
\eal
\]

Using $\om(x)\om(y) \geq 0$ for $x, y\in [0,\pi/2]$, the above estimate and \eqref{eq:damp_L1},  we derive 
 %and \eqref{eq:trigon2},
\[
0 \leq I_2 
\les \int_0^{\pi/2}\int_0^{\pi/2} \om(x)\om(y)
\B( (\cot x \cot y)^{\b} - \one_{x+y \geq \pi/2} \cot( x + y) \B) dx dy
= Q^2(\b, t)  + D(t) .
\]

%Next, we estimate $I_1$. 
For $I_1$, we cannot establish the desired estimate by comparing the kernel similar to the above
since 
\[
\cot y (\cot x)^{\b} 
\les (\cot x \cot y)^{\b} - \one_{x+y \geq \pi/2} \cot( x + y)
\]
does not holds for $x$ close to $0$ and $y$ close to $\pi/2$. In fact, for $\pi/2- y = t^{\b}, x = t$, with $t$ sufficiently small, the left hand side is of $O(1)$, while the right hand side is $o(t)$. The main difficulty lies in that $ (\cot y)^{\b}$ is too weak to control $\cot y$ for $y$ close to $\pi/2$. 

A key observation is that we can further impose the restriction $Q(\b, t) \leq u_x(0)  \les || \om||_{L^1}$. In fact, if $u_x(0) \leq Q(\b, t)$, we obtain the trivial estimate 
\[
I_1 = u_x(0) Q(\b, t) \leq Q^2(\b, t).
\]

In the other case $Q(\b, t) \leq u_x(0)$, thanks to the interpolation \eqref{eq:hol_ux0}, we derive 
\[
u_x(0) \les Q(\b, t)^{1/\b} || \om||_{L^1}^{1-1/\b} 
\leq (u_x(0) )^{1/\b} || \om||_{L^1}^{1-1/\b} ,
\]
which implies $u_x(0) \les || \om||_{L^1}$. Now, we decompose $I_1= u_x(0) Q(\b, t)$ as follows 
\[
I_1 %= u_x(0) Q(\b, t)
\les \int_0^{\pi/2} |\om| \cot y dy Q(\b, t)
= \int_0^{\pi/3} |\om| \cot y dy Q(\b, t)
+ \int_{\pi/3}^{\pi/2} |\om| \cot y dy Q(\b, t) \teq J_1 + J_2.
\]

For $J_1$, since $\cot y \les (\cot y)^{\b}$ for $ y \leq \pi / 3$, we get 
$J_1 \les Q^2(\b, t)$. For $J_2$, using $Q(\b, t)\leq u_x(0) \les || \om||_{L^1}$, we yield 
\[
J_2 \les \int_{\pi/3}^{\pi/2} |\om(y)| \cot y || \om||_{L^1}
\les \int_{\pi/3}^{\pi/2} \om(y) \cot y dy   \int_0^{\pi/2} \om(x) dx,
\]
where we have used $\om(x) \leq 0$ on $[0,\pi/2]$ to obtain the last inequality. Applying \eqref{eq:trigon3} and $\cot(\pi-x-y) = -\cot(x+y)$, we obtain 
\[
J_2 
\les \int_0^{\pi/2}\int_0^{\pi/2} \om(x)\om(y)
\B( (\cot x \cot y)^{\b} - \one_{x+y \geq \pi/2} \cot( x + y) \B) dx dy
= Q^2(\b, t)  + D(t) .
\]
Combining the above estimates on $J_1, J_2$, in the other case $Q(\b, t) \leq u_x(0)$, we prove 
\[
I_1 \les J_1 + J_2 \les  Q^2(\b, t)  + D(t) .
\]
Combining the above estimates on $I_1, I_2$, we establish the first inequality in \eqref{eq:extrap_Q0}. Estimate \eqref{eq:extrap_Q} follows directly from \eqref{eq:extrap_Q0}
by choosing $\lam > 0$ with  $C_2 \lam  \geq 2C$, e.g. $\lam= \f{2C}{C_2}$.
\end{proof}

\begin{remark}
We cannot apply \eqref{eq:hol_ux0} to estimate $u_x(0)$ in $I_1 $ directly, since 
%In fact, 
such estimate only offers 
%leads to 
\[
\f{d}{dt} ( Q(\b, t) + \mu || \om||_{L^1}) 
\les (2-\b)^{\g} ( Q(\b, t) + \mu || \om||_{L^1})^2
\]
with power $\g <1$ for any well chosen $\mu$, which is not sufficient for our purpose. Compared to \eqref{eq:extrap_Q}, the above estimate loses a small factor $(2-\b)^{1-\g}$,
%there is a loss of power $(2-\b)^{1-\g}$ in the above estimate, 
which is due to the fact that we do not have a good estimate on $|| \om||_{L^1}$, while for $Q(\b, t)$ we have the crucial small factor $2-\b$. We only add minimal amount of $|| \om||_{L^1}$ in the  energy in \eqref{eq:extrap_Q} due to a similar reason. 

\end{remark}

\subsubsection{A bootstrap estimate}\label{sec:boost}

Now, we are in a position to establish the global well-posedness result in Theorem \ref{thm:GWP} in the general case. It follows from a bootstrap lemma.

\begin{lem}\label{lem:boost}
Suppose that $\om_0$ satisfies the assumptions in Theorem \ref{thm:GWP}. Denote $m = || \om_0 x^{-1}||_{L^{\inf}}$. There exists some absolute constant $c$, such that for $\d = \f{c}{m}$, if  $\int_0^{T} u_x(0, s) ds < +\infty$, we have $\int_0^{T + \d} u_x(0, s) ds < +\infty$.
% Suppose that $\om_0$ satisfies the assumptions in Theorem \ref{thm:GWP} with $\om_{0,x} \neq 0$. There exists some absolute constant $c$, such that for $\d = \f{c}{ |\om_{0,x}(0) |}$, if  $\int_0^{T} u_x(0, s) ds < +\infty$, we have $\int_0^{T + \d} u_x(0, s) ds < +\infty$.

\end{lem}

\begin{proof}

% Denote 
% \[
% H(\b, t) = Q(\b, t) + \lam (2- \b) || \om||_{L^1} , \quad  m = - \om_{0,x}(0) = | \om_{0,x}(0)|.
% \]
% In view of Theorem \ref{thm:criterion} and \eqref{eq:hol_ux0}, for $\om_0 \in C^{1,\al} \cap X$, the solution $\om(x, t)$ remains in $C^{1,\al}$ if $H(\b, t) < +\infty$ for some $\b < 2$. Thus, it suffices to control $H$. Using Lemma \ref{lem:extrap_Q}, we have 
Without loss of generality, we assume $m > 0$. Recall $Q(\b, t)$ from \eqref{eq:def_Q}. Denote 
\[
H(\b, t) = Q(\b, t) + \lam (2- \b) || \om||_{L^1} .
\]
In view of Theorem \ref{thm:criterion} and \eqref{eq:hol_ux0}, for $\om_0 \in H^1 \cap X$, the solution $\om(x, t)$ remains in $H^1$ if $H(\b, t) < +\infty$ for some $\b < 2$. Thus, it suffices to control $H$. Using Lemma \ref{lem:extrap_Q}, we have 
\beq\label{eq:ODE_H}
\f{d}{dt} H(\b, t) \leq \mu (2-\b)H^2(\b,t)
\eeq
for some absolute constant $\mu>0$ and any $\b \in [\b_0, 2)$. Since $\int_0^T u_x(0, s) ds<0$, using Theorem \ref{thm:criterion}, we obtain $\sup_{t \leq T} || \om(t)||_{H^1} < +\inf, || \om(T)||_{L^1} < +\infty$. Using \eqref{eq:conserv}, we obtain 
\beq\label{eq:lead_R}
\bal
Q(\b, T )
&= \int_0^{\pi/2} |\om| (\cot y)^{\b} dy
\leq \int_0^{1} |\om| y^{\b} dy + C \int_0^{\pi/2} |\om| dy \\
&\leq m \int_0^{1} \B( y^{1-\b} + C y^{\g + 1 - \b} \exp( C T \sup_{t \leq T} || \om(t)||_{H^1} ) \B) dy
+ C || \om(T)||_{L^1} \\
&\leq \f{m}{2-\b} + C m \exp( C T \sup_{t \leq T} || \om(t)||_{H^1} )  + C || \om(T)||_{L^1},
\eal
\eeq
where $C$ is some absolute constant and we have used %$|\cot x - x^{-1}| \les x$ and 
$| (\cot x)^{\b} - x^{-\b}| \les |\cot x - x^{-1}| x^{-\b+1} \les x^{-\b+2} \les 1 $ in the first inequality.
% Recall the decomposition \eqref{eq:lead_Q}. The remainder term satisfies 
% \beq\label{eq:lead_R}
% \bal
% |\cR| & = \B| \int_1^{\pi/2} - \om  (\cot x)^{\b} dx 
% + \int_0^1 (-\om + \om_{x}(0)x) (\cot x)^{\b} 
% - \om_x(0) x ( ( \cot x)^{\b} - x^{-\b} ) dx \B| \\
% & \les || \om||_{L^{\inf}}
% + ||\om||_{C^{1,\al}} \int_0^1 |x|^{1+\al - \b}
% + x^{-\b + 3}    dx 
% \les_{\al} || \om ||_{C^{1,\al}},
% \eal
% \eeq
% where we have used $|\cot x - x^{-1}| \les x$ and $| (\cot x)^{\b} - x^{-\b}|
% \les |\cot x - x^{-1}| x^{-\b+1} \les x^{-\b+2} $. 
% Since $m =  -\om_{0,x}(0) > 0$, 
%From  \eqref{eq:lead_R}, 
Thus, there exists $\b_1 $ slightly less than $2$ , such that 
\[
\bal
H(\b_1, T)  &= Q(\b_1, T) + \lam (2-\b_1) || \om(T)||_{L^1} \\
&\leq   \f{m}{2-\b_1} + C m \exp( C T \sup_{t \leq T} || \om(t)||_{H^1} )  +  C || \om(T) ||_{L^1}  \leq \f{2 m }{2-\b_1}.
\eal
\]
% \[
% \bal
% H(\b_1, T)  &= Q(\b_1, T) + (2-\b_1) || \om(T)||_{L^1}\leq \f{m}{2 - \b_1} + C || \om(T)||_{C^{1,\al}} + (2-\b_1) || \om(T)||_{L^1}  \leq \f{2 m }{2-\b_1}.
% \eal
% \]

Solving the ODE \eqref{eq:ODE_H} with $\b = \b_1$ on $t \geq T$, we yield
\[
  \f{d}{dt} H^{-1}(\b_1, t) \geq - \mu(2-\b_1) ,
\]
which along with the estimate on $H(\b_1, T)$ imply 
\[
  H^{-1}(\b_1, T + \tau) 
  \geq H^{-1}(\b_1, T) - \mu(2 - \b_1) \tau
  \geq \f{2-\b_1}{2m } - \mu(2-\b_1) \tau .
  \]
 Note that $\mu$ is absolute. We choose $\d = \f{1}{4 m \mu} $. Then, for $t \in [T, T+\d]$, we yield
 \beq\label{eq:ODE_H2}
 H^{-1}(\b_1, t) \geq \f{2-\b_1}{2m} - \f{2-\b_1}{4m} = \f{2-\b_1}{4m}, \quad H(\b_1, t) \leq \f{4m}{2-\b_1}.
 \eeq
Applying \eqref{eq:hol_ux0}, we obtain $u_x(0, t) \les  \f{m}{ (2-\b_1)^2}$
%$u_x(0, t) \les C(\b_1, m)$ 
on $[T, T+\d]$. We conclude the proof.
\end{proof}

\begin{remark}\label{rem:boost}
Denote $V(t) = \int_0^t ( u_x(0, s) + 1) ds $. We can obtain an a-priori estimate for $V(t)$ by tracking the bounds in the above proof. Using standard energy estimates and \eqref{eq:w_linf}, we obtain
\[
\bal
%|| \om(t)||_{C^{1,\al}}
& C m \exp( C t \sup_{s \leq t} || \om(s)||_{H^1} )  +  C || \om(t)||_{L^1}  \leq g( V(t), C_1 ), \\
&g( x, c) \teq c \cdot \exp( c \cdot \exp( c\cdot \exp( c  \cdot \exp( c \cdot \exp (c \cdot  \exp(c x )) ))) ) ,
\eal
\]
for some constant $C_1>1$ depending only on the initial data. Note that the estimate of $|| \om||_{L^{\inf}}$ \eqref{eq:w_linf} is triple exponential growth, and then the estimate of $|| \om||_{ H^1 }$ is a quintuple one due to extrapolation in bounding $|| u_x||_{L^{\inf}}$. These estimates further lead to the above sextuple exponential growth. For any $T\geq 0$, choosing $\b_1$ with $2-\b_1 =c\cdot  \f{ m}{ g( V(T), C_1)} $ for some absolute constant $c$ and using \eqref{eq:hol_ux0}, \eqref{eq:ODE_H2}, we yield 
\[
V( T + \d ) \leq  g( V(T), C_2),
\]
for some constant $C_2>0$ depending only on $\om_0$. Since $\d$ and $C_2$ are independent of $T$, iterating the above estimate yields an a-priori estimate for $V(t)$ with any $t \geq 0$.
\end{remark}

\begin{remark}
The above estimate is consistent with the heuristic in the paragraph below \eqref{eq:extrap_Q} that the nonlinearity $(2-\b) Q^2$ in \eqref{eq:extrap_Q} or $(2-\b)H^2$
%the above $(2-\b_1) H^2(\b_1, t)$ 
is essentially linear. In fact, for $t \in [T, T+\d]$, \eqref{eq:ODE_H2} implies $(2-\b_1) Q(\b_1, t) \leq (2-\b_1) H(\b_1, t) \leq 4m$. Formally, $Q(\b, t)$ grows exponentially in $t$ for $\b$ close to $2$, which we can barely afford, while in the previous case, $Q(2,t)$ is bounded uniformly. This argument is similar in spirit to extrapolation, e.g. the BKM blowup criterion \cite{beale1984remarks}.
\end{remark}

% Since the estimate of $|| \om||_{L^{\inf}}$ \eqref{eq:w_linf} is triple exponential growth in $V(t)$, using standard energy estimates and extrapolation in bounding $|| u_x||_{L^{\inf}}$, we can obtain an estimate of $|| \om||_{C^{\al}}$ with 
% quintuple exponential growth in $V(t)$. This further leads to the following estimate of $|| \om||_{C^{1,\al}}$ with sextuple exponential growth 
% \[
% || \om(t)||_{C^{1,\al}} \leq g( V(t), C_1 ), \quad g( x, c) \teq c \cdot \exp( c \cdot \exp( c\cdot \exp( c \cdot \cdot \exp( c \cdot \exp (c \cdot  \exp(c x )) ))) ) ,
% \]
% for some constant $C_1>1$ depending only on the initial data. 

%The sextuple exponential growth is because the estimate of $|| \om||_{L^{\inf}}$ \eqref{eq:w_linf} is triple exponential growth, and then the estimate of $|| \om||_{ C^{\al}}$ is a quintuple one due to extrapolation in bounding $|| u_x||_{L^{\inf}}$. 

\section{Finite time blowup for $C^{\al} \cap H^s$ data}\label{sec:blowup}

In this Section, we prove Theorem \ref{thm:blowup} on finite time blowup for \eqref{eq:DG} with $C^{\al} \cap H^s$ data for any $\al \in (0, 1)$ and $s\in(1/2, 3/2)$. We will use ideas outlined in Section \ref{sec:idea}. 

Since we will adopt several estimates established in 
\cite{lei2019constantin,chen2020slightly}, for consistency, throughout this section, we assume that the solution $\om$ is $2\pi$ periodic. This modification also simplifies our notations. Theorem \ref{thm:blowup} can be established by applying the same argument to $\om_{\pi}( x) \teq \om_{2\pi}(2x)$. As a result, the Hilbert transform and the set $X$ \eqref{eq:X} becomes 
\[
H f \teq \f{1}{2\pi} P.V. \int_{-\pi}^{\pi} \cot \f{x-y}{2}  f(y) dy ,\quad 
X \teq \B\{ f: f \mathrm{ \ is \ odd \ }, 2 \pi-\mathrm{periodic \ and \ }   f(x) \leq 0 , x \in [0,  \pi ] \B\}.
\]

\subsection{ Slightly weakening the effect of advection}\label{sec:idea_blowup}

%There are two important observations on the competition between the advection and vortex stretching. Firstly, in view of Theorem \ref{thm:GWP} and its proof, the data cannot be too regular, e.g. $\om_0$ should be less regular than $C^1$, and the advection should be weakened. This can be done by choosing $C^{\al}$ data with sufficiently small $\al$ as shown in \cite{Elg17}. Secondly, 
%In our previous work \cite{chen2020slightly} on the gCLM model, we construct finite time blowup solution in the case where the advection is only slightly weaker than the vortex stretching. The results in \cite{chen2020slightly} and the discussion in Section \ref{sec:intro_comp} suggest that for any $\al$ slightly less than $1$, there exists $C^{\al}$ data of \eqref{eq:DG} that leads to finite time blowup without weakening the advection substantially. 

Recall the discussion on the competition between advection and vortex stretching in Section \ref{sec:intro_comp}. To characterize that the advection is relatively weak for $\om \in C^{\al} \cap X$ with $\om \approx - C x^{\al}$ near $x=0$, 
%slightly weaker for $\om \in C^{\al} \cap X$, 
%relatively weak for $\om \in C^{\al} \cap X$,
%we use a fundamental idea from 
we study \eqref{eq:DG} using the dynamic rescaling formulation 
\beq\label{eq:DGdyn}
\om_t + u \om_x = (c_{\om} + u_x) \om ,\quad u_x = H \om
\eeq
derived in \eqref{eq:rescal1}-\eqref{eq:rescal3} with the normalization condition 
\beq\label{eq:normal}
c_{\om}(t) = (\al - 1) u_x(0, t),
\eeq
where $c_{\om}$ is a rescaling factor. If $u_x(0, t)$ is bounded away from $0$ : $u_x(0, t) \geq C>0$ for all $t$, the competition between advection and the vortex stretching is encoded in the sign of $c_{\om}$ since $\mathrm{sign}(c_{\om}) = \mathrm{sign}(\al-1)$, which can determine the long time behavior of the solution. See the discussion below \eqref{eq:rescal3}. We remark that the idea and condition \eqref{eq:normal} are similar to those in \cite{chen2020slightly}, which play a crucial role in establishing singularity formation for the gCLM model.

\subsection{Dynamic rescaling formulation}\label{sec:dyn}

We follow the method in \cite{chen2019finite,chen2020slightly} to construct finite time blowup solution using the dynamic rescaling formulation of \eqref{eq:DG}.  Let $ \om(x, t), u(x,t)$ be the solutions of equation \eqref{eq:DG}. It is easy to show that 
\beq\label{eq:rescal1}
  \td{\om}(x, \tau) = C_{\om}(\tau) \om(   x,  t(\tau) ), \quad   \td{u}(x, \tau) = C_{\om}(\tau) 
u( x, t(\tau))
\eeq
are the solutions to the dynamic rescaling equations
 \beq\label{eq:rescal2}
\bal
\td{\om}_{\tau}  +   \td{u}  \td{\om}_x   =   c_{\om} \td{\om} + \td{u}_x \td{ \om  } , \quad \td{u}_x = H \td{\om} ,
% \td{\om}_{\tau}(x, \tau) +   \td{u}  \td{\om}_x(x, \tau)  =   c_{\om}(\tau) \td{\om} + \td{u}_x \td{ \om  } , \quad \td{u}_x = H \td{\om} ,
\eal
\eeq
where  
\beq\label{eq:rescal3}
\bal
  C_{\om}(\tau) = \exp\lt( \int_0^{\tau} c_{\om} (s)  d s \rt),
  % C_{\om}(0), 
  \quad   t(\tau) = \int_0^{\tau} C_{\om}( s) d s.
\eal
\eeq

We will impose some normalization condition on the time-dependent scaling parameter $c_{\om}(\tau)$, and establish that $-C_1 \leq c_{\om}(\tau) \leq -C <0 $ for all $\tau > 0$ and some $C_1, C >0$. Then the solution of \eqref{eq:rescal2} is equivalent to that of the original equation \eqref{eq:DG} via the transformations in \eqref{eq:rescal1}-\eqref{eq:rescal3}.
Moreover, we will establish that the solution $\td{\om}(\cdot, \tau)$ is nontrivial, e.g. $ || \td{\om}( \cdot,\tau) ||_{L^{\infty}} \geq c >0$, for all $\tau >0$. Then the rescaling relationship \eqref{eq:rescal1}-\eqref{eq:rescal3} implies 
\[
C_{\om}(\tau) \leq e^{-C\tau}, \quad t(\infty) \leq \int_0^{\infty}  e^{-C \tau } d \tau =C^{-1} <+ \infty 
\]
and that the solution 
\[
| \om( x,  t(\tau) ) | = C_{\om}(\tau)^{-1}  |\td{\om}(x, \tau) | \geq e^{C\tau} |\td{\om}(x, \tau) | 
\]
blows up at finite time $T = t(\infty)$.

Note that a similar dynamic rescaling formulation was employed in \cite{mclaughlin1986focusing,  landman1988rate} to study the nonlinear Schr\"odinger (and related) equation. 
This formulation is closely related to the modulation technique, 
which has been developed by Merle, Raphael, Martel, Zaag and others, see, e.g. \cite{merle1997stability,kenig2006global,merle2005blow,martel2014blow,merle2015stability}.
%In some literature, this formulation is called the modulation technique. 
It has been a very effective tool to study singularity formation for many problems like the nonlinear Schr\"odinger equation \cite{kenig2006global,merle2005blow}, the nonlinear wave equation \cite{merle2015stability}, the nonlinear heat equation \cite{merle1997stability}, the generalized KdV equation \cite{martel2014blow}. %, and other dispersive problems. 
Recently, it has been used to establish finite time blowup from smooth initial data in model problems for the 3D Euler equations, including the DG model \cite{chen2019finite}, the gCLM model \cite{chen2019finite,chen2020singularity,chen2020slightly,Elg19} and the Hou-Luo model \cite{chen2021HL}.
%, related to the 3D Euler equations.

%fluids mechanics. 

%model problems in fluids mechanics \cite{chen2019finite,chen2020singularity,chen2020slightly,chen2021HL,elgindi2019stability}.

%applied to construct finite time blowup solution from smooth initial data in model problems in fluids mechanics, %including the 
%e.g. the DG model \cite{chen2019finite}, the gCLM model \cite{chen2019finite,chen2020singularity,chen2020slightly,elgindi2019stability} and the Hou-Luo model \cite{chen2021HLsupp}. 
 %\cite{chen2019finite,chen2020singularity,chen2020slightly,chen2021HL}.

To simplify our presentation, we still use $t$ to denote the rescaled time in the rest of this section, unless specified, and drop $\td{\cdot}$ in \eqref{eq:rescal2}. Then \eqref{eq:rescal2} reduces to \eqref{eq:DGdyn}.

\subsection{ Construction of the $C^{\al}$ approximate steady state}

Based on the discussion in Sections \ref{sec:intro_comp} and \ref{sec:idea_blowup}, we first construct an approximate steady state $( \om_{\al},  c_{ \om, \al} )$ of \eqref{eq:DGdyn} with $\om_{\al} \in C^{\al}$ and $\om_{\al} \approx - C x^{\al}$ near $x = 0$. Following the idea in \cite{chen2020slightly}, we perform the construction by perturbing the equilibrium $  \sin(x)$ of \eqref{eq:DG}. A natural choice of $\om_{\al}$ is
\beq\label{eq:wal}
 \om_{\al} = - \sgn(x) |\sin(x)|^{\al} c_{\al}, \quad c_{\al} = ( \f{1}{\pi} \int_0^{\pi} (\sin x)^{\al} \cot \f{x}{2} dx )^{-1}. 
\eeq
We choose the above $c_{\al}$ to normalize $H \om_{\al}(0) = 1$.
% for some constant $c_{\al}$. We choose $c_{\al}$ according to $H \om_{\al}(0) = 1$, which gives
% \beq\label{eq:c_al}
% c_{\al} = ( \f{1}{\pi} \int_0^{\pi} (\sin x)^{\al} \cot \f{x}{2} dx )^{-1}. 
% \eeq
%which is equivalent to $ H \om_{\al}(0) = 1$. 
Let $u_{\al}$ be the associated velocity with $u_{\al, x} = H \om_{\al}$. We choose $c_{\om, \al}$ according to \eqref{eq:normal}
\beq\label{eq:normal_al}
c_{\om,\al} = (\al - 1) u_{\al,x}(0) =\al- 1.
\eeq

Denote 
\beq\label{eq:eta_al}
\om_1 = -\sin x, \quad u_1 = \sin x, \quad \eta_{\al} = \om_{\al} - \om_1.
\eeq

For $\al$ close to $1$, we expect that $(\om_{\al}, u_{\al})$ are close to $(\om_1, u_1)$. %We have the following estimates.

\begin{lem}\label{lem:profile}
Let $\kp_1 = \f{3}{4}, \kp_2 = \f{7}{8}$. For $\kp_2 <\f{9}{10}< \al < 1$ and $x \in [-\pi,\pi]$, we have 
\begin{align}
 &|  \pa_x^i \eta_{\al} | \les (1-\al) |\sin x|^{\kp_2 - i }, \quad  i = 1,2,3 , \label{eq:est_wal1} \\
 &  |H\eta_{\al} |  \les  (1-\al) |x|^{\kp_1}, \quad | \pa_x H \eta_{\al}|   \les (1-\al) |\sin x|^{\kp_1-1} ,  \label{eq:est_ual} \\
  & | (\al-1) \om_{\al} - \sin x ( \om_{\al, xx} - \om_{1,xx}) |  \les  ( (1-\al)  \wedge  |x|^2 )
|\sin x|^{\al-1}.
%\les (1-\al)^{1/2} |x| |\sin x|^{\al-1}
   .\label{eq:est_wal2}
\end{align}

\end{lem}

For $x$ near $0$, the above estimates on $\om_{\al}$ are similar to those for $ \om_{\al} = -x^{\al}$ and $\om_1 = -x$. The reader can think of $\kp_1, \kp_2$ close to $1$, and that $\al$ is even closer to $1$. %In the following discussion, we use the implicit notation $ O(f)$ to denote some term $g$ that satisfies $|g| \les f$. It can vary from line to line.

\begin{proof}
Due to symmetry, it suffices to consider $x \geq 0$. 

Firstly, using Lemma \ref{lem:1} and $1\les \al - \kp_2 $, we obtain 
\beq\label{eq:sin_1}
| (\sin x)^{\al} - \sin x | = (\sin x )^{\kp_2}  (\sin x)^{\al - \kp_2} (1 - (\sin x)^{1-\al})
\les (1-\al) (\sin x )^{\kp_2}.
\eeq

Recall $c_{\al}$ defined in \eqref{eq:wal}. Using the above estimate, we obtain 
\beq\label{eq:est_cal}
\f{1}{\pi} \int_0^{\pi} | (\sin x )^{\al} - \sin x | \cot \f{x}{2} dx \les 1- \al,
\quad |c_{\al} - 1| \les 1-\al.
\eeq

Next, we establish the estimate of $\om_{\al}$ defined in \eqref{eq:wal}. A direct calculation yields
\beq\label{eq:Dwal}
\om_{\al, x} = -c_{\al} \al (\sin x)^{\al-1} \cos x , \quad
\om_{\al, xx} =  - c_{\al} \al  (\al-1)( \sin x)^{\al- 2} \cos^2 x + \al c_{\al} (\sin x)^{\al}.
\eeq

 We consider a typical case $i=3$ in \eqref{eq:est_wal1}, and the case $i=1$ or $2$ can be proved similarly. Recall $\om_1, u_1, \eta_{\al}$ from \eqref{eq:eta_al}. Using \eqref{eq:sin_1}, \eqref{eq:est_cal} and $\kp_2 < \al$, we get 
\[
\bal
|\eta_{\al, xx}| &= | \om_{\al, xx} -\sin x|
\les |\al c_{\al} (\sin x)^{\al } - \sin x|  +   (1-\al) (\sin x)^{\al-2}  \\
&\les  |(\sin x)^{\al} - \sin x |+  (1-\al) (\sin x)^{\al-2} 
%\les (1-\al)(  (\sin x)^{\kp_2} +  (\sin x)^{\al-2} ) 
\les (1-\al) (\sin x)^{\kp_2 - 2}.
\eal
\]

For \eqref{eq:est_wal2}, the first bound $(1-\al) |\sin x|^{\al-1}$ follows directly from \eqref{eq:est_wal1}. Using \eqref{eq:Dwal}, $|\om_{1,xx}| = \sin x$ and a direct calculation, we yield 
\[
\bal
&| (\al-1) \om_{\al} - \sin x (\om_{\al, xx} - \om_{1,xx})| \\
\leq & | c_{\al} (\al-1) \al (\sin x)^{\al-1}  (\cos x - \cos^2 x) |
+ C (\sin x)^{\al + 1} +  \sin x |\om_{1,xx}| \\
 \les & (\sin x)^{a-1 } |x|^2 + (\sin x )^{a+1} \les (\sin x )^{a-1} |x|^2,
\eal
\]
where we have used $|1 - \cos x| \les x^2$.

Next, we prove \eqref{eq:est_ual}. Denote $D_x = \sin x \pa_x$. Using \eqref{eq:est_wal1} and $\kp_2 =\f{7}{8}$ close to $1$, we have 
\beq\label{eq:L4}
\bal
&|| \pa_x \eta_{\al}||_{L^4}
\les (1-\al) || \ |\sin x|^{\kp_2 - 1} ||_{L^4} \les 1-\al, \\
& || \pa_x (D_x \eta_{\al} ) ||_{L^4}
\les || \ |\pa_x \eta_{\al} | + |\sin x  \pa_x^2 \eta_{\al} | \ ||_{L^4}
\les (1-\al) |\sin x|^{\kp_2 - 1} ||_{L^4} \les 1-\al.
\eal
\eeq

Recall from \eqref{eq:wal} that $ u_{\al,x}(0) = H \om_{\al}(0) = 1 = u_{1,x}(0)$. It implies $H \eta_{\al}(0) = 0$. Since the Hilbert transform is $L^4$ bounded, using H\"older's inequality and \eqref{eq:L4}, we yield 
%\beq\label{eq:L4_2}
\[
\bal
| H \eta_{\al}(x)| &= | \int_0^x \pa_x H \eta_{\al}(y) dy  |
\leq || \pa_x H \eta_{\al}||_{L^4} \B(\int_0^x 1 dy \B)^{3/4} 
\les || \pa_x H \eta_{\al}||_{L^4} x^{3/4}  \\
&= x^{3/4} || H \pa_x \eta_{\al}||_{L^4}
\les x^{3/4} || \pa_x \eta_{\al}||_{L^4} \les (1-\al) x^{3/4}.
\eal
\]

Since $D_x H \eta_{\al}$ vanishes on $x = 0, \pi$, using an estimate similar to the above,
%\eqref{eq:L4_2}, 
we yield 
\[
| D_x H \eta_{\al}(x)|
\les || \pa_x ( D_x H \eta_{\al}(x) )  ||_{L^4}  ( |x|^{3/4} \wedge |\pi - x|^{3/4} )
\les || \pa_x ( D_x H \eta_{\al}(x) )  ||_{L^4}  |\sin x|^{3/4}.
\]

Applying Lemma \ref{lem:com} ($n=2$), we yield 
\[
\pa_x( D_x H \eta_{\al}) = \pa_x( H( D_x \eta_{\al} ) - H(D_x \eta_{\al} ) (0) )
=\pa_x( H( D_x \eta_{\al} ) = H(\pa_x D_x \eta_{\al}).
\]

Applying \eqref{eq:L4} and the fact that $H$ is $L^4$ bounded, we establish 
\[
| D_x H \eta_{\al}(x)|
\les  ||  H(\pa_x D_x \eta_{\al} ) ||_{L^4} |\sin x|^{3/4} 
\les ||  \pa_x D_x \eta_{\al}  ||_{L^4} |\sin x|^{3/4} 
\les (1-\al )|\sin x|^{3/4} ,
\]
which implies the second inequality in \eqref{eq:est_ual}.
\end{proof}

The above $L^4$ estimate on $H\eta_{\al}$ can be replaced by $L^p$ estimates with larger $p$, which offers more vanishing order of $ H\eta_{\al}$ near $x=0$. Here, the power $|x|^{3/4}$ is sufficient for our later weighted energy estimates. 

% We remark that the above $L^4$ estimate on $H\eta_{\al}$ can be replaced by other estimates, e.g. $L^p$ estimates with larger $p$, which offers more vanishing order of $ H\eta_{\al}$ near $x=0$. Here, the power $|x|^{3/4}$ is sufficient for our later weighted energy estimates. 
 
\subsection{Nonlinear stability of the approximate steady state}

In this Section, we follow \cite{chen2019finite,chen2020slightly} to perform stability analysis around $(\om_{\al}, c_{\om,\al})$ constructed in \eqref{eq:wal}, \eqref{eq:normal_al} and establish the finite time blowup results. We first introduce some weighted norms and spaces.

\begin{definition}\label{def}
Define the singular weight $\rho =( \sin \f{x}{2})^{-2}$, the standard inner product 
$\la \cdot ,\cdot \ra$ on $S^1$, the weighted norms $||\cdot ||_{\cH}$ and the Hilbert spaces $\cH$ as follows
\beq\label{eq:Hnorm}
%\rho =( \sin \f{x}{2})^{-2} , \quad 
\la f, g \ra = \int_0^{2\pi} f g dx, \quad
|| f||_{\cH}^2 \teq \f{1}{4\pi} \int_{-\pi}^{\pi} \f{ |f_{x}|^2}{\sin^2 \f{x}{2} } d x ,
\quad \cH \teq \{ f| f(0 )=0, ||f||_{\cH } < +\infty \} 
\eeq
with inner products $\la \cdot, \cdot \ra_{\cH}$ induced by the $\cH$ norm. 
% and the Hilbert spaces $\cH, X$  
% \[
% \cH \teq \{ f| f(0 )=0, ||f||_{\cH } < +\infty \} 
% \]
% with inner products $\la \cdot, \cdot \ra_{\cH}, \la \cdot , \cdot \ra_X$ induced by the $\cH, X$ norm. 
\end{definition}

The $\cH$ norm was introduced in \cite{lei2019constantin} for the stability analysis of the De Gregorio model. By definition, we have
\beq\label{eq:L2_cH}
\la f, g \ra_{\cH} =(4\pi)^{-1} \la f_x , g_x \rho \ra.
\eeq

\subsubsection{ Linearized equation}

Linearizing \eqref{eq:DGdyn} 
%the dynamic rescaling equation \eqref{eq:DGdyn} 
around $\om_{\al}, c_{\om,\al}$, we obtain the equation for the perturbation $\om, c_{\om}$ 
( ($\om + \om_{\al}, c_{\om} + c_{\om,\al}$) is the solution of \eqref{eq:DGdyn})
\beq\label{eq:lin}
\bal
\om_t &= - u_{\al} \om_x +  u_{\al, x} \om
+ u_x \om_{\al} - u \om_{\al, x} 
 + c_{\om,\al} \om  + c_{\om} \om_{\al}
+ N(\om) + F(\om_{\al}) \\
&\teq \cL_{\al} \om + N(\om) + F(\om_{\al}) ,
\eal
\eeq
where the nonlinear term $N(\om)$ and error term $F(\om_{\al})$ are given by 
\beq\label{eq:NF}
N(\om) = (c_{\om} + u_x) \om - u \om_x,  \quad F(\om_{\al}) = (c_{\om, \al} + u_{\al, x}) \om_{\al}  - u_{\al } \om_{\al, x}.
\eeq
We choose the normalization condition on $c_{\om}$ according to \eqref{eq:normal}
\beq\label{eq:normal_per}
c_{\om} = (\al-1)  u_x(0).
\eeq

Under the  conditions \eqref{eq:normal}, \eqref{eq:normal_per}, it is easy to obtain that the slope of $\om / x^{\al}$ is fixed, i.e.
\[
 \lim_{x \to 0 }\f{\om(x, t) + \om_{\al}(x)}{x^{\al}} =  \lim_{x \to 0 } \f{\om(x, 0) + \om_{\al}(x)}{x^{\al}} , \quad 
 \lim_{x \to 0 }\f{\om(x, t)}{x^{\al}} =  \lim_{x \to 0 } \f{\om(x, 0) }{x^{\al}} .
 %, \quad  \lim_{x \to 0 }\f{\om(x, t)  }{x^{\al}} =  \lim_{x \to 0 } \f{\om(x, 0) {x^{\al}} .
\]
%Since $\lim_{x\to 0} \om_{\al} / x^{\al} = C$ for some constant, 
In particular, if the initial perturbation $\om_0(x)$ vanishes near $x=0$ with order higher than $x^{\al}$, e.g. $x^{2\al}$, the perturbation $\om(x, t)$ will also vanish near $x=0$ with higher order. This allows us to perform energy estimates on $\om$ with a singular weight near $x=0$.

We treat the linearized operator $\cL_{\al}$ as a perturbation to $\cL_1$ 
\[
\cL_1 \om = - u_1 \om_x +  u_{1, x} \om
+  u_x \om_1 - u \om_{1, x} 
= - \sin x \om_x + \cos x \om  - u_x \sin x + u \cos x ,
\]
where we have used the explicit formulas \eqref{eq:eta_al}, and perform the following decomposition 
\beq\label{eq:Lin_decomp}
\bal
\cL_{\al} \om &= \cL_1 \om  - ( u_{\al} -  u_{1} ) \om_x
+ ( u_{\al, x} - u_{1, x}) \om 
+ u_x (\om_{\al} - \om_1)  \\
& \quad - u( \om_{\al, x} - \om_{1, x})    + c_{\om, \al} \om + c_{\om} \om_{\al}  \\
& = \cL_1 \om - u(\eta_{\al}) \om_x
+ H \eta_{\al} \cdot \om 
 + u_x \eta_{\al} - u \eta_{\al, x} 
 + c_{\om, \al} \om + c_{\om} \om_{\al} 
 \teq \cL_1 \om + \cR_{\al} \om, \\
\eal
\eeq
where  $u(\eta_{\al})$ denotes the odd velocity $u$ with $u_x = H \eta_{\al}$. In fact, we have $u(\eta_{\al}) = -( -\pa_{xx})^{-1/2} \eta_{\al}$.

 The operator $\cL_1$ enjoys an important coercive estimate 
established in \cite{lei2019constantin}. The following slight modification of the result in \cite{lei2019constantin} is from \cite{chen2020slightly}.
%Our stability analysis is built on the work of Lei et. al. \cite{lei2019constantin}, in which the authors proved the following results.

\begin{lem}\label{lem:linop}
Suppose that $ f, g \in \cH$ and $\int_{S^1} f dx = 0$. Denote $e_0( x) = \cos x -1$ and 
\[
 f_e = \la f, e_0 \ra_{\cH}  , \quad  \la f, g \ra_Y  \teq \la f - f_e e_0, g - g_e e_0 \ra_{\cH}.
\]
We have :
(a) Equivalence of norms : $( \cH / \R \cdot e_0, \la \cdot , \cdot \ra_Y)$ is a Hilbert space and the induced norm $|| \cdot ||_Y$ satisfies $\f{1}{2} || f ||_{\cH} \leq  || f ||_{Y} \leq || f ||_{\cH} .$

(b) Orthogonality : $||e_0||_{\cH} = 1$ and 
\[
\la f -  f_e e_0 , e_0 \ra_{\cH} = 0, \quad || f ||^2_{\cH} = f_e^2 + || f||_Y^2.
\]
 %$\la f -  f_e e_0 , e_0 \ra_{\cH} = 0$, $ || e_0||_{\cH} = 1$ and $|| f ||^2_{\cH} = f_e^2 + || f||_Y^2$.

(c) Coercivity : $\la \cL_1 f , f \ra_{Y} \leq -\f{3}{8} ||f ||^2_Y$. 
\end{lem}

Using \eqref{eq:L2_cH} and the above result (b), we can represent $\la \cdot , \cdot \ra_Y$ as follows 
\beq\label{eq:YH}
\la f, g \ra_Y =\la f - f_e e_0, g \ra_{\cH} = (4\pi)^{-1} \la f_x + f_e \sin x , g_x \rho \ra,
%= (4\pi)^{-1}\la f_x , g_x \rho\ra+ (4\pi)^{-1} f_e \la 2 \cot \f{x}{2} , g_x \ra
\eeq
where we have used $\pa_x e_0 = -\sin x$.

%\subsubsection{ Energy estimates}
\subsubsection{ Weighted $H^1$ estimates}
We consider odd perturbation $\om$, which satisfies $\int_{S^1} \om dx = 0$. Recall the linearized equation \eqref{eq:lin} and the decomposition \eqref{eq:Lin_decomp}.
Performing energy estimate on $\la \om, \om \ra_Y$ yields 
%\eqref{eq:lin} in the norm $\la \om, \om \ra_Y$ yields 
\beq\label{eq:H1}
\f{1}{2}\f{d}{dt} \la \om, \om \ra_Y = \la \cL_1 \om , \om \ra_Y + 
\la \cR_{\al} \om, \om \ra_Y + \la N(\om) , \om \ra_Y
+ \la F(\om_{\al}), \om \ra_Y .
\eeq

The estimate of the first term $\la \cL_1 \om , \om \ra_Y $ follows from Lemma \ref{lem:linop}
\beq\label{eq:est_main}
\la \cL_1 \om , \om \ra_Y \leq -\f{3}{8} || \om||_Y^2.
\eeq

For the remainder $\cR_{\al}$ in \eqref{eq:Lin_decomp}, a direct calculation yields 
%\beq\label{eq:est_R0}
\[
\bal
\pa_x \cR_{\al}\om &= - u(\eta_{\al}) \om_{xx} 
+ \pa_x H \eta_{\al} \cdot \om + u_{xx} \eta_{\al} - u \eta_{\al, xx} + c_{\om,\al} \om_x 
+ c_{\om} \om_{a, x} \teq  - u(\eta_{\al}) \om_{xx}  + \cR_{\al, 2}\om .
\eal
\]
%\eeq

Applying \eqref{eq:YH}, we derive
\beq\label{eq:est_R}
\bal
&\la \cR_{\al} \om, \om \ra_Y 
=  (4\pi)^{-1}  \la \pa_x \cR_{\al} \om, (\om_x + \om_e \sin x) \rho \ra  \\
 =& (4\pi)^{-1} \la  - u(\eta_{\al}) \om_{xx}, (\om_x + \om_e \sin x) \rho \ra  
 +  (4\pi)^{-1} \la  \cR_{\al, 2}, (\om_x + \om_e \sin x) \rho \ra   \teq I + II.
%= I + II , \\
% &I \teq (4\pi)^{-1} \la  - u(\eta_{\al}) \om_{xx}, (\om_x + \om_e \sin x) \rho \ra  , \quad 
% II \teq (4\pi)^{-1} \la  \cR_{\al, 2}, (\om_x + \om_e \sin x) \rho \ra  
% %\les | \la -u(\eta_{\al})\om_{xx},  (\om_x + \om_e \sin x) \rho \ra |
%+ 
\eal
\eeq

Recall $\rho =( \sin \f{x}{2})^{-2}$. Since $\sin \f{x}{2} \asymp x$, we can essentially treat $\rho$ as $x^{-2}$. 
For $II$, it suffices to estimate $|| \cR_{\al, 2} \om \rho^{1/2}||_2$. Since 
$c_{\om} = (\al-1)u_x(0)$ \eqref{eq:normal_per}, we decompose $\cR_{\al, 2}$ as follows 
\beq\label{eq:est_R2}
\bal
\cR_{\al, 2}
&=\pa_x H \eta_{\al} \cdot \om + u_{xx} \eta_{\al} - (u -u_x(0)\sin x) \eta_{\al, xx}  \\
& \quad + u_x(0)( (\al-1) \om_{ \al, x} - \sin x \cdot \eta_{\al, xx})+ c_{\om,\al} \om_x .
%\teq II_1 + II_2 + II_3 + II_4  + II_5.
\eal
\eeq

Next, we estimate the $L^2(\rho)$ norm of each term. The main difficulty is the estimate of the nonlocal term, e.g.  $|| u_{xx} \eta_{\al} \rho^{1/2}||_2$,
%$\la u_{xx} \eta_{\al}, \om_x \rho \ra$ 
%appeared in the estimate of \eqref{eq:est_R} 
due to the singular weight $\rho$ near $x=0$ and that the profiles $\om_{\al}, \eta_{\al}$ are not smooth near $x=0, \pi$. Since $\eta_{\al} \rho^{1/2} \notin L^{\inf}$ 
 (see \eqref{eq:wal},\eqref{eq:eta_al}), we need to perform a weighted estimate on $u_{xx}$. It is based on the lemma below, which shows that the Hilbert transform commutes with $\f{1}{x}$ up to some lower order terms. 
%Compared to \cite{chen2020slightly}, the essential new difficulty is the term $\la u_{xx} \eta_{\al}, \om_x \rho \ra$ in \eqref{eq:est_R} due to the singular weight $\rho$. In \cite{chen2020slightly}, such term is estimated by applying $L^2$ estimate on $u_{xx}$ and $L^{\inf}$ estimate on a term similar to $\eta_{\al} / \sin \f{x}{2}$. However, here, such term is unbounded (see \eqref{eq:wal},\eqref{eq:eta_al}). To overcome this difficulty, 
%We need the following commutator estimate near $x=0$, whose proof is deferred to Appendix \ref{app:hil}
\begin{lem}\label{lem:com_sing}
Suppose that $ \f{f}{x} \in L^2([-\pi,\pi])$. We have 
\[
\B| \f{ H f - Hf(0)}{x} - H (\f{f}{x}) \B|\les \int_{-\pi}^{\pi} \B| \f{f(y)}{y} \B| dy .
\]
\end{lem}
The proof is deferred to Appendix \ref{app:hil}.
%The above Lemma shows that the Hilbert transform commutes with $\f{1}{x}$ up to some lower order terms, which is similar to the identity in Lemma \ref{lem:com}.  %We will apply this Lemma to $f = \om_x$. In this case, we can further control $\om_x x^{-1}$ since $\om \in \cH$. Note that we cannot apply Lemma \ref{lem:com} since we do not have a good control on $\f{\om_x}{\sin x}$ given $\om \in \cH$
%The above Lemma is similar to the identity in Lemma \ref{lem:com}. The advantage of the above Lemma is that the function $x$ only vanishes at $x=0$, while $\sin x$ in Lemma \ref{lem:com} vanishes at both $x=0$ and $x =\pi$. We will apply this Lemma to $f = \om_x$. In this case, if we apply Lemma \ref{lem:com}, we do not have control on $\f{\om_x}{\sin x}$ since the weight in the norm $\cH$ \eqref{eq:Hnorm} is strong enough near $x=\pi$.
Since $u, \om$ are odd, we get $u_{xx}(0) = 0$. Applying the above Lemma with $f= u_x$ and using the fact that $H$ is $L^2$ bounded, we yield 
\beq\label{eq:uxx_sing}
|| \f{u_{xx}}{x} ||_{L^2}
= || \f{H \om_x}{x} ||_{L^2}
\les ||  H( \f{\om_x}{x} )||_{L^2} + || \f{\om_x}{x}||_{L^1}
\les || \f{\om_x}{x}||_{L^2} \les  || \om||_{\cH}.
\eeq
%We remark that we cannot apply Lemma \ref{lem:com} to estimate the weighted norm of $u_{xx}$ since such estimate would lead to $||\om_x (\sin x)^{-1}||_2$, which cannot be bounded by $|| \om||_{\cH}$ due to the singularity $(\sin x)^{-1}$ near $x = \pi$.
 %We will apply this Lemma to $f = \om_x$. In this case, we can further control $\om_x x^{-1}$ since $\om \in \cH$. Note that we cannot apply Lemma \ref{lem:com} since we do not have a good control on $\f{\om_x}{\sin x}$ given $\om \in \cH$

Applying \eqref{eq:est_wal1} in Lemma \ref{lem:profile}, we obtain 
%\beq\label{eq:est_R21}
\[
|| u_{xx}  \eta_{\al} \rho^{1/2}||_{L^2}
\les || u_{xx} x^{-1} ||_{L^2} || \eta_{\al}||_{L^{\inf}}
\les (1-\al) || \om ||_{\cH}.
\]
%\eeq

Denote $\td u =u-u_x(0) \sin x $. Next we estimate $|| \td u \eta_{\al, xx} \rho^{1/2}||_2$.  
From \eqref{eq:wal} and \eqref{eq:est_wal1}, $\eta_{\al, xx}$ is similar to $|\sin x|^{\al-2}$, which is singular both at $x=0, \pi$. To overcome the singularities from $\eta_{\al, xx}$ and $\rho^{1/2}$, we estimate $\td u (\sin x)^{-1} x^{-1} $. For $ |x| \geq \f{\pi}{2} $, since $\td u(\pi) = 0$ and $|\sin x | \les |\pi - |x| |^{-1}$, we yield 
\[
|\td u (\sin x)^{-1} x^{-1}| \les | \td u |\pi - |x||^{-1} | \les || \pa_x \td u||_{\inf} \les || u_{xx}||_2 \les || \om||_{\cH}.
\]

For $|x|\leq \f{\pi}{2}$, since $\td u(0) = \pa_x \td u(0) = 0$, using integration by parts, we obtain 
\[
|\td u (\sin x)^{-1} x^{-1}| \les \f{|\td u|}{x^2} = \f{1}{x^2} \B| \int_0^{x} \pa_{yy} \td u(y) \cdot (x- y) dy \B|
\les \f{1}{x^2} || \pa_{yy} \td u \cdot y^{-1} ||_2 \B(\int_0^x y^2(x-y)^2 dy \B)^{1/2}.
\]

Since $\pa_{yy} \td u(y) = \pa_{yy} u + u_x(0) \sin y$, using \eqref{eq:uxx_sing}, we derive 
\[
|\td u (\sin x)^{-1} x^{-1}| \les x^{-2} ( || u_{xx} x^{-1} ||_2 + |u_x(0)| ) x^{5/2} 
\les || \om||_{\cH}.
\]

Since $\rho^{1/2} = (\sin \f{x}{2})^{-1} \asymp x^{-1}$, applying the above estimate and \eqref{eq:est_wal1}, we obtain 
\[
\bal
|| (u - u_x(0)\sin x ) \eta_{\al, xx} \rho^{1/2}||_2
&\les || \td u (\sin x)^{-1} \rho^{1/2}||_{\inf} 
|| \eta_{\al, xx} \sin x||_2  \\
&\les (1-\al) || \om||_{\cH} || \ |\sin x |^{\kp_2 - 1} ||_2 \les (1-\al) || \om||_{\cH}.
\eal
\]

% The estimates of other terms in \eqref{eq:est_R2} and $I$ in \eqref{eq:est_R} are relatively simple. Recall that $\rho = (\sin \f{x}{2})^{-1/2}$ in Definition \ref{def}. 
% Since $\om, u-u_x(0)\sin x$ vanishes at $x = 0, \pi$, using the Hardy-type inequality in Lemma \ref{lem:hardy}, we yield 
% \[
% \bal
% &|| \om (\sin x)^{-1} \rho^{1/2} ||_2
% \les || \om x^{-2}||_2 + || \om | \pi- |x||^{-1} ||_2 
% \les || \om_x x^{-1}||_2 + || \om_x||_2  \les || \om_x x^{-1} ||_2
% \les || \om ||_{\cH} ,\\
% &|| (u -u_x(0) \sin x) (\sin x)^{-1} \rho^{1/2} ||_2
% \les   || ( u_x - u_x(0) \cos x ) x^{-1} ||_2 \\
% & \qquad \qquad \qquad \qquad \qquad \qquad \qquad \les  || u_{xx} + u_x(0) \sin x||_2
% \les || u_{xx}||_2 \les || \om_x||_2 \les || \om||_{\cH}.
% \eal
% \]

The estimates of other terms in \eqref{eq:est_R2} and $I$ in \eqref{eq:est_R} are relatively simple. Since $\om$ vanishes at $x = 0, \pi$, using the Hardy-type inequality in Lemma \ref{lem:hardy}, we yield 
\[
\bal
&|| \om (\sin x)^{-1} \rho^{1/2} ||_2
\les || \om x^{-2}||_2 + || \om | \pi- |x||^{-1} ||_2 
\les || \om_x x^{-1}||_2 + || \om_x||_2  \les || \om_x x^{-1} ||_2
\les || \om ||_{\cH} .\\
\eal
\]

Applying the above estimates and \eqref{eq:est_ual} in Lemma \ref{lem:profile}, we obtain 
\[
\bal
&||\pa_x H \eta_{\al} \cdot  \om \rho^{1/2}||_2 
\les  ||\om (\sin x)^{-1} \rho^{1/2}||_2 || \sin(x)  \pa_x H \eta_{\al}||_{L^{\inf}} \les (1-\al) || \om ||_{\cH} .\\
% &|| (u - u_x(0) \sin x) \eta_{\al, xx} \rho^{1/2}||_2 
% \les || (u - u_x(0) \sin x) (\sin x)^{-1} \rho^{1/2}||_2  || \eta_{\al, xx} \sin x||_{L^{\inf}} 
% \les (1-\al) || \om ||_{\cH}.
\eal
\]
%\eeq

Applying \eqref{eq:est_wal2} in Lemma \ref{lem:profile} and $(1-\al) \wedge x^2 \les (1-\al)^{1/2} |x|$, we yield 
\[
\bal
 || u_x(0)( (\al-1) \om_{ \al, x} - \sin x \cdot \eta_{\al, xx}) \rho^{1/2}||_2
\les &  ||u_x||_{L^{\inf}} 
||   (1-\al) \wedge |x|^2 ) |\sin x|^{\al-1} x^{-1} ||_2 \\
 \les & || u_{xx}||_2 (1-\al )^{1/2} ||  \ |\sin x|^{\al-1}||_2 \\
  \les & (1-\al)^{1/2} || \om_x||_2 \les  (1-\al)^{1/2} || \om||_{\cH}.
\eal
\]
% Using $(1-\al) \wedge x^2 \les (1-\al)^{1/2}|x|$ and $|| u_x||_2 \les || u_{xx}||_2 \les || \om_x||_2 \les || \om ||_{\cH}$, we derive 
% \[
% J \les ||\om||_{\cH} (1-\al)^{1/2} || (\sin x)^{\al- 1}||_2 \les  (1-\al)^{1/2}||\om||_{\cH}.
% \]

Recall $c_{\om,\al} = (\al-1) $ from \eqref{eq:normal_al}. The estimate of the last term in \eqref{eq:est_R2} is trivial 
\[
|| c_{\om, \al} \om_x \rho^{1/2}||_2 \les (1-\al) || \om||_{\cH}.
\]

Combining the above $L^2(\rho)$ estimates of each term in \eqref{eq:est_R2}, we establish 
 %of each term in \eqref{eq:est_R2} 
%and applying result (b) in Lemma \ref{lem:linop}, 
\beq\label{eq:est_RII}
|II | \les || \cR_{\al, 2}\rho^{1/2}||_2 || (\om_x + \om_e \sin x )\rho^{1/2} ||_2
\les (1-\al)^{ \f{1}{2}} || \om||_{\cH}
( || \om||_{\cH} + | \om_e|  )
\les (1-\al)^{ \f{1}{2}} || \om||^2_{\cH},
\eeq
where we have applied $|\om_e| \les || \om||_{\cH} $ from Lemma \ref{lem:linop} in the last inequality.

Next, we estimate the term $I$ from \eqref{eq:est_R}. Applying integration by parts, we yield 
\[
\bal
I_1  &\teq \la - u(\eta_{\al}) \om_{xx},  \ \om_x \rho \ra 
= \la  - u(\eta_{\al})\rho, \ \f{1}{2}\pa_x (\om_x)^2 \ra 
= \f{1}{2}\la \pa_x( u(\eta_{\al}) \rho) \rho^{-1}, \ \om_x^2 \rho \ra, \\
I_2 &\teq \la - u(\eta_{\al})\om_{xx}, \ \om_e \sin x \cdot \rho \ra
= \om_e \la \pa_x( u(\eta_{\al}) \rho \cdot \sin x  ), \ \om_x \ra .
%& \les || \pa_x( u(\eta_{\al}) \rho) \rho^{-1}||_{\inf} || \om_x \rho^{1/2}||_2^2\les ||  |u_x(\eta_{\al})| +\B| \f{u(\eta_{\al})}{x} \B| ||_{L^{\inf}} || \om ||_{\cH}
\eal
\]
Since $\rho = (\sin \f{x}{2})^{-2}, |\pa_x \rho | \les  \rho |x|^{-1} $, and $\pa_x u(\eta_{\al}) = H \eta_{\al}$, applying \eqref{eq:est_ual}, we derive 
%\beq\label{eq:est_RI}
\[
\bal
|\pa_x( u(\eta_{\al}) \rho) |
&\les ( |  \pa_x u(\eta_{\al}) | + \B| \f{u(\eta_{\al})}{x} \B| ) \rho 
\les || \pa_x u(\eta_{\al}) ||_{\inf} \rho \les (1-\al) \rho  , \\
| \pa_x( u(\eta_{\al}) \rho \cdot \sin x  )|
&\les |u(\eta_{\al}) \rho | + |  x \pa_x( u(\eta_{\al}) \rho )|
\les  
|| \pa_x u (\eta_{\al} ) ||_{\inf} |x \rho|
+ (1-\al) |x|\rho \les (1-\al) |x|\rho .
\eal
\]

Using the above estimate and the result (b) in Lemma \ref{lem:linop}, we establish 
\beq\label{eq:est_RI}
\bal
|I_1| 
 &\les || \pa_x( u(\eta_{\al}) \rho) \rho^{-1}||_{\inf} || \om_x \rho^{1/2}||_2^2\les
(1-\al) || \om||^2_{\cH} ,\\
|I_2|  &\les (1-\al)  |\om_e| \cdot || \om_x x \rho ||_{L^1}
\les (1-\al)  ||\om||_{\cH} || \om_x \rho^{1/2}||_{L^1} \les (1-\al)  || \om||^2_{\cH} .
\eal
\eeq

Plugging the estimates \eqref{eq:est_RII} and \eqref{eq:est_RI} in \eqref{eq:est_R} and then applying Lemma \ref{lem:linop}, we obtain 
\beq\label{eq:est_Rdone}
|\la \cR_{\al} \om, \om \ra_Y | \les  (1-\al)^{1/2} || \om||_{\cH}^2 
\les  (1-\al)^{1/2} || \om||_{Y}^2  .
\eeq

\subsubsection{Estimates of nonlinear and error terms}

Recall the nonlinear term $N(\om)$ and error term $F(\om_{\al})$ from \eqref{eq:NF}. Since 
%the nonlinear term 
$N(\om)$ is similar to that in \cite{lei2019constantin,chen2020slightly} and the perturbation $\om$ lies in the same space $\cH$, the estimate of $N(\om)$ is almost identical to that in \cite{lei2019constantin,chen2020slightly}. In particular, we yield 
\beq\label{eq:est_N}
|\la N(\om), \om \ra_Y| \les || \om||_{\cH}^3 \les || \om||_{Y}^3
\eeq
and refer the detailed estimates to  \cite{lei2019constantin,chen2020slightly}.

In the following derivation, we use the implicit notation $ O(f)$ to denote some term $g$ that satisfies $|g| \les f$. It can vary from line to line. Due to symmetry, we focus on $x \in [0, \pi]$. 

 For the error term $F(\om_{\al})$, we first compute $\pa_x F(\om_{\al})$ %and then perform a decomposition 
\beq\label{eq:est_F}
\pa_x F(\om_{\al}) 
=  u_{\al, xx} \om_{\al} - u_{\al} \om_{\al, xx}
+ c_{\om,\al} \om_{\al, x} .
%=  u_{\al, xx} \om_{\al}-   u_{\al} \eta_{\al, xx} - u_{\al} \om_{1,xx}  + c_{\om,\al} \om_{\al, x} 
%=  ( u_{1,xx} + \pa_x H \eta_x )  \om_{\al}
\eeq

Recall $u_1 = \sin x, \om_1 = - \sin x, \eta_{\al} = \om_{\al} -  \om_1 $ from \eqref{eq:eta_al}, and $u_{\al,x} - u_{1,x } = H \eta_{\al}$. Applying Lemma \ref{lem:profile} and $|\om_{\al}| \les |\sin x|^{\al}$ \eqref{eq:wal}, we yield 
\beq\label{eq:est_F1}
\bal
 u_{\al, xx} \om_{\al}
 &= (u_{1,xx} + \pa_x H \eta_{\al, x} ) \om_{\al}
 = u_{1,xx} \om_{\al} + O( (1-\al) |\sin x|^{\kp_1- 1 + \al} ) \\
 & =u_{1,xx}  \om_{1} - \sin x \cdot  \eta_{\al} + O( (1-\al) |\sin x|^{\kp_1- 1 + \al} )  \\
 &= (\sin x)^2 + O( (1-\al) |\sin x|^{\kp_1- 1 + \al} )  .\\
 \eal
\eeq

We decompose the second term in \eqref{eq:est_F} as follows 
\beq\label{eq:est_F2}
\bal
u_{\al} \om_{\al, xx} & = u_{\al} \eta_{\al, xx} + u_{\al} \om_{1,xx}
= (u_{\al} - \sin x) \eta_{\al, xx} +  \sin x \cdot \eta_{\al, xx} 
+  u_{\al} \om_{1,xx}  \\
& \teq I_1 + I_2 + I_3.
\eal
\eeq

Using \eqref{eq:est_ual}, we yield 
%Using the second estimate in \eqref{eq:est_ual}, we yield 
\[
|u_{\al,xx}| \les |u_{1,xx}| + |\pa_x H \eta_{\al}| \les |\sin x|^{\kp_1 - 1}, \quad 
|u_{\al} - \sin x| \les  (|| u_x(\eta_{\al}) ||_{\inf} + 1 ) |\sin x | \les |\sin x|.
\]

Recall $u_{\al, x}(0) = 1$ from \eqref{eq:wal}. For $0 \leq x \leq \f{\pi}{2}$, the above estimate implies 
\[
\bal
|u_{\al} -  \sin x|
& \leq |u_{\al} -  x | + C |x|^3
%O(x^3)
=  \B| \int_0^x ( u_{\al, x}(x) - u_{\al,x}(0) ) dx \B| + C|x|^3 \\
&= \B| \int_0^x   u_{\al,xx}(y)  \cdot (x- y) dy \B| + C |x|^3 
 \les \int_0^x y^{\kp_1 - 1} (x-y) dy + C|x|^3
%= x^{\kp_1 + 1} \int_0^1 s^{\kp_1 - 1} (1 - s) ds + C|x|^3 
\les |x|^{\kp_1 + 1 }.
\eal
\]
%where we have performed a change of variable $y = x s $. 
Therefore, we yield 
\[
|u_{\al} - \sin x| \les 
\one_{x \leq \pi/2}  |x|^{\kp_1 + 1}  + \one_{x>\pi/2} |\sin x|
%\min( |x|^{\kp_1 + 1} , |\sin x|) 
\les |\sin x| \cdot  |x|^{\kp_1},
\]
which along with \eqref{eq:est_wal1} imply the estimate of $I_1$ in \eqref{eq:est_F2}
\[
|I_1 | \les (1-\al) |\sin x|^{\kp_2 - 1} |x|^{\kp_1}.
\]

For $I_3$ in \eqref{eq:est_F2}, applying \eqref{eq:est_ual} and $u_1 = \sin x, \om_1 = -\sin x$, we get
\[
I_3 = u_1 \om_{1,xx} + (u_{\al} - u_1) \om_{1,xx}
 = (\sin x )^2 + O( |\sin x|^2 || u_{\al, x}||_{\inf}  )
 = (\sin x )^2 + O( (1-\al)|\sin x|^2 ).
\]

Recall $c_{\om,\al} = \al-1$ from \eqref{eq:normal_al}. We combine $I_2$ in \eqref{eq:est_F2} and $c_{\om,\al} \om_{\al, x}$ in \eqref{eq:est_F} and then apply \eqref{eq:est_wal2} to obtain  
\[
|   c_{\om,\al} \om_{\al, x} -I_2|
 = | (\al-1) \om_{\al, x} - \sin x \cdot \eta_{\al, xx} |
 \les ( (1-\al) \wedge |x|^2 ) |\sin x|^{ \al - 1}
 \les  (1-\al)^{1/2} |x| \cdot |\sin x|^{ \al - 1} .
\]
%+ ( u_1 + u(\eta_{\al}) ) \om_{1,xx}

Plugging the above estimates on $I_i$ and $c_{\om, \al}$ in \eqref{eq:est_F2}, we establish 
\beq\label{eq:est_F3}
u_{\al} \om_{\al, xx} - c_{\om,\al} \om_{\al, x}
=  I_1 + I_3 + ( I_2 -  c_{\om,\al} \om_{\al, x} )
= (\sin x)^2 
+ O( (1-\al)^{1/2} |x|^{\kp_1} | \sin x|^{\kp_2-1} ),
\eeq
where we have used $|\sin x| \leq |\sin x|^{\kp_2- 1}, |\sin x| \les |x| \les 1$ and $\kp_2 < \al$ to combine the estimates of $I_i$ in the last estimate. 

Recall $\kp_1 = \f{3}{4}, \kp_2 = \f{7}{8}$ from Lemma \ref{lem:profile}. Combining \eqref{eq:est_F}, \eqref{eq:est_F1} and \eqref{eq:est_F3}, we establish
\[
\bal
\pa_x F(\om_{\al})
&= (\sin x)^2 \cdot (1 - 1) 
+  O( (1-\al) |\sin x|^{\kp_1- 1 + \al} )  +
O (1-\al)^{1/2} |x|^{\kp_1} | \sin x|^{\kp_2-1}  \\
&=(1-\al)^{1/2} |\sin x|^{\kp_2-1} |x|^{\kp_1},  
\eal
\]
where we have used $| \sin x|^{\kp_1 + \al - \kp_2} \les |\sin x|^{\kp_1} \les |x|^{\kp_1}$
%$|\sin x|^{\kp_1}  \les |x|^{\kp_1}$ and $\kp_2 < \al$ 
to obtain the last estimate. Using the above estimate and Lemma \ref{lem:linop}, we prove 
\beq\label{eq:est_Fdone}
\bal
|\la F(\om_{\al}), \om \ra_Y|
&\les || F(\om_{\al}) ||_Y ||\om||_Y
\les || \pa_x F(\om_{\al}) \rho^{1/2} ||_2 || \om||_Y \\
&\les (1-\al)^{1/2} ||\  |\sin x|^{\kp_2-1} |x|^{\kp_1-1} ||_2 || \om||_Y
\les (1-\al)^{1/2} ||\om||_Y,
\eal
\eeq
where the integral is bounded since $2\kp_2 - 2 = - \f{1}{4} > -1,  2\kp_2 + 2 \kp_1 - 4
= - \f{3}{4} > -1$.

\subsubsection{Nonlinear stability and finite time blowup}\label{sec:non_stab}

Combining \eqref{eq:est_main}, \eqref{eq:est_Rdone}, \eqref{eq:est_N} and \eqref{eq:est_Fdone}, we establish the following nonlinear estimate for some absolute constant $C>0$
\[
\f{1}{2}\f{d}{dt}  || \om||_Y^2 
\leq - ( \f{3}{8}  - C|1-\al|^{1/2} )  || \om ||_Y^2
+ C |1-\al|^{1/2} || \om ||_Y + C|| \om ||_Y^3.
\]
Therefore, there exist absolute constants $\al_0 < 1$ sufficiently close to $1$ and  $\mu >0$, such that for any $\al \in (\al_0, 1)$, if the initial perturbation satisfies $|| \om_0 ||_Y <  \mu |1-\al|^{1/2}$, then 
%\beq\label{eq:a-priori}
\[
%\bal
|| \om(t)||_Y < \mu |1-\al|^{1/2},  \quad c_{\om, \al } + c_{\om}(t) = (\al - 1)( 1 + u_x(0))
\leq (\al-1) (1 - C|\al-1|^{1/2}) \leq \f{1}{2} (\al - 1)
\]
%\eal
%\eeq
holds true for all $t> 0$. Since the weight $\rho = O(1)$ near $x=\pi$ and $ ( \pa_x \om_{\al})^2 \rho $ is integrable near $x = \pi$, we can choose initial perturbation $\om_0$ such that $|| \om_0||_Y < \mu |1-\al|^{1/2}$, $\om_0 \in C^{2}( (-\pi/3, \pi/ 3))$ and $\om_0 + \om_{\al} \in C^{\al} \cap C^{\inf} (S^1 \bsh \{0 \} )$. For example, $\om_0$ can be $-\om_{\al}$
near $x=\pi$, $\om_0 = 0$  near $x=0$ and smooth in the intermediate region. A simple Lemma \ref{lem:sobolev} shows that $ \om_0 + \om_{\al} \in H^s$ for any $s < \al + \f{1}{2}$, and a direct calculation gives $\int_0^{\pi} | \sin x \cdot f_x^2 / f | dx <+\inf$ where $f =\om_0 + \om_{\al}$. Using the rescaling argument in Section \ref{sec:dyn}, we establish finite time blowup of \eqref{eq:DG} from $\om_0 + \om_{\al}$. 
%We refer more detailed argument to \cite{chen2019finite,chen2020slightly}.

The condition $\int_0^T u_{phy,x}(0, t) dt = \infty $ in Theorem \ref{thm:blowup}, where $u_{phy}$ is the velocity in \eqref{eq:DG}, follows from Theorem \ref{thm:criterion} or a calculation using the above a-priori estimates on the perturbation and the rescaling relations \eqref{eq:rescal1}-\eqref{eq:rescal3}. Due to the inclusion  $ C^{\al} \subset C^{\al_1} , H^s \subset H^{s_1} $ for $0 <\al_1  < \al, s_1 < s$, we conclude the proof of Theorem \ref{thm:blowup}.

\section{Concluding remarks}\label{sec:conclude}

%In this paper, %we studied the regularity conjecture of the De Gregorio model \eqref{eq:DG} for initial data in the class $X$ \eqref{eq:X}. 
We have constructed a finite time blowup solution of the De Gregorio model \eqref{eq:DG} from $C^{\al}$ initial data for any $0<\al< 1$, and established the global well-posedness (GWP) from initial data $\om_0 \in H^1 \cap X$ with $\om_0(x) x^{-1} \in L^{\infty}$, based on a one-point blowup criterion. These results verified the conjecture on global regularity of the DG model on $S^1$ for smooth data in $X$, and showed that the advection can prevent singularity formation if the initial data is smooth enough. 

Our analysis provides valuable insights on the global well-posedness of \eqref{eq:DG} with more general data, and it is likely that some results are generalizable. A potential direction is to generalize the one-point blowup criterion to a finite-points version. For simplicity, we assume that the number of zeros of $\om(x,t)$ is finite, and the zeros are $x_i(t), i=1,2,..,n$ with $\pa_x \om( x_i(t), t) \neq 0$. 
It is shown in \cite{Sve19} that the number $n$ and $\pa_x \om( x_i(t), t), i=1,2.,,.n$ are conserved. Denote $N_{\pm}(t) \teq \{ x: \om(x,t) = 0, \ \sgn( \om_x(x,t)) = \pm 1 \} $. A natural generalization of Theorem \ref{thm:criterion} is that %under suitable assumption on $\om_0$, 
the solution of \eqref{eq:DG} cannot be extended beyond $T$ if and only if 
\beq\label{eq:Mpoint}
\int_0^T \sum_{ x \in N_-(t)} |u_x( x , t)| d t = \infty. 
\eeq
A weaker version is that $\sum_{ i=1}^n |u_x( x_i(t), t)|$ controls the breakdown of the solution. These blowup criteria are consistent with that of the CLM model. See the discussion in Section \ref{sec:CLM}. We believe that these criteria are important for the GWP from general smooth initial data. 

Passing from \eqref{eq:Mpoint} to the GWP, a possible approach is to estimate functionals and  quadratic forms similar to those in Section \ref{sec:regular} in suitable moving frames. We remark that our proof of Lemma \ref{lem:comp} does not require the assumption on the sign of $\om$. Thus, it is conceivable that the argument can be adapted to study other scenarios. 

Our analysis has benefited from the property that the zeros of $\om$ with  $\om \in X$ \eqref{eq:X} are essentially fixed. For more general data, controlling the locations of the zeros of $\om$ can be a challenging problem.

For the gCLM model on a circle with a parameter $a>1$ and $\om_0 \in C^{\inf} \cap X$, 
%following the argument in Sections \ref{sec:regular}, \ref{sec:GWP}, we can study 
monotonicity of $\int_0^{\pi/2} | \om(y)| (\cot y)^{\b} dy$ with $\b = \b(a) < 2$ and a-priori estimates of $|| \om(t)||_{L^1}, u_x(0, t)$ can be studied by the argument in Sections \ref{sec:regular}, \ref{sec:GWP}.
%following the argument in Sections \ref{sec:regular}, \ref{sec:GWP}, we can study the conservation of $\int_0^{\pi/2} | \om(y)| (\cot y)^{\b} dy$ with $\b = \b(a) < 2$ and a-priori estimates of $|| \om(t)||_{L^1}, u_x(0, t)$. 
 These a-priori estimates shed some helpful light on the regularity of the gCLM model with $\om_0 \in C^{\inf} \cap X$.
Note that for $a> a_0$ with $a_0 \approx 1.05$, these estimates have been established in the arXiv version of \cite{chen2019finite}. 

% for $a> a_0$ with $a_0 \approx 1.05$. %These a-priori estimates 
%They shed some useful light on the regularity of the gCLM model. 

%Some estimates in Sections \ref{sec:regular}, \ref{sec:GWP} can be applied to study the 
%Estimates in Sections \ref{sec:regular}, \ref{sec:GWP} shed some useful light on theregularity of the gCLM model with parameter $a>1$ and $\om_0 \in X$. For $a>1$, following similar argument, we can study the conservation of $\int_0^{\pi/2} \om(y) (\cot y)^{\b}$ with $\b = \b(a) < 2$ and establish a-priori estimate of $|| \om(t)||_{L^1}, u_x(0, t)$.

\vspace{0.2in}
\noindent
{\bf Acknowledgments.} 
JC is grateful to Vladimir Sverak for introducing the De Gregorio model at the AIM Square. He would like to thank Yao Yao for the discussion at the AIM square on the potential blowup criterion for the gCLM model, which inspired Section \ref{sec:CLM}. He also acknowledges the support from AIM. He is also grateful to Thomas Hou for valuable comments on an earlier version of this work. He would also like to thank the referee for the constructive comments on the original manuscript and a question that inspires the author to weaken the regularity assumption in Theorem \ref{thm:GWP} in the original manuscript.
%which improve the quality of our paper. 
This research was supported in part by grants DMS-1907977 and DMS-1912654 from the National Science Foundation.

\appendix

\section{}

\subsection{Properties of the Hilbert transform and functional inequalities}\label{app:hil}

The following Cotlar's identity for the Hilbert transform is well known, see, e.g. \cite{duoandikoetxea2001fourier,chen2019finite,Elg17}.
\begin{lem} \label{lem:tri}
For $f \in C^{\inf}(S^1)$, we have 
\[
H (f H f ) = \f{1}{2} ( (Hf)^2  - f^2).
\]
\end{lem}

We have the following commutator identity from Lemma 2.6 in \cite{chen2020slightly}.
\begin{lem}\label{lem:com}
For $f \in H^1(S^1)$ with  period $n \pi$, we have 
\[
H(  \sin (\f{2x }{n}) f_x ) - \sin( \f{2x}{n}) H f_x = - \f{2}{n^2 \pi } \int f \sin(2 y) dy = H( \sin ( \f{2x}{n}) f_x)(0).
\]
\end{lem}

The case $n=2$ is proved in \cite{chen2020slightly}. The general case follows by a rescaling argument.

We use the following important Lemma to establish the energy estimate in Section \ref{sec:criterion}.
\begin{lem}\label{lem:cancel}
Suppose that $\om \in H^1$ is $\pi$-periodic and odd. We have $\int_{S^1} \om_x H \om_x \cdot \sin(2x) dx = 0 $.
\end{lem}

%\subsubsection{Proof of \eqref{eq:cancel1}}
\begin{proof}
We prove the identity for smooth function $\om \in C^{\inf}$, and the general case $\om \in H^1$ can be obtained by approximation. Applying Lemma \ref{lem:com} with $f = \om$ and $n=1$ yields
\[
S\teq \int_{S^1} \om_x H \om_x \cdot \sin(2x) dx  = \int_{S^1} \om_x \B( H( \sin(2x) \om_x) - H( \sin(2x) \om_x)(0) \B) dx
= \int_{S^1} \om_x  H( \sin(2x) \om_x)  dx.
%= \int_{S^1} \f{1}{\sin(2x)} \sin(2x) \om_x  H( \sin(2x) \om_x)  dx
\]

Denote $f = \sin (2 x) \om_x$. Using $\f{1}{\sin (2x)} = \f{1}{2} (\tan x + \cot x) = \f{1}{2} (\cot(\f{\pi}{2} - x) + \cot(x))$, \eqref{eq:hil} and Lemma \ref{lem:tri}, we obtain 
\[
\bal
S &= \f{1}{2} \int_{S^1}  (\cot(\f{\pi}{2} - x) + \cot(x)) f \cdot Hf dx 
= \f{\pi}{2} \B( H ( f Hf)(\f{\pi}{2}) -   H ( f Hf)( 0) \B) \\
&= \f{\pi}{4} \B( (Hf)^2(\f{\pi}{2}) - f^2(\f{\pi}{2}) -  ( Hf)^2(0) - f^2(0))  \B).
\eal
\]

Since $\om \in C^{\inf}$ and it is odd, we get $f(0) = f(\f{\pi}{2}) = 0$. Note that 
\[
Hf(\f{\pi}{2}) - Hf(0 )
= \f{1}{\pi} \int_{S^1} \B( \cot( \f{\pi}{2} - x) + \cot x \B) \sin(2x) \om_x dx
= \f{1}{\pi}\int_{S^1} \f{2 }{ \sin(2x) } \sin(2x) \om_x dx =0.
\]

We obtain $S = 0$ and establish the desired result.
\end{proof}

We use the following simple Lemma from \cite{chen2019finite2} to estimate the profile in Section \ref{sec:blowup}.
\begin{lem}\label{lem:1}
For $x \in [0,1]$, $\al , \lam >0$, we have
\[
(1 - x^{\al} ) x^{\lam} \leq \f{\al}{\lam}.
\]
\end{lem}

\begin{proof}
For the sake of completeness, we present the proof. Using Young's inequality, we prove 
\[
(1 - x^{\al} ) x^{\lam} 
= \f{\al}{\lam}  \cdot ( \f{\lam}{\al} (1 - x^{\al}) )  (x^{\al})^{\lam / \al}
\leq \f{\al}{\lam}  \B( \f{ \f{\lam}{\al}( 1 - x^{\al}) + \f{\lam}{\al} x^{\al} }{ 1 + \lam / \al} \B)^{\lam/\al + 1}
= \f{\al}{\lam}  \B( \f{ \lam / \al}{ 1 + \lam / \al} \B)^{\lam/\al + 1}\leq \f{\al}{ \lam}.
\]
\end{proof}

We have the following Hardy-type inequality \cite{hardy1952inequalities} in bounded domain.
\begin{lem}\label{lem:hardy}
For $p > 1$ and $L > 0$, suppose that $f x^{-p/2}, f_x x^{-p/2 + 1} \in L^2( [0, L])$. We have 
\[
\int_0^L \f{f^2}{x^{p}} dx \les_p \int_0^L \f{f_x^2}{x^{p-2}} dx.
\]

\end{lem}

It can be proved by applying an integration by parts argument. A proof can be founded in the Supplementary material of \cite{chen2021HL}. 

Next, we prove the commutator-type Lemma \ref{lem:com_sing}.

\begin{proof}[Proof of Lemma \ref{lem:com_sing}]
A direct calculation yields
\[
\bal
S &\teq \f{1}{x}  (H f - Hf(0)) - H (\f{f}{x})
= \f{1}{2\pi} \int_{-\pi}^{\pi} ( \f{1}{x} \cot \f{x-y}{2}  
 +   \f{1}{x}  \cot \f{y}{2} - \f{1}{y} \cot \f{x-y}{2} 
  )  f (y) dy  \\
  & = \f{1}{2\pi} \int_{-\pi}^{\pi} \f{1}{x} ( y \cot \f{y}{2} - (x-y) \cot \f{x-y}{2} )  \f{f(y)}{y} dy = \f{1}{2\pi} \int_{-\pi}^{\pi} \f{1}{x} ( g(y) - g(y-x))  \f{f(y)}{y} dy,
  \eal
\]
where $g(z) = z \cot \f{z}{2}$ and it satisfies $g(z) = g(-z)$. Since $g$ is Lipschitz on $[-3\pi/2, 3\pi/2]$
\[
|g^{ \prime}(z)|  = \B| \cot \f{z}{2} - \f{z}{2 ( \sin \f{z}{2})^2} \B|
= \f{ |\sin z - z|}{ 2 (\sin \f{z}{2})^2} \les \f{z^3}{z^2} \les 1,
\]
applying $ |g(y) -g(y-x)| \les |x|$, we prove the desired result.
\end{proof}

\subsection{Derivation of a model for 2D Boussinesq equations}\label{app:bous}

%In this subsection, 
We derive the model \eqref{eq:Bous2}-\eqref{eq:Bous3}, and discuss its connections with \eqref{eq:DG}. Recall the Boussinesq equations \eqref{eq:Bous} and \eqref{eq:bous20}
% Taking $x-$derivative on \eqref{eq:Bous} and using the incompressible condition $u_{2, y} = -u_{1,x}$, we yield 
%\beq\label{eq:bous20}
\[
\pa_t \th_x + \uu \cdot \na  \th_{x} = - u_{1,x} \th_x - u_{2,x}\th_y =
u_{2,y}\th_x -  u_{2,x}\th_y.
\]
%\eeq

Inspired by the anisotropic property of $\th$ in \cite{chen2019finite2}, i.e. $|\th_y| << |\th_x| $ near the origin, we drop the $\th_y$ term. To study the $y$-advection, we further drop the $x$-advection. Then we obtain \eqref{eq:Bous2}
%\beq\label{eq:Bous30}
\[
\pa_t \th_x + u_2  \pa_y \th_x = u_{2, y} \th_x.
\]
%\eeq
Since $\th_x$ is the forcing term in the $\om$ equation in \eqref{eq:Bous}, it leads to a strong alignment between $\th_x$ and $\om$. Thus, we simplify the $\om$-equation in \eqref{eq:Bous} by  $\om = \th_x$, which leads to the following Biot-Savart law in \eqref{eq:Bous3}
%we further assume $\om = \th_x$, which leads to the Biot-Savart law in \eqref{eq:Bous3}
\[
\uu = \na^{\perp} (-\D)^{-1} \th_x, \quad u_{2, y} = \pa_{xy}(-\D)^{-1} \th_x.
\]

% %Equation \eqref{eq:Bous30} captures the competition between the vortex stretching $u_{2,y}\th_x$ and the $y-$ advection in the Boussinesq equations \eqref{eq:bous20}. 
% This model relates to \eqref{eq:DG} via the connections $\th_x \to -\om, \pa_{xy}(-\D)^{-1} \to -H$. The velocities of two models $u_2$ and $u$ are related via $u_{2,y} = \pa_{xy}(-\D)^{-1} \th_x \approx - H ( -\om) = H \om = u_x$. Both solutions enjoy similar sign and symmetry properties. In fact, suppose that $\th_x$  satisfies the symmetry and sign properties in the hyperbolic flow scenario stated in Section \ref{sec:incomp}. The induced flow $u_2(x,y)$ is odd in $y$ with $u_2(x,y)  > 0$ in the first quadrant near $(0,0)$. The odd symmetries of $\th_x, u_2$ in $y$ are the same as those of $\om, u$ in \eqref{eq:DG} for the class $X$ \eqref{eq:X}. Moreover, for fixed $x>0$, $-\th_x(x, \cdot)$ and $\om$ satisfy similar sign conditions, and $u_2(x, \cdot)$ and $u$ satisfy similar sign conditions near the origin.

This model relates to \eqref{eq:DG} via the connections $\th_x \to -\om, \pa_{xy}(-\D)^{-1} \to -H$. The velocities of the two models $u_2$ and $u$ are related via $u_{2,y} = \pa_{xy}(-\D)^{-1} \th_x \approx - H ( -\om) = H \om = u_x$. 
Moreover, the solutions of the two models enjoy similar sign and symmetry properties. 
Suppose that $\th_x$  satisfies the sign and symmetry properties in the hyperbolic-flow scenario.
The induced flow $u_2(x,y)$ is odd in $y$ with $u_2(x,y)  > 0$ in the first quadrant near $(0,0)$. The odd symmetries of $\th_x, u_2$ in $y$ are the same as those of $\om, u$ in \eqref{eq:DG} for class $X$ \eqref{eq:X}. Moreover, for fixed $x>0$, $-\th_x(x, \cdot)$ and $\om$ satisfy similar sign conditions, and $u_2(x, \cdot)$ and $u$ satisfy similar sign conditions near the origin.

\subsection{Derivation of \eqref{eq:sym}-\eqref{eq:kernel} }\label{app:sym}

Recall the formulas of $u_x, u $ in \eqref{eq:hil} and the quadratic form in \eqref{eq:def_Q}. Using integration by parts, we obtain 
\[
B(\b) = \int_0^{\pi/2} (2 u_x \om - (u\om)_x ) \cot^{\b} x dx 
= 2 \int_0^{\pi/2} u_x \om \cot^{\b} x dx 
-\b \int_0^{\pi/2} u \om \cot^{\b-1} x \f{1}{\sin^2 x} dx \teq I + II.
\]
The boundary terms $u \om \cot^{\b} x \B|_0^{\pi/2}$ in the integration by parts vanish since $u(\pi/2) = 0$ and $u(x) = O(x), \om(x) = O( x^{ \g})$ with $\g > \b - 1$ near $x =0$ by the assumption in Lemma \ref{lem:comp}.

Since $\om$ is odd, using \eqref{eq:hil} and symmetrizing the kernel,  we yield 
\[
I= \f{2}{\pi}\int_0^{\pi/2} \om(x) \cot^{\b} x \int_0^{\pi/2} \om(y) ( \cot(x- y) - \cot( x+ y)) dy
= \f{1}{  \pi} \int_0^{\pi/2}\int_0^{\pi/2} \om(x) \om(y) P_1(x, y) dx dy, 
\]
where 
\[
P_1(x, y) = \cot^{\b} x  (\cot(x- y) - \cot( x+ y)) 
+ \cot^{\b} y( \cot(y-x) - \cot(x+y) ) .
\]

Recall $s =  \f{\cot x}{\cot y}$ in \eqref{eq:def_s}. We get $\cot x = s \cot y$. We expand $\cot(x-y), \cot(x+y)$ as follows 
\[
\bal
\cot(x- y) &= \f{ \cot x \cot y + 1}{ \cot y - \cot x} 
= \f{s \cot^2 y +1}{ \cot y \cdot (1-s)} ,  \quad \cot(x+ y) = \f{ \cot x \cot y - 1}{ \cot y + \cot x} 
= \f{s \cot^2 y -1}{ \cot y \cdot (1 + s)} .
\eal
\]
Thus, we obtain 
\[
\bal
 \cot(x- y) - \cot( x+ y)
 &=  \cot y ( \f{s}{1-s} - \f{s}{1+s} ) + \f{1}{\cot y} (\f{1}{1-s} + \f{1}{1+s})
 =  \cot y \f{2s^2}{1-s^2} + \f{1}{\cot y}\f{2}{1-s^2} ,\\
 \cot(y- x) - \cot( x+ y)
 &=  \cot y ( -\f{s}{1-s} - \f{s}{1+s} ) + \f{1}{\cot y} ( - \f{1}{1-s} + \f{1}{1+s}) \\
 &=  -\cot y \f{2s}{1-s^2} - \f{1}{\cot y}\f{2 s}{1-s^2}.
 \eal
\]

Using the above formulas and $\cot^{\b} x = s^{\b} \cot^{\b} y$, we yield 
\[
\bal
P_1 &= \cot^{\b} y \cdot s^{\b} ( \cot y \f{2s^2}{1-s^2} + \f{1}{\cot y}\f{2}{1-s^2} )
+ \cot^{\b} y ( -\cot y \f{2s}{1-s^2} - \f{1}{\cot y}\f{2 s}{1-s^2} ) \\
& = \cot^{ \b + 1} y  ( s^{\b+1} - 1) \f{2s}{ 1 - s^2}
+ \cot^{\b-1} y (s^{\b-1} - 1) \f{2s} {1 - s^2 }.
\eal
\]
We remark that $P_1 \leq 0$ since $\f{1- s^{\tau} }{1-s^2 } \geq 0$ for $s >0, \tau > 0$. 

For $II$, using \eqref{eq:hil}, we get 
\[
II = \f{\b}{\pi} \int_0^{\pi/2} \f{\om(x)}{\sin^2 x} \cot^{\b-1} x \int_0^{\pi/2} \om(y) \log \B| \f{\sin(x+y)}{\sin(x-y)} \B| dy
= \f{1}{\pi} \int_0^{\pi/2}\int_0^{\pi/2} P_2(x, y) \om(x) \om(y) dx dy,
\]
where 
\[
P_2 = \f{\b}{2}
\B(  \f{ \cot^{\b-1} x}{\sin^2 x}  + \f{ \cot^{\b-1} y}{\sin^2 y}   \B)
\log \B| \f{\sin(x+y)}{\sin(x-y)} \B| .
\]

Note that 
\[
\f{\cot^{\b-1} z}{\sin^2 z} = \cot^{\b-1} z + \cot^{\b+1} z, \quad 
%\cot^2 x + 1, 
\B| \f{\sin(x+y)}{\sin(x-y)} \B|
= \B|\f{\cot x + \cot y}{ \cot x - \cot y} \B|= \B| \f{1+s}{1-s} \B|.
\]
 We derive
\[
P_2(x, y) = \f{\b}{2} 
\B( \cot^{\b+1} y ( 1 +s^{\b+1})\log\B| \f{1+s}{1-s} \B| 
+
\cot^{\b-1} y ( 1 +s^{\b-1})\log\B| \f{1+s}{1-s} \B|  \B) .
\]
We remark that $P_2$ is positive. Combining the formulas of $P_1, P_2$, we derive 
\eqref{eq:sym}-\eqref{eq:kernel}.

\subsection{ Positive definiteness of the kernel}\label{app:pos_W}

In this subsection, we prove Lemmas \ref{lem:pos_W} and Lemma \ref{lem:pos_W2}, which are related to the positive definiteness of the kernel $K_{i,\b}$. We establish \eqref{eq:ver2} for $x_0 = \log \f{5}{3}$ in Appendix \ref{app:convex}.

%\subsection{Proof of Lemma \ref{lem:pos_W} }\label{app:conver}

\begin{proof}[ Proof of Lemma \ref{lem:pos_W}]

We show that there exists $\b_0\in (1,2)$, such that conditions \eqref{eq:ver10}-\eqref{eq:ver3} hold for $W = W_{1,\b}, G = G_{1,\b}$ with $ \b \in [\b_0, 2]$. Then using the same argument as that in Section \ref{sec:pos_W1}, we obtain $G_{1,\b}(\xi) \geq 0$ for all $\xi$ and $\b \in [\b_0, 2]$. 

Firstly, we impose $\b \in [1.9, 2]$. Recall $G_{j, \b}$ defined in \eqref{eq:def_G} 
\beq\label{eq:conv_W11}
G_{j, \b}(\xi) = \int_{0}^{\inf} W_{j,\b}(x) \cos(x \xi) dx,
\eeq
and $W_{1,\b}$ in  \eqref{eq:tK1}, \eqref{eq:kernelW}. Clearly, $W_{1,\b}(x)$ converges to $W_{1,2}(x)$ as $\b \to 2$ almost everywhere. Moreover, from the formula of $W_{1,\b}$ and the decay estimate \eqref{eq:decay_W1}, we have  
\beq\label{eq:conv_W12}
|W_{1,\b}(z) | \les \one_{|z| > 1} e^{- |z| / 4} + \one_{ |z| \leq 1} (1 +| \log |z| | ),
\eeq
where the term $\log |z|$ is due to the logarithm singularity $\log |s-1| = \log |e^z - 1|$ in \eqref{eq:tK1}. Thus, using dominated convergence theorem, we yield 
%that $G_{1,\b}(\xi) \to G_{1,2}(\xi)$ as $\b \to 2^-$.
\[
\lim_{\b\to 2^-} G_{1,\b}(\xi) = G_{1, 2}( \xi).
\]
Using \eqref{eq:conv_W11} and \eqref{eq:conv_W12}, we obtain that $G_{1,\b}(\xi)$ is equi-continuous 
\[
|\pa_{\xi} G_{j,\b}(\xi)|
\leq \int_0^{\inf} |W_{j,\b}(x)| |x| dx \les 1.
\]
Thus, we obtain that $G_{1,\b}(\xi)$ converges to $G_{1,2}(\xi)$ uniformly for $\xi \in [0, M]$, where $M$ is the parameter in Lemma \ref{lem:pos_W}.

For $x$ near $0$, from \eqref{eq:tK1} and \eqref{eq:kernelW}, we have 
\[
W_{1,\b}(x) = -\f{\b}{2} ( e^{ \f{\b+1}{2}x  }  + e^{ - \f{\b+1}{2} x}) \log| e^x - 1| + S_{\b}(x),
\]
where $S_{\b}(x)$ is smooth near $x=0$. Thus a direct calculation yields 
\beq\label{eq:conv_W13}
\bal
\pa_{xx} W_{1,\b}(x) 
&\geq -\f{\b}{2} ( e^{ \f{\b+1}{2}x  }  + e^{ - \f{\b+1}{2} x}) \pa_{xx}  \log| e^x - 1|
- \f{C}{|x|} \\
&\geq \f{\b}{2} ( e^{ \f{\b+1}{2}x  }  + e^{ - \f{\b+1}{2} x}) \f{e^x}{ (e^x-1)^2}
- \f{C}{ |x|}   \geq \f{\b}{x^2} - \f{C}{|x|}
\eal
\eeq
for some absolute constant $C > 0$ and $|x| < \f{1}{2}$. Therefore, there exists $\d > 0$, such that 
\beq\label{eq:conv_W14}
\pa_{xx} W_{1,\b}(x) > 0, \quad x \in [0,\d].
\eeq
%$\pa_{xx} W_{1,\b}(x) > 0$ for $x \in [0, \d]$. 
Note that  $W_{1,\b}(x) = \td K_{1,\b}( e^x)$ \eqref{eq:tK1} is smooth for $(\b, x ) \in [1.9 , 2] \times [\d, x_0]$,
where $x_0$ is the parameter in Lemma \ref{lem:pos_W}. We get that $\pa_{xx} W_{1,\b}(x) $ converges to $\pa_{xx} W_{1,2}(x)$ uniformly for $x  \in [\d, x_0] $ as $\b \to 2$, and that $ \pa_x W_{1, \b}(x_0 ) \to  \pa_x W_{1,2}(x_0)$ as $\b \to 2$.

Next, we consider the integral on $W^{\prime \prime \prime}$ in \eqref{eq:ver3}. We need the decay estimate of $W_{1,\b}^{\prime \prime \prime}$. For $r = e^{x_0} > 1$ and $s \geq r > 1$, performing Taylor expansion on $\log | \f{s+1}{s-1}|$ and $\f{1}{s^2 - 1}$, we obtain that the kernel $\td K_{1,\b}$ \eqref{eq:tK1} enjoys the expansion 
\beq\label{eq:conv_W142}
\td K_{1,\b} = \sum_{i \geq 1} a_i(\b) s^{-\al_i(\b)} , \quad |a_i(\b)| \les 1, \quad   \max( \f{\b-1}{2}, \f{i - 2}{10} )  \leq  \al_i(\b) \leq 10 (i+1).
\eeq
with $\al_i(\b)$ increasing. Since the expansions for $\log | \f{s+1}{s-1}|$ and $\f{1}{s^2-1}$ converge uniformly for $s \geq r > 1$, the above expansion also converges uniformly. Thus, we can exchange the summation and derivatives when we compute $\pa_x^k \td K_{1,\b}$. We are interested in the leading order term in the above expansion. It
%The leading order term 
decays at least $s^{- (\b-1)/2}$ since other terms in $\td K_{1,\b}$ that decay more slowly, such as $s^{ \f{\b-1}{2}}$, are canceled. Using $W_{1,\b}(x) = \td K_{1,\b}(e^x)$ and \eqref{eq:conv_W142}, for $x \geq x_0 > 0$, we yield 
\[
|\pa^3_{x}W_{1,\b} (x)|
= \B| \pa_x^3  \sum_{i\geq 1} a_i(\b) e^{- \al_i(\b) x} \B|
= \B|  \sum_{i\geq 1} a_i(\b) (-\al_i(\b))^3 e^{- \al_i(\b) x} \B|
\les e^{-  \f{\b-1}{2} x} \les e^{- x / 4},
\]
where the implicit constant can depend on $x_0$. Note that $\pa_x^3 W_{1,\b}(x) \to \pa_x^3 W_{1,2}(x)$ for any $x \geq x_0 > 0$ as $\b \to 2$. Using dominated convergence theorem, we yield 
\beq\label{eq:conv_W15}
\lim_{ \b \to 2^-} \int_{x_0}^{\inf} |\pa_{x}^3 W_{1,\b}(x)| dx 
= \int_{x_0}^{\inf} |\pa_{x}^3 W_{1,2}(x)| dx .
\eeq

Note that the conditions \eqref{eq:ver10}-\eqref{eq:ver3} hold with strictly inequality for $W= W_{1,2}, G = G_{1,2}$. From the uniform convergences $G_{1,\b}(\xi) \to G_{1,2}(\xi)$ on $[0, M]$, $\pa_x^2 W_{1,\b}(x) \to \pa_x^2 W_{1,2}(x)$ on $[\d, x_0]$, 
$ \pa_x W_{1, \b}(x_0 ) \to  \pa_x W_{1,2}(x_0)$ as $\b \to 2$, \eqref{eq:conv_W14}
and \eqref{eq:conv_W15}, we conclude that there exists $\b_0 \in (1,2)$, such that \eqref{eq:ver10}-\eqref{eq:ver3} hold for $ W=W_{1,\b}, G = G_{1,\b}$ with $\b \in [\b_0,2]$.
\end{proof}

\subsubsection{ Convexity of $W_{i,\b}$}\label{app:convex} %Condition \eqref{eq:ver2} and Lemma \ref{lem:pos_W2} relate to the convexity of the kernel. 

We first establish \eqref{eq:ver2} for $x_0 = \log \f{5}{3}$ and then prove Lemma \ref{lem:pos_W2}. 

Since $W_{i,\b}$ is given explicitly in \eqref{eq:tK1}, \eqref{eq:kernelW} and \eqref{eq:kernel_W2},, to simplify the derivations, we have used \textit{Mathematica}. All the symbolic derivations and simplification steps are given in Mathematica (version 12) \cite{DG2021Matlabcode}. We only provide the steps that require estimates. 

%Since $W_{i,\b}$ is given explicitly in \eqref{eq:kernelW} and \eqref{eq:kernel_W2}, to simplify the derivations, we have used Mathematica for symbolic calculation. All the derivation and simplification steps are given in Mathematica. We only provide the steps that require estimates.

Suppose that $W(x) = K(e^x)$ and denote $s= e^x$. Using the chain rule, we yield 
\beq\label{eq:Wxx}
\bal
\pa_{xx} W_{i,\b}(x) &=\pa_{xx} \td K_{i,\b}(e^x)
= e^{2x}  ( \pa^2 \td K_{i,\b} )(e^x)
+ e^x ( \pa \td K_{i,\b})(e^x)  \\
&= s^2 \pa^2 \td K_{i,\b}(s) + s \pa \td K_{i,\b}(s) \teq I_i(s, \b).
\eal
\eeq
To establish \eqref{eq:ver2}: $\pa_{xx} W_{1,2}(x) > 0$ for $x \in [0, x_0],  x_0 = \log \f{5}{3}$, it suffices to prove $I_1(s, 2) >0$ for $s \in [1, 5/3]$. For $i = 1, \b = 2$, 
using symbolic calculation, we yield 
\[
I_1(s, 2) = \f{P_1 +  P_2 }{4 s^{3/2} (1+s)^3} , \quad P_2 = 9(1+s)^4(1-s+s^2) \log \B| \f{s+1}{s-1} \B| .
\]
We do not write down the expression of $P_1$ since it is an intermediate term and is not used directly. We provide its formula in Mathematica \cite{DG2021Matlabcode}. Using $\log (1 +z) \leq z$ for $z> -1$, we yield 
\beq\label{eq:log_ineq}
\log \B| \f{ 1 + s}{1-s} \B| = -\log \B| \f{ 1-s}{1 + s} \B| \geq - ( - \f{2}{1+s} ) = \f{2}{1+s}.
\eeq

Using the above inequality and simplifying the expression, we yield 
\[
\bal
I_1(s, 2)& \geq \f{1}{4 s^{3/2} (1+s)^3} ( P_1 +  
9(1+s)^4(1-s+s^2) \f{2}{s+1} ) = \f{P_3}{4 s^{3/2} (1+s)^3}, \\
  P_3 & = - \f{2( -9 + 9 s + 27 s^2 -18 s^3 - 59 s^4 + 9 s^5 + 9 s^6)}{ (s-1)^2}.
\eal
\]

%Thus, to obtain $I_1(s, 2) \geq 0$ for $s \in [1, 5/3]$, it suffices to verify $P_3 \geq 0$. z

Since $s \in [1, \f{5}{3}]$, using $s^i \leq s^j, i \leq j$ and $ 9s + 9s^2 \leq 15 + 25 < 41$, we obtain 
\[
-9 + 9 s + 27 s^2 -18 s^3 - 59 s^4 + 9 s^5 + 9 s^6
<  (9s + 27 s^2 - 18 s^3 - 18 s^4)
+ s^4( 9 s  + 9 s^2 - 41) < 0,
\]
which implies $P_3 > 0$ on $[1, \f{5}{3}]$. It follows $I_1(s, 2) >0$ on $[1, 5/3]$ and \eqref{eq:ver2} with $x_0 = \log \f{5}{3}$.

Next, we prove Lemma \ref{lem:pos_W2}.
\begin{proof}
Recall $W_{2,\b}(x) = \td K_{2,\b}(e^x)$ and their formulas from \eqref{eq:kernel_W2}. Denote $s = e^x$. Using \eqref{eq:Wxx}, it suffices to prove that $I_2(s,\b) \geq 0$ for all $s = e^x \geq 1$. Using symbolic calculation, we have
% Since $x \geq 0$, we get $s \geq 1$. Using the chain rule, we yield 
% \[
% \pa_{xx} W_{2,\b}(x) =\pa_{xx} \td K_{2,\b}(e^x)
% = e^{2x}  ( \pa^2 \td K_{2,\b} )(e^x)
% + e^x ( \pa \td K_{2,\b})(e^x) 
% = s^2 \pa^2 \td K_{2,\b}(s) + s \pa \td K_{2,\b}(s) \teq I(s, \b).
% \]
% Thus, it suffices to prove that $I(s,\b) \geq 0$ for all $s = e^x \geq 1$. To simplify the derivation, we have used Mathematica for symbolic calculation. In particular, we have
\[
I_2(s, \b) = \f{\b}{2} s^{-a} ( I_{2,1}(s,\b) + a^2  (1 + s^{2a}) \log \f{1 + s}{s-1}   ), \quad a = \f{\b-1}{2} .
\]
where $I_{2,1}(s,\b)$ is an intermediate term and its formula is given in Mathematica \cite{DG2021Matlabcode}.
%We do not write down the expression of $I_{2,1}(s,\b)$ since it is an intermediate term and is not used directly. We provide its formula in Mathematica.  
Since $\b >0$, using \eqref{eq:log_ineq}, we yield
%Since $\b > 0, s \geq 1$, using $\log (1 + z)\leq z, z > -1$, we get $\log \f{1+s}{s-1}  = - \log \f{s-1}{s+1} \geq - (- \f{2}{s+1}) = \f{2}{s+1}$. Thus, we yield 
\[
I_2(s,\b)
\geq \f{\b}{2} s^{-a} ( I_{2,1}(s,\b) + a^2  (1 + s^{2a}) \f{2}{1+s}  )
\teq  \f{\b}{2} s^{-a} I_{2,2}(s,\b).
\]

Next, we show that $I_{2,2}(s,\b) \geq 0$.  Simplifying the expression, we obtain
\[
\bal
I_{2,2}(s,\b) &= \f{P_1 + P_2 + P_3 }{(s^2 - 1)^3}, 
\quad P_1 = -2a^2 (s^2 - 1)^2 (1 -2s + s^{2a}), \\
 P_2 & = 8as (s^2 - 1)(s^2 + s^{2a}) , \quad P_3  = 4s (3 s^2 + s^4 - s^{2a} - 3 s^{2+2a}).
\eal
\] 
Since $a = \f{\b-1}{2} \in [0, \f{1}{2}]$ and $ s \geq 1$, we get $2 s -1 - s^{2a} \geq 2s - 1 -s = s-1 \geq 0 $. Thus, we obtain $P_1, P_2 \geq 0$. Using $s^{2a} \leq s$ again, we derive 
\[
P_3 \geq 4 s (3s^2 + s^4 - s -3 s^3)
 = 4 s^2(  s^3 - 1  + 3 s- 3s^2)
 = 4s^2(s-1)( s^2 + s+1 - 3 s)
 = 4s^2(s-1)^3 \geq 0 . 
\]
Combining the above estimates of $P_i$, we establish $I_2(s,\b) \geq 0$ for $s\geq 1, \b > 1$, 
which further implies $\pa_{xx} W_{2,\b} \geq 0$ for $x \geq 0$.
% Thus, we obtain $P_1 \geq 0$. Recall that we assume $\b \in [\f{7}{4}, 2]$. It follows $a \in [\f{3}{8}, \f{1}{2}]$. Since $s \geq 1$, we derive 
% \[
% \bal
% P_2 & \geq 3 s(s^2 - 1)(s^2 + s^{2a})
% + 4s (3 s^2 + s^4 - s^{2a} - 3 s^{2+2a})
% = s ( 9s^2 + 7 s^4 - 7 s^{2a} - 9 s^{2 +2a}) \\
% & \geq s ( 9s^2 + 7 s^4 - 7 s - 9 s^{3}) 
% = s^2 ( 7 s^3 - 7 + 9 s - 9 s^2   )
% = s^2 (s-1)( 7s^2 + 7s + 7 - 9s)  \geq 0,
% \eal
% \]
% where we have used $ s^{2a} \leq s$ in the second inequality, and $s \geq 1$ in the last inequality. Therefore,  for $\b \in [\f{7}{4},2]$, we prove $P_1, P_2 \geq 0$, which implies $I(s,\b) \geq 0$.
\end{proof}

\subsection{Proof of other Lemmas}\label{app:lemma}

\begin{proof}[Proof of Lemma \ref{lem:trigon}]
Recall that $x, y\in [0, \pi/2]$ and $\b \in [3/2, 2]$. In the following estimates, the reader can think of the special case $\b = 2$.

For $x + y \leq \f{\pi}{2}$, since $y \leq \f{\pi}{2} - x$ and $\cot z$ is decreasing on $[0, \pi]$, we have 
\beq\label{eq:tri_pf_1}
\cot x \cot y \geq \cot x \cot (\pi/2- x) = 1.
\eeq
Since $\min(x, y)\leq \f{1}{2}(x+y) \leq \f{\pi}{4}$,  we obtain $\max(\cot x, \cot y) \geq 1$ and 
\[
(\cot x \cot y )^{\b} \geq \cot x \cot y \geq \min( \cot x , \cot y) \geq \cot (x+y).
\]
The case $x + y \geq \f{\pi}{2}$ is trivial, and we prove \eqref{eq:trigon1} in Lemma \ref{lem:trigon}. Next, we consider \eqref{eq:trigon2}
\[
I \teq \cot y (\cot x )^{\b-2}  \wedge \cot x (\cot y)^{\b-2} 
\les (\cot x \cot y)^{\b} + \one_{x+y \geq \pi/2} \cot(\pi - x- y)  \teq J.
\]
Note that $\one_{x+y \geq \pi/2} \cot(\pi - x- y)$ is nonnegative. Without loss of generality, we assume $x \leq y$. Since $\b \leq 2$ and $\cot x \geq \cot y$, we get 
\[
 I = \cot y (\cot x)^{\b-2} .
\]
%since $\b \leq 2$. 

\paragraph{\bf{Case 1: $x+y \leq \pi/2$}}
Since $x \leq y$ and $x \leq \f{1}{2} (x+y) \leq \f{\pi}{4}$, using \eqref{eq:tri_pf_1},  $\cot x \geq 1$, $\cot x \geq \cot y$ and $\b \in [1,2]$, we yield 
\[
J \geq (\cot x \cot y)^{\b} 
\geq (\cot x \cot y)^{\b-1} 
\geq (\cot y)^{\b-1} \geq \cot y (\cot x)^{\b-2} = I.
\]

\paragraph{\bf{Case 2: $x+y > \f{\pi}{2}$}}
In this case, $J$ contains the term $\cot (\pi - x- y) \geq 0$.

\textbf{Case 2.a:} $x > \f{\pi}{3}$. Since $y \geq x \geq \f{\pi}{3}$, we have $\cot y \leq \cot x, \cot x \les 1$ and $\cot( \pi - x- y) \geq \cot \f{\pi}{3} \gtr 1$. It follows 
\[
I \leq \cot x (\cot x)^{\b-2} = (\cot x)^{\b-1}\les 1 \les  \cot( \pi - x- y) \les   J .
\]

\textbf{Case 2.b:} $x \leq \f{\pi}{3}$ and $ \pi - x - y \leq y$. Since $1 \les \cot x$ and $\cot z$ is decreasing on $[0, \pi]$, we yield 
\[
I \les \cot y \leq \cot(\pi - x- y).
\]

\textbf{Case 2.c:} $x \leq \f{\pi}{3}$ and $ \pi - x - y \geq y$. Since $ y \geq \f{1}{2}(x+y) \geq \f{\pi}{4}, x \leq \f{\pi}{3}$, we have 
\[
\cot x \gtr  x^{-1}, \quad \cot y  \gtr \cos y \gtr \pi / 2- y.
%\quad \cot x \cot y \gtr \f{\pi/2- y}{x} \gtr 1,
\]
%where in the last inequality, we have used $\pi/2- y \geq x$ from $\pi - x - y \geq y$.
Note that  $\pi - x - y \geq y$ implies $\pi/2 - y \geq x / 2$. We yield 
\[
\cot x \cot y \gtr  \f{\pi/2 - y}{x} \gtr 1,
\]
which along with $1 \les \cot x, \cot y \leq \cot x, \b \in [1,2]$ imply 
\[
I \leq \cot y (\cot y)^{\b-2}
= (\cot y)^{\b-1} \les (\cot x \cot y)^{\b-1}
\les (\cot x \cot y)^{\b} \les J.
\]
We conclude the proof of \eqref{eq:trigon2}. 

%\paragraph{ \bf{ Proof of \eqref{eq:trigon3}} 
Next, we prove \eqref{eq:trigon3}
\[
II \teq  \cot y \one_{ y\geq \pi/3}
\les (\cot x \cot y)^{\b} + \one_{x+y \geq \pi/2} \cot(\pi - x- y)  = J.
\]

We focus on $ y \geq \pi / 3$.
We consider three cases: (a) $x + y \leq \pi /2$, (b) $x+y > \pi/2$ and $\pi-x- y\leq y$, (c) $x+y > \f{\pi}{2}, \pi - x - y \geq y$. In the first case, from \eqref{eq:tri_pf_1}, we have $J \geq 1\gtr II$. In the second case, since $\cot z$ is decreasing, we get 
\[
J \geq \cot(\pi-x-y) \geq \cot(y) \geq II.
\]

In the third case, since $ x \leq \pi - 2 y \leq \pi - 2\pi / 3\leq \pi / 3$, $ y \geq \f{\pi}{3}$
and $\pi /2 - y \geq x/2$, using the same argument as that in the above Case 2.c, we yield 
\[
\cot x \cot y \gtr 1,  \quad J \geq  (\cot x  \cot y)^{\b} \gtr 1 \gtr II . 
\]
So far, we conclude the proof of \eqref{eq:trigon3} and Lemma \ref{lem:trigon}.
\end{proof}

%We have the following 

%the following Lemma shows that 
The initial data constructed in Section \ref{sec:non_stab} enjoys the following regularity in Sobolev space.

\begin{lem}\label{lem:sobolev}
Suppose that $\om_0$ satisfies $\om_0 + \om_{\al} \in C^{\inf}(S^1 \bsh \{ 0\})$ and $\om_0 \in C^{2}( -\pi/3, \pi/3)$, then $\om_0 + \om_{\al} \in H^s$ for any $s < \al + \f{1}{2}$.
\end{lem}

\begin{proof}
Let $\chi$ be a smooth even cutoff function on $S^1$ ($2\pi$ periodic) with $\chi(x)= 1$ for $|x| \leq \f{\pi}{8}$ and $\chi(x) = 0$ for $|x| \geq \f{\pi}{4}$. We decompose $\om_0 + \om_{\al}$ as follows 
\[
\om_0 + \om_{\al} = \chi \om_{\al} + \chi \om_0 + (1-\chi)(\om_0 + \om_{\al}) = I + II + III.
\]
Clearly, $II, III \in C^2 \subset H^{s_1}$ for any $s_1 \leq 2$. Denote $f_{\al} = \chi \om_{\al}$. Since $f_{\al}$ is odd, it enjoys an expansion $\om_{\al}(x) = \sum_{k \geq 1} a_k \sin (kx)$. 
Next, we estimate $a_k$. Using integration by parts, we yield 
\[
a_k = C \int_0^{\pi} f_{\al}  \sin (kx) dx
=  \f{C}{k} \int_0^{\pi} f^{\prime}_{\al} \cos k x dx
=  \f{C}{k} \int_0^{\pi} (\one_{x\leq 1/k} + \one_{ 1/k \leq x \leq \pi/4} ) f^{\prime}_{\al} \cos k x dx \teq J_1 + J_2,
%= \f{C}{k} (\int_0^{1/k}  +  \int_{1/k}^{\pi/2} )f^{\prime}_{\al} \cos k x  dx
%f^{\prime}_{\al} \cos k x  dx + \int_{1/k}^{\pi/2} f^{\prime}_{\al} \cos k x  dx )  
\]
where the restriction $\one_{x \leq \pi/4}$ is due to the fact that $\chi$ is supported in $|x| \leq \pi/4$. Recall the formula of $\om_{\al}$ from \eqref{eq:wal}. A direct calculation yields 
\[
|J_1| \les_{\al} k^{-1} \int_0^{1/k} |f_{\al}^{ \prime }| dx
\les \int_0^{1/k} |x|^{\al-1} dx \les_{\al} k^{-1-\al}.
\]
For $J_2$, using $\cos k x = \pa_x \f{\sin kx}{k}, |\pa_x^i \om_{\al} (x)|\les |x|^{\al-i}$ and integration by parts again, we derive 
\[
\bal
|J_2 | &\les_{\al} k^{-1} \B( \B|  \f{\sin (k \cdot k^{-1}) }{k} f_{\al}^{\prime}( \f{1}{k} )  
 \B|  + \f{1}{k} \int_{1/k}^{\pi/4}  \B| f^{\prime \prime}_{\al} \sin k x \B| dx  
 \B)
 \les_{\al} k^{-1} ( \f{1}{k} (\f{1}{k})^{\al-1} 
+ \f{1}{k} \int_{1/k}^{\pi/4} |x|^{\al - 2} dx
 ) \\
& \les_{\al} k^{-1} ( k^{-\al} + k^{-1}  (k^{-1})^{\al-1} ) \les_{\al} k^{-\al-1}.
 \eal
\]

Therefore, for $s < \al + \f{1}{2}$, we establish 
\[
\sum_{k\geq 1} |a_k|^2 k^{2s} \leq \sum_{ k \geq 1} k^{-2 - 2\al + 2s} < +\infty,
\]
which implies $ \om_{\al} \chi = f_{\al} \in H^s$. We conclude the proof.
\end{proof}

%\subsection{\bf{Rigorous error control}}\label{app:ver}

\subsection{\bf{Rigorous verification }}\label{app:ver}

To establish Lemma \ref{lem:comp}, we need to verify conditions \eqref{eq:ver10}, \eqref{eq:ver3} in Lemma \ref{lem:pos_W}. Note that condition \eqref{eq:ver2} has been verified in Appendix \ref{app:convex}.

Since the kernel $W_{1,2}$ is explicit  \eqref{eq:tK1},\eqref{eq:kernelW}, to simplify the derivations, we have used \textit{Mathematica}. All the symbolic derivations and simplification steps are given in Mathematica (version 12). We only provide the steps that require estimates. All the numerical computations and quantitative verifications are performed in MATLAB (version 2019a) in double-precision floating-point operations. The Mathematica and MATLAB codes can be found via the link \cite{DG2021Matlabcode}. We will also use interval arithmetic \cite{rump2010verification,moore2009introduction} and refer the discussions to Appendix \ref{app:interval}.

To obtain \eqref{eq:ver10}, using the approach in Section \ref{sec:ver}, we only need to verify \eqref{eq:ver1}. Conditions \eqref{eq:ver1} and \eqref{eq:ver3} involve a finite number of integrals and the Lipschitz constant $b_1$ in \eqref{eq:lip}. Since these conditions are not tight, we use the following simple method to verify them. %Similar ideas ave been used in \cite{chen2019finite,chen2021HL} for rigorous estimates of integrals. 

%\subsubsection{Rigorous esim}

%\subsubsection{ Rigorous estimates of integrals.}

To estimate the integral of $f$ on $[A, \infty)$ with $A \geq0$, we first choose $B$ sufficiently large and partition $[A, B]$ into $A = y_0  < y_1 < ... < y_N = B$. We will estimate the decay of $f$ in the far field in Appendix \ref{app:decay}, and treat the integral in $[B,\infty)$ as a small error. For each small interval $I = [y_i, y_{i+1}]$,
%$[y_i, y_{i+1}]$ with $i < N$, 
we use a trivial first order method to estimate the integral 
\beq\label{eq:ing_low}
 |I| \min_{x\in I} f(x) \leq \int_I f(x) dx \leq  |I| \max_{ x\in I} f(x), \quad |I| = y_{i+1} - y_i.
 %(b-a) \min_{x\in I} f(x) \leq \int_a^b f(x) dx \leq  (b-a) \max_{ x\in [a,b]} f(x).
\eeq

Denote by $ f^u(I), f^l(I)$ the upper and lower bounds for $f$ in $I$. To use \eqref{eq:ing_low}, we estimate $f^l(I), f^u(I)$ for each interval $I=[y_i, y_{i+1}]$. 
For simplicity, we drop the dependence on $I$.
%To simplify the notation, we drop the dependence on $I$.

We simplify $W_{1,2}$ defined in \eqref{eq:tK1},\eqref{eq:kernelW} as $W$. All the integrands involved in \eqref{eq:ver1}, \eqref{eq:ver3}, \eqref{eq:lip} are $ W(x) \cos(x \xi)$ for $\xi = ih, i =0,1,.., \f{M}{h}, |W(x) x|, |W^{\prime \prime \prime}(x)|$. To obtain the piecewise upper and lower bounds for these integrands, using basic interval arithmetics, see, e.g. \cite{gomez2019computer}
\beq\label{eq:int_basic}
\bal
&(f g )^u = \max( f^u g^u, f^l g^u, f^u g^l, f^l g^l), \quad  (f g )^l = \min( f^u g^u, f^l g^u, f^u g^l, f^l g^l),   \\
&|f|^u = \max( |f^l|, |f^u|)    , \quad
(f-g)^l = f^l - g^u, \quad (f-g)^u = f^u - g^l,
\eal
\eeq
we only need to obtain the bounds for $\cos(x\xi), W, |W x|, W^{\prime \prime \prime}$. Those for $x$ are trivial.

\subsubsection{ Upper and lower bounds for $W, W  x, W^{\prime \prime \prime}$}
We simplify $\td K_{1,2}$ in \eqref{eq:tK1} as $\td K$. Denote $s= e^x$. Using the chain rule and $W(x) = \td K(e^x) = \td K(s)$, we get 
\[
\bal
 \pa_x^3 W(x) &= \pa_x^3 \td K(e^x) 
= e^{3x} (\pa^3 \td K)(e^x) + 3 e^{2x} (\pa^2 \td K)(e^x) + e^x (\pa \td K)(e^x) \\
&=s^3 \pa^3 \td K(s) + 3 s^2 \pa^2 \td  K(s) + s \pa \td K(s) \teq D^3 \td K(s).
\eal
\]
Since $e^x$ is increasing, the bounds for $W$ on $[x_l, x_u]$ and those for $\td K$ on $[e^{x_l}, e^{x_u}]$ enjoys 
%$[s_l, s_u]$with $s_l = e^{x_l}, s_u  =e^{x_u}]$ 
\beq\label{eq:W_lu} 
\bal
f^l &= g^l( e^{x_l}, e^{x_u}), \quad  f^u = g^u( e^{x_l}, e^{x_u}),  \\
 (f, g) &= (W, \td K), \quad  ( \pa_x^3 W, D^3 \td K) , \quad ( W(x) x, \td K(s) \log s)
% f^l = g^l( e^{x_l}, e^{x_u}), \quad  f^u = g^u( e^{x_l}, e^{x_u}), 
% \quad (f, g) = (W, \td K), \mathrm{ \ or \ } (f, g) = ( \pa_x^3 W, D^3 \td K).
% W_l = \td K^l( e^{x_l}, e^{x_u}), \  W_u = \td K^u( e^{x_l}, e^{x_u}),
% \quad
% \pa_x^3 W_l = D^3 \td K^l( e^{x_l}, e^{x_u}), \  \pa_x^3 W_u = D^3 \td K^u( e^{x_l}, e^{x_u}),
\eal
\eeq

%to obtain the upper and lower bound of $W, \pa_x^3 W$, 
Thus it suffices to get bounds for $\td K, \td K \log (s), D^3 \td K$. Recall $\td K$ from \eqref{eq:tK1} with $\b=2$. 
\beq\label{eq:tK_b2}
 \td K(s) = ( s^{ \f{3}{2}}  + s^{- \f{3}{2}} ) \log \B| \f{s+1}{s-1} \B| - \f{  s^{ \f{ 3}{2}}  - s^{- \f{ 3}{2}} }{s^2 -1} 2s 
 = ( s^{ \f{3}{2}}  + s^{- \f{3}{2}} ) \log \B| \f{s+1}{s-1} \B| - 2 s^{-\f{1}{2} } \f{s^2 + s + 1}{s+1}.
\eeq
%For $ 1 \leq s_l < s_u$, 

In the interval $s \in [s_l, s_u]$ with  $ 1 \leq s_l < s_u$, using monotonicity, e.g. $s^{3/2} \in  [s_l^{3/2},s_u^{3/2}]$,
%, s^{-1/2} \in [s_u^{-1/2}, s_l^{-1/2}] $, 
the fact that $\log \B| \f{s+1}{s-1} \B| $ is decreasing  and \eqref{eq:int_basic},  we get the upper and lower bounds for $\td K$% on $[s_l, s_u]$ 
\beq\label{eq:K_lu}
\bal
\td K^l(s_l, s_u) &= (s_l^{3/2} + s_u^{-3/2} ) \log \B| \f{ s_u+1}{s_u-1} \B|
- 2 s_l^{-1/2}  \f{s_u^2 + s_u +1}{s_l + 1} , \\
\td K^u(s_l, s_u) &= (s_u^{3/2} + s_l^{-3/2} ) \log \B| \f{ s_l+1}{s_l-1} \B|
- 2 s_u^{-1/2}  \f{s_l^2 + s_l +1}{s_u + 1} .
\eal
\eeq

%For $W(x)x$ in \eqref{eq:lip}, since $W(x) x = \td K(s) \log(s), s = e^x$, using \eqref{eq:W_lu}, we only need to get the bounds for $\td K(s) \log(s)$. 
Next, we consider $\td K \log s$. For $s \in [s_l, s_u]$ with $s_l \geq 1$, since $\log s \geq 0$,  we get 
\[
   \td K(s) \log (s) \leq \td K^u \log (s) \leq \max( \td K^u \log s_l, \td K^u \log s_u).
\]
Similarly, we obtain the lower bound for $ \td K \log s$.
%The bounds for $W x$ can be obtained using \eqref{eq:int_basic} and \eqref{eq:K_lu}.
 Yet, near $s=1$, the upper bound blows up due to $\log |s_l-1|$ in $\td K^u$. Note that $\log s \leq s-1$. Using \eqref{eq:log_ineq}, for $s\geq 1$, we get
\[
\pa_s ( (s-1)\log \B| \f{s+1}{s-1} \B|   ) = (\f{1}{s+1} - \f{1}{s-1}) (s-1) + \log \f{s+1}{s-1}
= -\f{2}{s+1} +  \log \f{s+1}{s-1} \geq 0.
\]
Thus, $  \log \B| \f{s+1}{s-1} \B|  (s-1) $ is increasing on $[s_l, s_u]$ and 
\[
\log \B| \f{s+1}{s-1} \B|  \log s \leq \log \B| \f{s+1}{s-1} \B|  \cdot (s-1) \leq 
\log \B| \f{s_u+1}{s_u-1} \B|  \cdot (s_u-1).
\]

We obtain the following improvement for the upper bound of $\td K(s) \log s$ on $[s_l, s_u]$
\beq\label{eq:Klog}
\td K(s ) \log( s)
\leq ( s_u^{3/2} + s_l^{-3/2}) \log \B|\f{s_u+1}{s_u-1} \B| \cdot (s_u-1) 
- 2 s_u^{-1/2}  \f{s_l^2 + s_l +1}{s_u + 1}  \cdot \log(s_l)
\eeq

For $D^3 \td K(s)$, firstly, using symbolic computation, we yield 
\beq\label{eq:D3tK_1}
\bal
D^3 \td K(s) &= \f{P_{42}(s) - P_{41}(s) + P_5(s)}{P_6(s)},  \quad  P_{42}(s) = 180 s^3 + 180 s^7 ,  \\
P_{41}(s) & = 54  s + 54 s^2  + 266  s^4 + 124  s^5 + 266  s^6  + 54 s^8 + 54 s^9 ,\\
  P_5(s) &= 27 (s^2 - 1)^4(1+s + s^2) \log \B| \f{s+1}{s-1} \B|, 
\quad P_6(s) = 8 (s-1)^3 s^{3/2}(1+s)^4 .
\eal
\eeq
Since $  1 \leq s_l < s_u$ and $P_{41}, P_{42}, P_6$ are increasing, we get $P_m^u = P_m(s_u), P_m^l = P_m(s_l)$ for index $m=41,42 $ or $m=6$. The bounds for $P_5$ are also trivial 
\[
P_5^l = 27 (s_l^2 - 1)^4(1+s_l + s_l^2) \log \B| \f{s_u+1}{s_u-1} \B|,
\  P_5^u = 27 (s_u^2 - 1)^4(1+s_u + s_u^2) \log \B| \f{s_l +1}{s_l-1} \B|.
\] 
Using the bounds for $P_{41}, P_{42}, P_5, P_6$ and \eqref{eq:int_basic}, we can further derive the bounds for $D^3 \td K$. 

% With the above bounds for $\td K$, for $s_l = e^{x_l}, s_u = e^{x_u}$, on $[x_l, x_u]$, we yield the bounds for $W$
% \[
% W_l = \td K^l( e^{x_l}, e^{x_u}), \quad 
% W_u = \td K^u( e^{x_l}, e^{x_u})
% \]

\subsubsection{Upper and lower bounds for $\cos( x \xi)$}

For $f \in C^2([a,b])$ and $x \in [a, b]$, the basic linear interpolation implies $f(x) = \f{x-a}{b-a} f(b) + \f{b-x}{b-a} f(a)+ \f{1}{2}f^{\prime \prime}( x_1) (x-a)(x-b)$ for some $x_1 \in [a,b]$ and 
\[
\min( f(a), f(b)) - \f{(b-a)^2}{8} || f^{\prime \prime }||_{L^{\inf}[a,b]} \leq f(x) \leq 
\max( f(a), f(b))+ \f{ (b-a)^2}{8} || f^{\prime \prime }||_{L^{\inf}[a,b]} .
\]
Applying the above estimate to $ f(x) = \cos(x \xi)$ and $| f^{\prime \prime}(x) | \leq \xi^2$, we derive the upper and lower bounds for $ \cos (x \xi )$ on $[a, b]$.

To verify \eqref{eq:ver1}, it suffices to get a lower bound for $G(\xi)$ with $\xi = j h$. Applying \eqref{eq:W_lu},  \eqref{eq:K_lu}, the above estimate for $\cos (x \xi)$ and \eqref{eq:ing_low}, we yield 
\[
\int_{y_i}^{y_{i+1}} \cos (x \xi) W(x) dx  \geq (y_{i+1} - y_i) \cdot I^l, \quad I(x) \teq \cos ( x \xi) W( x ) .
\]
%The lower bound 
The term $I^l$ can be obtained using \eqref{eq:int_basic}. For $y_i$ close to $0$, we should avoid using \eqref{eq:int_basic} to derive $I^l$ since it involves $ W^u( x_l, x_u) = \td K^u( e^{x_l}, e^{x_u})$ \eqref{eq:K_lu}, which blows up near $x =0$. For $x \xi \leq \pi/2$, since $\cos(x \xi) \geq 0$, we derive $I^l$ using
\[
\cos (x\xi) W(x) \geq \cos(x \xi) W^l \geq \min( (\cos (\cdot \xi))^l W^l, (\cos(\cdot \xi ) )^u W^l ) . 
\]

For large $ \xi$, the above estimate is not sharp due to large oscillation in $\cos (x \xi)$. Denote $m = \f{W_l + W_u}{2}, h_0 = b-a$. We consider an improved estimate 
\[
\bal
&\int_{a}^{b} \cos (x \xi) W(x) dx
= \int_{a}^{b} \cos (x \xi) (W(x) - m) dx 
+ m \int_a^b \cos (x \xi ) dx   \\
\geq & m \f{\sin (x \xi)}{\xi} \B|_a^b - h_0  | \cos ( x \xi ) |^u |W - m|^u
\geq %\f{W_l + W_u}{2} \f{\sin (x \xi)}{\xi} \B|_a^b
\f{W_l + W_u}{2} \f{\sin (b \xi) - \sin(a\xi)}{\xi} 
- h_0  | \cos ( x \xi ) |^u \f{W_u - W_l}{2},
\eal
\]
where we have used $W - m \in [W_l - m, W_u -m] = [ - \f{W_u -W_l}{2} ,\f{W_u - W_l}{2}]$. 

Using the above estimates, we obtain the lower bound of the integral in $G(\xi)$ \eqref{eq:def_G} in a finite domain. The integrals in \eqref{eq:lip} and \eqref{eq:ver3} in a finite domain are estimated  similarly.

\subsubsection{ Decay estimates of $W, \pa_x^3 W$}\label{app:decay}

It remains to estimate the integrals in \eqref{eq:ver1}, \eqref{eq:def_G}, \eqref{eq:lip} and \eqref{eq:ver3} in the far field. For $s>1$, using Taylor expansion, we yield 
\beq\label{eq:log_ineq2}
\log \B| \f{s+1}{s-1} \B| 
= \sum_{k\geq 1} \f{2}{2k-1} s^{-( 2k-1)}, 
\quad  \B| \log \B| \f{s+1}{s-1} \B| - \f{2}{s}  \B| \leq \f{2}{3} \sum_{k\geq 2} s^{-(2k-1)}
= \f{2}{3} \f{s^{-3}}{1-s^{-2}}.
\eeq

Using the above estimate and \eqref{eq:tK_b2}, we obtain 
\[
\bal
|\td K| &\leq \B| s^{3/2} \cdot \f{2}{s} - 2 s^{- \f{1}{2}}  \f{s^2 + s + 1}{1+s}  \B|
+ s^{3/2} \cdot \f{2}{3}\f{s^{-3}}{1-s^{-2}} + s^{-3/2} \log \B| \f{s+1}{s-1} \B|  
\teq I_1 + I_2 + I_3.
\eal
\]

Note that $I_1 = \f{2s^{-1/2}}{s+1 } \leq 2s^{-3/2}$. We derive 
\beq\label{eq:Ktail_1}
|\td K | \leq s^{-3/2} ( 2 + \f{2}{3} \f{1}{1 - s^{-2}}  + \log \B| \f{s+1}{s-1} \B| ) 
\teq s^{-3/2} \td K_{tail}(s).
\eeq

Next, we estimate  $D^3 \td K$ \eqref{eq:D3tK_1}. Using \eqref{eq:log_ineq2}, we decompose $P_5$ in \eqref{eq:D3tK_1} as follows 
\[
|P_5 - P_{5, M}| \leq  P_{5,err}, \quad 
P_{5,M} =  27 (s^2 - 1)^4(1+s + s^2) \f{2}{s},
\quad  P_{5, err} = 27 (s^2 - 1)^4(1+s + s^2) \f{2}{3} \f{s^{-3}}{1 - s^{-2}}.
\]

Recall $P_{41}, P_{42}, P_6$ from \eqref{eq:D3tK_1}. Denote $P_7 = P_{42} - P_{41} + P_{5, M}$. We estimate \eqref{eq:D3tK_1} as follows 
\beq\label{eq:D3tK_2}
|D^3 \td K| \leq  \f{ | P_{42} - P_{41} + P_{5, M} | + P_{5, err}  }{P_6} 
\leq \f{ |P_7|}{ P_6} + \f{ P_{5,err}}{P_6}.
\eeq

By definition, $P_7$ is a sum of a polynomial of $s$ and $s^{-1}$. Simplifying the expression of $P_7$ (see details in \cite{DG2021Matlabcode}) and using the triangle inequality, we yield 
\[
|P_7| \leq P_8 = 54 + 54s^{-1} + 216 s + 270 s^2 + 288 s^3 + 58 s^4 + 16 s^5 + 482 s^6 + 
  18 s^7 \teq  s^7 P_{8, tail}(s),
\]
where $P_{8,tail} \teq P_8 (s) s^{-7}$ is decreasing in $s$.  %The detail is recorded in Mathematica. 
For $P_6$ \eqref{eq:D3tK_1} and the error term $P_{5, err}$, we have 
\[
\bal
P_6 &= 8 (-1 + s)^3 s^{3/2} (1 + s)^4
\geq  s^{7 + 3/2} \cdot 8 (1- s^{-1})^3 
 \teq s^{7 + 3/2} P_{6, tail} (s),  \\
\f{ P_{5,err}}{P_6 } &= \f{9 (1 + s + s^2)}{ 4 s^{5/2} (1 + s)}
\leq s^{-5/2} \f{9(1+s)}{4} \leq s^{-3/2} \f{9}{4} (1 + s^{-1})
\teq s^{-3/2} E_{tail}(s).
 \eal
\]

Plugging the above estimates in \eqref{eq:D3tK_1}, \eqref{eq:D3tK_2}, we obtain 
\beq\label{eq:Ktail_2}
|D^3 \td K(s)| \leq \f{ |P_7| }{P_6}  +  \f{P_{5, err}}{ P_6}
%\f{ | P_{41 } - P_{42} + P_{5, M} | + P_{5,err} }{P_6}
\leq 
\f{ P_8 }{P_6}  +  \f{P_{5, err}}{ P_6}
\leq s^{- \f{3}{2}} ( \f{P_{8, tail}}{ P_{6,tail}} + E_{tail})
\teq s^{- \f{3}{2} }  \td K_{tail, 2}.
\eeq

Clearly, $\td K_{tail}(s) $ is decreasing. Since $P_{8,tail}, E_{tail}$ are decreasing and $ P_{6,tail}$ is increasing, $\td K_{tail, 2}$ is decreasing. Using $W(x) = \td K(e^x)$, we estimate the integrals in  $G(\xi)$ \eqref{eq:def_G} and \eqref{eq:lip} in the far field 
as follows 
\beq\label{eq:int_tail}
\bal
\B| \int_B^{\inf} W(x) \cos (x\xi) dx \B| 
&\leq \td K_{tail}(e^B) \int_B^{\inf} e^{-3x/2} dx 
= \td K_{tail}(e^B) \f{2}{3} e^{-3B/2} ,  \\
 \int_B^{\inf} |W(x) x | dx 
&\leq \td K_{tail}(e^B) \int_B^{\inf} e^{-3x/2}  x dx 
= \td K_{tail}(e^B) ( \f{2 B}{3} + \f{4}{9}) e^{-3B/2}, 
\eal
\eeq
and treat them as error. Similarly, we estimate the integral in \eqref{eq:ver3} in the far field. 

So far, we conclude the estimates of all the integrals in \eqref{eq:ver1}, \eqref{eq:def_G}, \eqref{eq:lip} and \eqref{eq:ver3}. 

% %The tail of the integral in \eqref{eq:ver3} 
% Similarly, we estimate the tail of the integral in \eqref{eq:ver3}.
% We conclude the estimates of the integrals in \eqref{eq:ver1}, \eqref{eq:def_G}, \eqref{eq:lip} and \eqref{eq:ver3}. 

\subsubsection{ Interval arithmetic}\label{app:interval}

To implement the above estimates and verify \eqref{eq:ver1}, \eqref{eq:ver3} rigorously, we adopt the standard method of interval arithmetic \cite{rump2010verification,moore2009introduction}. In particular, we use the MATLAB toolbox INTLAB (version 11 \cite{Ru99a}) for the interval computations. Every single real number $p$ involved in the above estimates is represented by an interval $[p_l, p_r]$ that contains $p$, where $[p_l ,p_r]$ are some floating-point numbers. We refer to \cite{gomez2019computer,chen2019finite,chen2021HL} for related discussion.

% \subsubsection{ Interval arithmetic}

% To implement the estimates below and verify \eqref{eq:ver1}, \eqref{eq:ver3} rigorously, we adopt the standard method of interval arithmetic (see \cite{rump2010verification,moore2009introduction}). Every single real number $p$ involved in the estimates is represented by an interval $[p_l, p_r]$ that contains $p$, where $[p_l ,p_r]$ are some floating-point numbers. We refer to \cite{gomez2019computer,chen2019finite,chen2021HL} for more discussion. In particular, we use the MATLAB toolbox INTLAB (version 11 \cite{Ru99a}) for the interval computations. 

\bibliographystyle{plain}
\bibliography{selfsimilar}

\end{document}